\documentclass[11pt, twoside, leqno]{article}

\usepackage{amsmath}
\usepackage{amssymb}
\usepackage{titletoc}
\usepackage{mathrsfs}
\usepackage{amsthm}
\usepackage{indentfirst}
\usepackage{color}
\usepackage{txfonts}
\usepackage{enumerate}

\allowdisplaybreaks

\newtheorem{theorem}{Theorem}[section]
\newtheorem{lemma}[theorem]{Lemma}
\newtheorem{corollary}[theorem]{Corollary}
\newtheorem{proposition}[theorem]{Proposition}

\theoremstyle{definition}
\newtheorem{remark}[theorem]{Remark}
\newtheorem{definition}[theorem]{Definition}

\numberwithin{equation}{section}

\begin{document}

\title{\bf\Large Matrix-Weighted Besov-Type and Triebel--Lizorkin-Type Spaces I:
$A_p$-Dimensions of Matrix Weights and $\varphi$-Transform Characterizations
\footnotetext{\hspace{-0.35cm} 2020 {\it Mathematics Subject Classification}.
Primary 46E35; Secondary 47A56, 42B25, 42C40, 42B35.\endgraf
{\it Key words and phrases.}
matrix weight, Besov-type space,
Triebel--Lizorkin-type space, $A_p$-dimension, $\varphi$-transform.\endgraf
This project is supported by
the National Key Research and Development Program of China
(Grant No.\ 2020YFA0712900),
the National Natural Science Foundation of China
(Grant Nos.\ 12371093, 12071197, and 12122102),
the Fundamental Research Funds
for the Central Universities (Grant No. 2233300008),
and the Academy of Finland (Grant Nos.\ 314829 and 346314).}}
\date{}
\author{Fan Bu, Tuomas Hyt\"onen,
Dachun Yang and Wen Yuan\footnote{Corresponding author, E-mail:
\texttt{wenyuan@bnu.edu.cn}/{\color{red} December 24, 2023}/Final version.}}

\maketitle

\vspace{-0.8cm}

\begin{center}
\begin{minipage}{13cm}
{\small {\bf Abstract}\quad Let $s\in{\mathbb R}$, $q\in (0,\infty]$,
and $\tau\in[0,\infty)$. It is well known that
Besov-type spaces $\dot B^{s,\tau}_{p,q}$
with $p\in (0,\infty]$ and Triebel--Lizorkin-type spaces $\dot F^{s,\tau}_{p,q}$ with
$p\in (0,\infty)$ when $\tau\in [0,\infty)$ or
with $p\in (0,\infty]$ when $\tau=0$ on $\mathbb{R}^n$ consist of
a general family of function spaces that cover not only
the well-known Besov and Triebel--Lizorkin spaces $\dot B^{s}_{p,q}$
and $\dot F^{s}_{p,q}$ (when $\tau=0$)
but also several other function spaces of interest,
such as Morrey spaces and $Q$ spaces.
In three successive articles, the authors develop a complete
real-variable theory of matrix-weighted Besov-type spaces
$\dot B^{s,\tau}_{p,q}(W)$ and matrix-weighted Triebel--Lizorkin-type spaces
$\dot F^{s,\tau}_{p,q}(W)$ on $\mathbb{R}^n$,
where $W$ is a matrix-valued Muckenhoupt $A_p$ weight.
This article is the first one, whose main novelty exists in that
the authors introduce the new concept,
$A_p$-dimensions of matrix weights,
and intensively study their properties, especially
those elaborate properties expressed via reducing operators. The authors
then introduce the spaces
$\dot B^{s,\tau}_{p,q}(W)$ and $\dot F^{s,\tau}_{p,q}(W)$
and, using $A_p$-dimensions and their nice properties,
the authors establish the $\varphi$-transform characterization
of $\dot B^{s,\tau}_{p,q}(W)$ and $\dot F^{s,\tau}_{p,q}(W)$.
The $A_p$-dimensions of matrix weights and their properties
also enable the authors
to obtain the sharp boundedness of almost diagonal operators on
related sequence spaces in the subsequent second article and
the optimal characterizations of molecules and wavelets, trace theorems,
and the optimal boundedness of pseudo-differential operators and Calder\'on--Zygmund operators
in the subsequent third article.
}
\end{minipage}
\end{center}

%\vspace{0.2cm}

\tableofcontents

%\vspace{0.2cm}

\section{Introduction}

In three successive articles, we will develop a complete
real-variable theory of matrix-weighted Besov-type spaces
and matrix-weighted Triebel--Lizorkin-type spaces
of $\mathbb C^m$-valued distributions on $\mathbb{R}^n$.
We consistently denote by $m$ the dimension of the target space
of our distributions and hence our matrix weights always take values
in the space of $m\times m$ complex matrices.
This article is the first one.

The study of Besov spaces $B_{p,q}^s$ on the Euclidean space $\mathbb{R}^n$
was started in the 1950s.
In 1951, Nikol'ski\u{\i} \cite{n51}
introduced the Nikol'ski\u{\i}--Besov spaces
which are nowadays denoted by $B^s_{p,\infty}(\mathbb{R}^n)$,
but he mentioned that his work was based on earlier works of
Bern\v{s}te\u{\i}n \cite{b47} and Zygmund \cite{z45}.
By introducing the third index $q$,
Besov \cite{b59, b61} complemented this scale.
Around 1970, Lizorkin \cite{li72,li74} and Triebel \cite{t73}
independently began to investigate the scale $F_{p,q}^s(\mathbb{R}^n)$.
Furthermore, we mention the contributions \cite{p73,p75,p76} of Peetre,
which extend the ranges of admissible parameters $p$ and $q$ to values less than one.
Besov spaces and Triebel--Lizorkin spaces have been widely applied to
various branches of analysis and we refer to monographs
\cite{t83,t92,t06} of Triebel as well as \cite{Sa18,Sa20} of Sawano for more studies on these spaces.
Nowadays, Besov spaces and Triebel--Lizorkin spaces have been
generalized and developed in various different settings
(see, for instance, \cite{b07,bh06,bdy12,bbd22,cgn17,cgn18,cgn19,gjn17,gn16,whhy21}).
We specifically mention extensive recent studies of versions of these spaces associated with different operators
(see, for instance, \cite{bbd20,b20a,b20b,bd15,bd21a,bd21b,bd21c,gkkp17,gkkp19,zy20}).

In recent decades,  there exists an increasing interest in developing generalized Besov
and Triebel--Lizorkin spaces built on Morrey spaces.
Recall that the Besov--Morrey spaces
were introduced by Kozono and Yamazaki \cite{KY94} and Mazzucato \cite{M03} in order to study
Navier--Stokes equations.
Later on  Tang and Xu \cite{TX05} introduced and studied the
Triebel--Lizorkin--Morrey spaces.
From then on, these spaces received a lot of attention and were further  intensively developed  by
Sawano \cite{S08,S09,S10,S10b} and Sawano and Tanaka \cite{ST07,ST09}.
Around 2010, to clarify the relations among Besov spaces, Triebel--Lizorkin spaces,
and $Q$ spaces on $\mathbb{R}^n$, Yang et al. \cite{yy08,yy10,ysy10}
introduced another scale of generalized Besov and Triebel--Lizorkin spaces related to Morrey spaces, which are called Besov-type and Triebel--Lizorkin-type spaces, denoted by $A^{s,\tau}_{p,q}\in\{B^{s,\tau}_{p,q},F^{s,\tau}_{p,q}\}$.
These spaces consist of a general family of function spaces that cover not only
the well-known Besov and Triebel--Lizorkin spaces
but also several other function spaces of interest, such as $Q$ spaces,   Morrey spaces,
and Triebel--Lizorkin--Morrey spaces (see \cite[Section 1.4]{ysy10}).
Various properties and characterizations of
Besov-type and Triebel--Lizorkin-type spaces were later studied in
\cite{HT23,wyyz,yhmsy,yhsy2,yhsy,ysy13,ysy20,zcy19,zyy14}.
Some of these spaces have also been used to study the existence and
the regularity of the solutions of some
partial differential equations such as heat and Navier--Stokes equations;
see, for instance, \cite{le16,lxy14,ly13,t13a,t14,x07,zyz23}.
We also refer to the articles \cite{HS12,HS13,HS14,HS17,HS20} of Haroske and Skrzypczak,
\cite{ghs21,ghs13,HMS16,HMS20,HMS22} of Haroske et al.,
the surveys \cite{si12,si13} of Sickel, as well as the monographs \cite{t13a,t14}
of Triebel for more studies on these spaces and their applications.
A further generalisation $A^{s,\varphi}_{p,q}$, with a function parameter $\varphi$,
is recently due to Haroske et al. \cite{HL23,HLMS23,HMS23}.

The study of the space $L^2(W)$ with a \emph{matrix weight} $W$ on $\mathbb{R}^n$ goes back to
Wiener and Masani \cite[\S 4]{wm58} in their development on the prediction theory
for multivariate stochastic processes.
To solve the problem about the angle between past and future of the multivariate random stationary process
and the problem about the boundedness of the inverse of Toeplitz operators,
Treil and Volberg \cite{tv97} found the right analogue
(in the sense of being necessary and sufficient for operator norm estimates
of interest in these spaces) of the Muckenhoupt $A_2$ weight condition
in this matrix-valued context.
Extensions to $L^p(W)$ with $W\in A_p$ for general $p\in(1,\infty)$
were later found by Nazarov and Treil \cite{nt96} and with a different approach by Volberg \cite{v97}.
The extent to which the classical self-improvement property of $A_p$ weights remains
(or not) valid for matrix weights was investigated by Bownik \cite{b01}
and versions of maximal function estimates appropriate for this setting
were brought to this theory by Christ and Goldberg \cite{cg01,g03}.

After these developments in $L^p(W)$, matrix-weighted Besov spaces $\dot B^s_{p,q}(W)$
were introduced by Roudenko \cite{ro03} for any $p\in(1,\infty)$
and by Frazier and Roudenko \cite{fr04} for any $p\in(0,1]$,
which were further studied by these authors in \cite{fr08,ro04}.
Versions of the classical identification of $L^p$ with a Triebel--Lizorkin space $\dot F^{0}_{p,2}$
in the matrix-valued setting were already obtained in \cite{nt96,v97}
and another approach to these results is due to Isralowitz \cite{isr21},
but a systematic study on the full scale of matrix-weighted Triebel--Lizorkin spaces
$\dot F^s_{p,q}(W)$ is only recently due to Frazier and Roudenko \cite{fr21}.
Soon afterwards, Wang et al. \cite{wyz}
studied the Littlewood--Paley characterization of $\dot F^s_{p,q}(W)$.
Around the same time, Bu et al. \cite{byy} introduced
the homogeneous matrix-weighted Besov spaces on spaces of homogeneous type
and established various real-variable characterizations of these spaces.

Our goal in this article   and two subsequent articles \cite{bhyy2} and \cite{bhyy3}
is to consolidate the existing theories of both
(unweighted) Besov-type and Triebel--Lizorkin-type spaces
$\dot A^{s,\tau}_{p,q}\in\{\dot B^{s,\tau}_{p,q},\dot F^{s,\tau}_{p,q}\}$ over $\mathbb{R}^n$ on the one hand
and matrix-weighted Besov and Triebel--Lizorkin spaces
$\dot A^s_{p,q}(W)\in\{\dot B^{s}_{p,q}(W),\dot F^{s}_{p,q}(W)\}$ over $\mathbb{R}^n$ on the other hand
into a coherent theory of \emph{matrix-weighted Besov-type and Triebel--Lizorkin-type spaces}
$\dot A^{s,\tau}_{p,q}(W)
\in\{\dot B^{s,\tau}_{p,q}(W),\dot F^{s,\tau}_{p,q}(W)\}$.
On this level of generality, we wish to prove results that naturally extend
and reproduce the existing ones in $\dot A^{s,\tau}_{p,q}$ when specialised to a constant weight
and in $\dot A^{s}_{p,q}(W)$ when specialised to $\tau=0$.
However, we actually achieve more: In several cases, our general results
turn out to improve the existing theory
even in the aforementioned special cases already treated in the literature.

For completeness, let us briefly mention some other recent topics in matrix-weighted function spaces
that we will not develop here. There is quite an extensive and growing literature
on the sharp dependence of operator norms on $L^p(W)$ on the weight constant $[W]_{A_p}$;
see \cite{cim18,dly21,hpv,isr20,im19,ipr21,nptv}.
We mostly ignore these quantitative aspects here,
which is at least partially justified by the following point:
whereas in $L^p(W)$ the norm of a function and hence the norm of an operator
are pretty much canonical, a major aspect of the theory of spaces $\dot A^{s,\tau}_{p,q}(W)$
is the equivalence of various different norms, among which there exists no obvious preferred choice,
and the quantitative bounds for the norms may significantly depend on the particular choice of the norm.
Of course, it might still be of interest to quantify these bounds,
possibly in several different versions depending on the chosen norms,
but this aspect is mostly not addressed in the present treatment.
Recently, Bownik and Cruz-Uribe \cite{bc22} proved both the Jones factorization theorem
and the Rubio de Francia extrapolation theorem for matrix $A_p$ weights,
which are two very important and useful tools of analysis.

Another very recent line of investigation is the theory of spaces $L^p(W)$
on \emph{product domains} $\mathbb R^{n_1}\times\mathbb R^{n_2}$
with corresponding matrix-valued \emph{strong $A_p$ weights};
this has been only lately initiated in \cite{dkps}.
A real-variable theory of scalar-weighted Besov and Triebel--Lizorkin spaces on product domains
has been developed, for instance, in \cite{lz13}, and extending this to the matrix-weighted case
(perhaps combining our present techniques with those of \cite{dkps}
to deal with matrix weights on product domains)
is a possible topic for future investigation.

In this article, we first introduce a new concept of
$A_p$-dimensions for matrix weights
and intensively study their properties, especially
those elaborate properties expressed via reducing operators.
For any $s\in\mathbb R$, $\tau\in[0,\infty)$,
$p\in(0,\infty)$, and $q\in(0,\infty]$, we then introduce
the matrix-weighted Besov-type space $\dot B^{s,\tau}_{p,q}(W)$ and
the matrix-weighted Triebel--Lizorkin-type space $\dot F^{s,\tau}_{p,q}(W)$ on $\mathbb{R}^n$,
where $W$ is a matrix-valued Muckenhoupt $A_p$ weight and, using $A_p$-dimensions
and their nice properties, we establish the $\varphi$-transform characterization
of $\dot B^{s,\tau}_{p,q}(W)$ and $\dot F^{s,\tau}_{p,q}(W)$.
As applications, we find that $\dot B^{s,\tau}_{p,q}(W)$ and $\dot F^{s,\tau}_{p,q}(W)$
are well-defined and obtain their lifting property.
The $\varphi$-transform characterization
establishes the relations between function spaces $\dot B^{s,\tau}_{p,q}(W)$ and $\dot F^{s,\tau}_{p,q}(W)$
and corresponding sequence spaces $\dot b^{s,\tau}_{p,q}(W)$ and $\dot f^{s,\tau}_{p,q}(W)$.
These relations are extensively utilized in the subsequent articles \cite{bhyy2,bhyy3}
to give further properties and applications of these matrix-weighted
Besov--Triebel--Lizorkin-type spaces.

It is worth mentioning that, as one main novelty of this article,
the newly introduced $A_p$-dimension
for matrix weights has been proved to play an irreplaceable role in all these
three successive articles. Indeed, the $A_p$-dimension quantitatively describes
the doubling property of matrix weights via reducing operators
and enables us to achieve several sharp or optimal results later.
To be precise, using $A_p$-dimensions, we obtain in the subsequent article \cite{bhyy2}
the sharp boundedness of almost diagonal operators on
related sequence spaces $\dot b^{s,\tau}_{p,q}(W)$ and $\dot f^{s,\tau}_{p,q}(W)$
and also establish in the subsequent third article \cite{bhyy3} the optimal
characterizations of molecules and wavelets, trace theorems,
and the optimal boundedness of pseudo-differential operators and Calder\'on--Zygmund operators on
function spaces $\dot B^{s,\tau}_{p,q}(W)$ and $\dot F^{s,\tau}_{p,q}(W)$.
Based on these, we have reasons to believe that this concept may also
be useful in other studies related to matrix weights.

The organization of the remainder of this article is as follows.

In Section \ref{Ap dimension},
we introduce a new concept of $A_p$-dimensions for matrix weights
and study their elaborate properties. One of the important results is that,
for any given matrix weight $W\in A_p$ on $\mathbb R^n$,
we provide a method to calculate the critical point $d_p(W)\in[0,n)$
for the $A_p$-dimension of $W$, that is, for any $\varepsilon\in(0,\infty)$,
$d_p(W)+\varepsilon$ is an $A_p$-dimension of $W$ but
$d_p(W)-\varepsilon$ is not (see Proposition \ref{critical point} below).
The other is the sharp estimate via reducing operators (see Lemma \ref{22 precise} below)
which plays a key role in characterizing the minimal almost
diagonal conditions in \cite{bhyy2}.
Moreover, several of our results are conveniently stated in terms of the concept of $A_p$-dimensions
and, in some cases, the obtained estimates are shown to be sharp.

In Section \ref{BF type spaces},
we introduce matrix-weighted Besov-type
and Triebel--Lizorkin-type spaces
$\dot A^{s,\tau}_{p,q}(W)\in\{\dot B^{s,\tau}_{p,q}(W),\dot F^{s,\tau}_{p,q}(W)\}$
and corresponding averaging spaces $\dot A^{s,\tau}_{p,q}(\mathbb{A})\in\{\dot B^{s,\tau}_{p,q}(\mathbb{A}),\dot F^{s,\tau}_{p,q}(\mathbb{A})\}$.
Using the properties of matrix $A_p$-weights from Section \ref{Ap dimension},
we prove the equality $\dot A^{s,\tau}_{p,q}(W)=\dot A^{s,\tau}_{p,q}(\mathbb{A})$,
which allows us to choose, in our subsequent considerations,
whichever definition of these spaces that is most convenient for a particular purpose.
Moreover, we introduce matrix-weighted Besov-type
and Triebel--Lizorkin-type sequence spaces
$\dot a^{s,\tau}_{p,q}(W)\in\{\dot b^{s,\tau}_{p,q}(W),\dot f^{s,\tau}_{p,q}(W)\}$
and corresponding averaging spaces $\dot a^{s,\tau}_{p,q}(\mathbb{A})\in\{\dot b^{s,\tau}_{p,q}(\mathbb{A}),\dot f^{s,\tau}_{p,q}(\mathbb{A})\}$,
for which we again obtain the equality
$\dot a^{s,\tau}_{p,q}(W)=\dot a^{s,\tau}_{p,q}(\mathbb{A})$.
Finally, we establish the $\varphi$-transform characterization
of $\dot A^{s,\tau}_{p,q}(W)$ and use it to prove that both
$\dot A^{s,\tau}_{p,q}(W)$ and $\dot A^{s,\tau}_{p,q}(\mathbb{A})$
are independent of the choice of $\varphi$.
As an application, we obtain the lifting property of these spaces.

In Section \ref{F_infty},
we introduce the averaging matrix-weighted Triebel--Lizorkin space  $\dot F_{\infty,q}^s(\mathbb{A})$
and the corresponding sequence space $\dot f_{\infty,q}^s(\mathbb{A})$
and obtain the $\varphi$-transform characterization of $\dot F_{\infty,q}^s(\mathbb{A})$.
Using this characterization and the relation that $\dot f_{\infty,q}^s(\mathbb{A})
=\dot f_{p,q}^{s,\frac{1}{p}}(\mathbb{A})$ which is a simple application of \cite[Corollary 5.7]{fj90},
we obtain $\dot F_{\infty,q}^s(\mathbb{A})=\dot F_{p,q}^{s,\frac{1}{p}}(\mathbb{A})$.

At the end of this introduction, we make some conventions on notation.
The \emph{ball} $B$ of $\mathbb{R}^n$,
centered at $x\in\mathbb{R}^n$ with radius $r\in(0,\infty)$,
is defined by setting
$$
B:=\{y\in\mathbb{R}^n:\ |x-y|<r\}=:B(x,r);
$$
moreover, for any $\lambda\in(0,\infty)$, $\lambda B:=B(x,\lambda r)$.
A \emph{cube} $Q$ of $\mathbb{R}^n$ always has finite edge length
and edges of cubes are always assumed to be parallel to coordinate axes,
but $Q$ is not necessary to be open or closed.
For any cube $Q$ of $\mathbb{R}^n$,
let $c_Q$ be its center and $\ell(Q)$ its edge length.
For any $\lambda\in(0,\infty)$ and any cube $Q$ of $\mathbb{R}^n$,
let $\lambda Q$ be the cube with the same center of $Q$ and the edge length $\lambda\ell(Q)$.
For any $r\in\mathbb{R}$, $r_+$ is defined as $r_+:=\max\{0,r\}$
and $r_-$ is defined as $r_-:=\max\{0,-r\}$.
For any $a,b\in\mathbb{R}$, $a\wedge b:=\min\{a,b\}$ and $a\vee b:=\max\{a,b\}$.
The symbol $C$ denotes a positive constant which is independent
of the main parameters involved, but may vary from line to line.
The symbol $A\lesssim B$ means that $A\leq CB$ for some positive constant $C$,
while $A\sim B$ means $A\lesssim B\lesssim A$.
Let $\mathbb N:=\{1,2,\ldots\}$, $\mathbb Z_+:=\mathbb N\cup\{0\}$, and $\mathbb Z_+^n:=(\mathbb Z_+)^n$.
For any multi-index $\gamma:=(\gamma_1,\ldots,\gamma_n)\in\mathbb Z_+^n$
and any $x:=(x_1,\ldots,x_n)\in\mathbb R^n$,
let $|\gamma|:=\gamma_1+\ldots+\gamma_n$,
$x^\gamma:=x_1^{\gamma_1}\cdots x_n^{\gamma_n}$,
and $\partial^\gamma:=(\frac{\partial}{\partial x_1})^{\gamma_1}
\cdots(\frac{\partial}{\partial x_n})^{\gamma_n}$.
We use $\mathbf{0}$ to denote the \emph{origin} of $\mathbb{R}^n$.
For any set $E\subset\mathbb{R}^n$,
we use $\mathbf 1_E$ to denote its \emph{characteristic function}.
The \emph{Lebesgue space} $L^p(\mathbb{R}^n)$
is defined to be the set of all measurable functions
$f$ on $\mathbb{R}^n$ such that $\|f\|_{L^p(\mathbb{R}^n)}<\infty$, where
$$
\|f\|_{L^p(\mathbb{R}^n)}:=\left\{\begin{aligned}
&\left[\int_{\mathbb{R}^n}|f(x)|^p\,dx\right]^{\frac{1}{p}}
&&\text{if }p\in(0,\infty),\\
&\mathop{\mathrm{\,ess\,sup\,}}_{x\in\mathbb{R}^n}|f(x)|
&&\text{if }p=\infty.
\end{aligned}\right.
$$
The \emph{locally integrable Lebesgue space}
$L^p_{\mathrm{loc}}(\mathbb{R}^n)$ is defined to be the set of
all measurable functions $f$ on $\mathbb{R}^n$ such that,
for any bounded measurable set $E$,
$$
\|f\|_{L^p(E)}:=\|f\mathbf{1}_E\|_{L^p(\mathbb{R}^n)}<\infty.
$$
In what follows, we denote $L^p(\mathbb{R}^n)$ and $L^p_{\mathrm{loc}}(\mathbb{R}^n)$
simply, respectively, by $L^p$ and $L^p_{\mathrm{loc}}$.
For any measurable function $w$ on $\mathbb{R}^n$
and any measurable set $E\subset\mathbb{R}^n$, let
$$
w(E):=\int_Ew(x)\,dx.
$$
For any measurable function $f$ on $\mathbb{R}^n$
and any measurable set $E\subset\mathbb{R}^n$ with $|E|\in(0,\infty)$, let
$$
\fint_Ef(x)\,dx:=\frac{1}{|E|}\int_Ef(x)\,dx.
$$
The \emph{Hardy--Littlewood maximal operator} $\mathcal{M}$ is defined by setting,
for any $f\in L^1_{\mathrm{loc}}(\mathbb{R}^n)$ and $x\in\mathbb{R}^n$,
\begin{equation}\label{maximal}
\mathcal{M}(f)(x):=\sup_{\mathrm{ball}\,B\ni x}\fint_B|f(y)|\,dy.
\end{equation}
For any space $X$, the product space $X^m$ with $m\in\mathbb{N}$
is defined by setting
\begin{equation*}
X^m:=\left\{\vec f:=(f_1,\ldots,f_m)^{\mathrm{T}}:\
\text{for any}\ i\in\{1,\ldots,m\},\ f_i\in X\right\}.
\end{equation*}
Also, when we prove a theorem
(and the like), in its proof we always use the same
symbols as those appearing in
the statement itself of the theorem (and the like).

\section{Preliminaries on Matrix Weights}
\label{Ap dimension}

In this section, we recall the definition and several known facts about matrix $A_p$-weights,
and we then introduce a new concept of the $A_p$-dimension for matrix weights
and study its properties. We first recall some basic concepts of matrices.

For any $m,n\in\mathbb{N}$,
the set of all $m\times n$ complex-valued matrices is denoted by $M_{m,n}(\mathbb{C})$,
and $M_{m,m}(\mathbb{C})$ is simply denoted by $M_{m}(\mathbb{C})$.
For any $A:=[a_{ij}]\in M_{m,n}(\mathbb{C})$,
the \emph{conjugate} of $A$,
denoted by $\overline{A}$, is the matrix in $M_{m, n}(\mathbb{C})$
whose $(i,j)$ entry is the conjugate of $a_{ij}$,
the \emph{transpose} of $A$,
denoted by $A^\mathrm{T}$, is the matrix in $M_{n,m}(\mathbb{C})$
whose $(i,j)$ entry is $a_{ji}$, and
the \emph{conjugate transpose} of $A$
is denoted by $A^*:=\overline{A^\mathrm{T}}$.

For any $A\in M_m(\mathbb{C})$, let
\begin{equation}\label{n1}
\|A\|:=\sup_{\vec z\in\mathbb{C}^m,\,|\vec z|=1}|A\vec z|.
\end{equation}

In what follows, we regard $\mathbb{C}^m$ as $M_{m,1}(\mathbb{C})$
and let $\vec{\mathbf{0}}:=(0,\ldots,0)^\mathrm{T}\in\mathbb{C}^m$.
Moreover, for any $\vec z:=(z_1,\ldots,z_m)^\mathrm{T}\in\mathbb{C}^m$,
let $|\vec z|:=(\sum_{i=1}^m|z_i|^2)^{\frac12}$.

Let $A:=[a_{ij}]\in M_m(\mathbb{C})$.
The matrix $A$ is called a \emph{Hermitian matrix} if $A^*=A$
and called a \emph{unitary matrix} if $A^*A=I_m$,
where the \emph{identity matrix} $I_m$ is defined by setting
\begin{equation}\label{Im}
I_m:=\left[\begin{matrix}
1&0&\cdots&0&0\\
0&1&\cdots&0&0\\
\vdots&\vdots&\ddots&\vdots&\vdots\\
0&0&\cdots&1&0\\
0&0&\cdots&0&1
\end{matrix}\right]\in M_m(\mathbb{C}).
\end{equation}
The matrix $A$ is called a \emph{diagonal matrix} if,
for any $i,j\in\{1,\ldots,m\}$ and $i\neq j$, $a_{ij}=0$
and called a \emph{real diagonal matrix}
if it is a diagonal matrix and,
for any $i\in\{1,\ldots,m\}$, $a_{ii}\in\mathbb{R}$.
For any $\{\lambda_i\}_{i=1}^m\subset\mathbb{C}$,
$$
\operatorname{diag}\,(\lambda_1,\ldots,\lambda_m)
:=\left[\begin{matrix}
\lambda_1&0&\cdots&0&0\\
0&\lambda_2&\cdots&0&0\\
\vdots&\vdots&\ddots&\vdots&\vdots\\
0&0&\cdots&\lambda_{m-1}&0\\
0&0&\cdots&0&\lambda_m
\end{matrix}\right]
$$
is called the \emph{diagonal matrix generated by} $\{\lambda_i\}_{i=1}^m$.
If there exist $\lambda\in\mathbb{C}$ and
$\vec z\in\mathbb{C}^m\setminus\{\vec{\mathbf{0}}\}$
such that $A\vec z=\lambda\vec z$, then $\lambda$ is called an \emph{eigenvalue} of $A$
and $\vec z$ an \emph{eigenvector} of $A$ associated with $\lambda$.
The matrix $A$ is said to be \emph{invertible} if
there exists a matrix $A^{-1}\in M_m(\mathbb{C})$ such that $A^{-1}A=I_m$.

Now, we recall the concepts of positive definite matrices
and nonnegative definite matrices (see, for instance, \cite[(7.1.1a) and (7.1.1b)]{hj13}).

\begin{definition}
A matrix $A\in M_m(\mathbb{C})$ is said to be \emph{positive definite}
if, for any $\vec z\in\mathbb{C}^m\setminus\{\vec{\mathbf{0}}\}$, $\vec z^*A\vec z>0$,
and $A$ is said to be \emph{nonnegative definite} if,
for any $\vec z\in\mathbb{C}^m$, $\vec z^*A\vec z\geq0$.
\end{definition}

\begin{remark}\label{1001}
It is well known that any nonnegative definite matrix is always Hermitian
(see, for instance, \cite[Theorem 4.1.4]{hj13}).
\end{remark}

From Remark \ref{1001} and \cite[Theorem 5.6.2(d)]{hj13},
we immediately deduce the following conclusion; we omit the details.

\begin{lemma}\label{exchange}
Let $A,B\in M_m(\mathbb{C})$ be two nonnegative definite matrices.
Then $\|AB\|=\|BA\|$ with the same norm $\|\cdot\|$ as in \eqref{n1}.
\end{lemma}

Let $A\in M_m(\mathbb{C})$ be a positive definite matrix
and have eigenvalues $\{\lambda_i\}_{i=1}^m$.
Due to \cite[Theorem 2.5.6(c)]{hj13},
we find that there exists a unitary matrix $U\in M_m(\mathbb{C})$ such that
\begin{equation}\label{500}
A=U\operatorname{diag}\,(\lambda_1,\ldots,\lambda_m)U^*.
\end{equation}
Moreover, by \cite[Theorem 4.1.8]{hj13},
we find $\{\lambda_i\}_{i=1}^m\subset(0,\infty)$.
The following definition is based on these conclusions
and can be found in \cite[p.\,407]{hj94} (see also \cite[Definition 1.2]{h08}).

\begin{definition}
Let $A\in M_m(\mathbb{C})$ be a positive definite matrix
and have eigenvalues $\{\lambda_i\}_{i=1}^m$.
For any $\alpha\in\mathbb{R}$, define
$$
A^\alpha:=U\operatorname{diag}\left(\lambda_1^\alpha,\ldots,\lambda_m^\alpha\right)U^*,
$$
where $U$ is the same as in \eqref{500}.
\end{definition}

\begin{remark}
From \cite[p.\,408]{hj94}, we infer that $A^\alpha$
is independent of the choices of both the order of $\{\lambda_i\}_{i=1}^m$ and $U$,
and hence $A^\alpha$ is well defined.
\end{remark}

A \emph{scalar weight} is a nonnegative locally integrable function on $\mathbb{R}^n$
that takes values in $(0,\infty)$ almost everywhere.
Next, we recall the concept of scalar $A_p$-weights
(see, for instance, \cite[Definitions 7.1.1 and 7.1.3]{g14c}).
\begin{definition}\label{def Ap}
Let $p\in[1,\infty)$.
A scalar weight $w$ is called an \emph{scalar $A_p(\mathbb{R}^n)$-weight}
if $w$ satisfies that
$$
[w]_{A_1(\mathbb{R}^n)}:=\sup_{\mathrm{cube}\,Q}
\fint_Qw(x)\,dx\left\|w^{-1}\right\|_{L^\infty(Q)}<\infty
$$
or that, when $p\in(1,\infty)$,
$$
[w]_{A_p(\mathbb{R}^n)}:=\sup_{\mathrm{cube}\,Q}
\fint_Qw(x)\,dx\left\{\fint_Q[w(x)]^{-\frac{p'}{p}}\,dx\right\}^{\frac{p}{p'}}<\infty,
$$
where $\frac{1}{p}+\frac{1}{p'}=1$.
Define $A_\infty(\mathbb{R}^n):=\bigcup_{p\in[1,\infty)}A_p(\mathbb{R}^n)$.
\end{definition}

Now, we recall the concept of matrix weights (see, for instance, \cite{nt96,tv97,v97}).

\begin{definition}\label{MatrixWeight}
A matrix-valued function $W:\ \mathbb{R}^n\to M_m(\mathbb{C})$ is called
a \emph{matrix weight} if $W$ satisfies that
\begin{enumerate}[\rm(i)]
\item for any $x\in\mathbb{R}^n$, $W(x)$ is nonnegative definite;
\item for almost every $x\in\mathbb{R}^n$, $W(x)$ is invertible;
\item the entries of $W$ are all locally integrable.
\end{enumerate}
\end{definition}

Corresponding to Definition \ref{def Ap}, we have the following concept of $A_p$-matrix weights
(see, for instance, \cite[p.\,490]{fr21}).

\begin{definition}
Let $p\in(0,\infty)$. A matrix weight $W$ on $\mathbb{R}^n$
is called an $A_p(\mathbb{R}^n,\mathbb{C}^m)$-\emph{matrix weight}
if $W$ satisfies that, when $p\in(0,1]$,
$$
[W]_{A_p(\mathbb{R}^n,\mathbb{C}^m)}
:=\sup_{\mathrm{cube}\,Q}\mathop{\mathrm{\,ess\,sup\,}}_{y\in Q}
\fint_Q\left\|W^{\frac{1}{p}}(x)W^{-\frac{1}{p}}(y)\right\|^p\,dx
<\infty
$$
or that, when $p\in(1,\infty)$,
$$
[W]_{A_p(\mathbb{R}^n,\mathbb{C}^m)}
:=\sup_{\mathrm{cube}\,Q}
\fint_Q\left[\fint_Q\left\|W^{\frac{1}{p}}(x)W^{-\frac{1}{p}}(y)\right\|^{p'}
\,dy\right]^{\frac{p}{p'}}\,dx
<\infty,
$$
where $\frac{1}{p}+\frac{1}{p'}=1$.
\end{definition}

In what follows, if there exists no confusion,
we denote $A_p(\mathbb{R}^n,\mathbb{C}^m)$ simply by $A_p$.
Next, we recall the concept of reducing operators (see, for instance, \cite[(3.1)]{v97}).

\begin{definition}\label{reduce}
Let $p\in(0,\infty)$, $W$ be a matrix weight,
and $E\subset\mathbb{R}^n$ a bounded measurable set satisfying $|E|\in(0,\infty)$.
The matrix $A_E\in M_m(\mathbb{C})$ is called a \emph{reducing operator} of order $p$ for $W$
if $A_E$ is positive definite and,
for any $\vec z\in\mathbb{C}^m$,
\begin{equation}\label{equ_reduce}
\left|A_E\vec z\right|
\sim\left[\fint_E\left|W^{\frac{1}{p}}(x)\vec z\right|^p\,dx\right]^{\frac{1}{p}},
\end{equation}
where the positive equivalence constants depend only on $m$ and $p$.
\end{definition}

\begin{remark}
In Definition \ref{reduce}, the existence of $A_E$ is guaranteed by
\cite[Proposition 1.2]{g03} and \cite[p.\,1237]{fr04}; we omit the details.
\end{remark}

It is useful to know that the relation \eqref{equ_reduce}
also extends to any $M\in M_m(\mathbb{C})$ in place of any vector $\vec z$ as follows.

\begin{lemma}\label{reduceM}
Let $p\in(0,\infty)$, $W$ be a matrix weight,
and $E\subset\mathbb{R}^n$ a bounded measurable set satisfying $|E|\in(0,\infty)$.
If $A_E$ is a reducing operator of order $p$ for $W$,
then, for any matrix $M\in M_m(\mathbb{C})$,
\begin{equation*}
\|A_EM\|\sim\left[\fint_E\left\|W^{\frac{1}{p}}(x)M\right\|^p\,dx\right]^{\frac{1}{p}},
\end{equation*}
where the positive equivalence constants depend only on $m$ and $p$.
\end{lemma}

\begin{proof}
Let $\{\vec{e}_i\}_{i=1}^m$ be any orthonormal basis of $\mathbb C^m$.
By \cite[Lemma 3.2]{ro03}, we find that, for any matrix $M\in M_m(\mathbb{C})$,
$$
\|M\|\sim\left(\sum_{i=1}^m\left|M\vec{e}_i\right|^p\right)^{\frac{1}{p}},
$$
where the positive equivalence constants depend only on $m$ and $p$.
From this and \eqref{equ_reduce}, we deduce that, for any matrix $M\in M_m(\mathbb{C})$,
\begin{align*}
\|A_E M\|^p
&\sim\sum_{i=1}^m\left|A_EM\vec{e}_i\right|^p
\sim\sum_{i=1}^m\fint_E\left|W^{\frac{1}{p}}(x)M\vec{e}_i\right|^p\,dx \\
&=\fint_E\sum_{i=1}^m\left|W^{\frac{1}{p}}(x)M\vec{e}_i\right|^p\,dx
\sim\fint_E\left\|W^{\frac{1}{p}}(x) M \right\|^p\,dx.
\end{align*}
This finishes the proof of Lemma \ref{reduceM}.
\end{proof}

Using Lemma \ref{reduceM}, we obtain an equivalent characterization of $A_p$-matrix weights.

\begin{proposition}\label{equivalent}
Let $p\in(0,1]$.
Then there exists a positive constant $C$, depending only on $m$ and $p$,
such that, for any matrix weight $W$,
$$
[W]_{A_p}\leq[W]_{A_p}^*\leq C[W]_{A_p},
$$
where
$$
[W]_{A_p}^*:=\sup_{\mathrm{cube}\,Q}
\fint_Q\mathop{\mathrm{\,ess\,sup\,}}_{y\in Q}
\left\|W^{\frac{1}{p}}(x)W^{-\frac{1}{p}}(y)\right\|^p\,dx.
$$
\end{proposition}

\begin{proof}
Let $W$ be a matrix weight.
Obviously, $[W]_{A_p}\leq[W]_{A_p}^*$.
Now, we show that $[W]_{A_p}^*\lesssim[W]_{A_p}$.
Let $\{A_Q\}_{\mathrm{cube}\,Q}$ be a family of
reducing operators of order $p$ for $W$.
By Lemma \ref{reduceM}, we find that,
for any cube $Q\subset\mathbb{R}^n$, any $x\in Q$, and almost every $y\in Q$,
\begin{align*}
\left\|W^{\frac{1}{p}}(x)W^{-\frac{1}{p}}(y)\right\|^p
&\leq\left\|W^{\frac{1}{p}}(x)A_Q^{-1}\right\|^p\left\|A_QW^{-\frac{1}{p}}(y)\right\|^p \\
&\sim\left\|W^{\frac{1}{p}}(x)A_Q^{-1}\right\|^p\fint_Q\left\|W^{\frac{1}{p}}(t) W^{-\frac{1}{p}}(y)\right\|^p\,dt\\
&\leq\left\|W^{\frac{1}{p}}(x)A_Q^{-1}\right\|^p[W]_{A_p},
\end{align*}
and hence
$$
[W]_{A_p}^*
\lesssim[W]_{A_p}\sup_{\mathrm{cube}\,Q}
\fint_Q\left\|W^{\frac{1}{p}}(x)A_Q^{-1}\right\|^p\,dx
\sim[W]_{A_p}.
$$
This finishes the proof of Proposition \ref{equivalent}.
\end{proof}

Next, we give a result that is well known in the scalar case.

\begin{proposition}\label{178}
Let $0<p<q<\infty$. Then $A_p\subset A_q$.
Moreover, there exists a positive constant $C$,
depending only on $m$, $p$, and $q$, such that, for any matrix weight $W$,
\begin{equation}\label{232}
[W]_{A_q}\leq C[W]_{A_p}.
\end{equation}
\end{proposition}

\begin{proof}
Let $0<p<q<\infty$ and $W$ be a matrix weight.
We first establish a useful estimate.
By the proof of \cite[Lemma 2]{ipr21}
(in which the symbols $p$ and $q$ are used in the opposite roles,
and it is asumed that $p,q\geq 1$, but the proof works verbatim for any $p,q>0$),
we find that, for any $x\in\mathbb{R}^n$ and almost every $y\in\mathbb{R}^n$,
\begin{equation}\label{176}
\left\|W^{\frac{1}{q}}(x)W^{-\frac{1}{q}}(y)\right\|^q
\lesssim\left\|W^{\frac{1}{p}}(x)W^{-\frac{1}{p}}(y)\right\|^p,
\end{equation}
where the implicit positive constant depends only on $m$, $p$, and $q$.
Now, to prove \eqref{232}, we consider the following three cases on both $p$ and $q$.

\emph{Case 1)} $0<p<q\leq1$.
In this case, using \eqref{176},
we conclude that, for any cube $Q\subset\mathbb{R}^n$
and almost every $y\in\mathbb{R}^n$,
$$
\fint_Q\left\|W^{\frac{1}{q}}(x)W^{-\frac{1}{q}}(y)\right\|^q\,dx
\lesssim\fint_Q\left\|W^{\frac{1}{p}}(x)W^{-\frac{1}{p}}(y)\right\|^p\,dx,
$$
and hence $[W]_{A_q}\lesssim[W]_{A_p}$ in this case.

\emph{Case 2)} $1<p<q<\infty$.
In this case, let $\frac{1}{p}+\frac{1}{p'}=1$ and $\frac{1}{q}+\frac{1}{q'}=1$.
By \eqref{176}, we find that,
for any $x\in\mathbb{R}^n$ and almost every $y\in\mathbb{R}^n$,
\begin{equation}\label{177}
\left\|W^{\frac{1}{q}}(x)W^{-\frac{1}{q}}(y)\right\|^{q'}
\lesssim\left\|W^{\frac{1}{p}}(x)W^{-\frac{1}{p}}(y)\right\|^{p'r},
\end{equation}
where $r:=\frac{p}{p'}\frac{q'}{q}$. Observe that
$r=\frac{p-1}{q-1}\in(0,1).$
From this, \eqref{177}, and H\"older's inequality, we infer that,
for any cube $Q\subset\mathbb{R}^n$,
\begin{align*}
\fint_Q\left[\fint_Q
\left\|W^{\frac{1}{q}}(x)W^{-\frac{1}{q}}(y)\right\|^{q'}\,dy\right]^{\frac{q}{q'}}\,dx
&\lesssim\fint_Q\left[\fint_Q \left\|W^{\frac{1}{p}}(x)W^{-\frac{1}{p}}(y)\right\|^{p'r}\,dy\right]^{\frac{1}{p'r}p}\,dx\\
&\leq\fint_Q\left[\fint_Q
\left\|W^{\frac{1}{p}}(x)W^{-\frac{1}{p}}(y)\right\|^{p'}\,dy\right]^{\frac{1}{p'}p}\,dx,
\end{align*}
and hence $[W]_{A_q}\lesssim[W]_{A_p}$ in this case.

\emph{Case 3)} $0<p\leq1<q<\infty$.
In this case, using \eqref{176}, we obtain,
for any cube $Q\subset\mathbb{R}^n$ and any $x\in Q$,
\begin{align*}
\left[\fint_Q
\left\|W^{\frac{1}{q}}(x)W^{-\frac{1}{q}}(y)\right\|^{q'}\,dy\right]^{\frac{q}{q'}}
&\lesssim\left[\fint_Q
\left\|W^{\frac{1}{p}}(x)W^{-\frac{1}{p}}(y)\right\|^{p\frac{q'}{q}}\,dy\right]^{\frac{q}{q'}}\\
&\leq\mathop{\mathrm{\,ess\,sup\,}}_{y\in Q}
\left\|W^{\frac{1}{p}}(x)W^{-\frac{1}{p}}(y)\right\|^p,
\end{align*}
which, together with Proposition \ref{equivalent}, further implies that
$[W]_{A_q}\lesssim[W]_{A_p}^*\sim[W]_{A_p}.$
This finishes the proof of Proposition \ref{178}.
\end{proof}

\begin{remark}
\begin{enumerate}[\rm(i)]
\item In \cite[Theorem 2.5]{b01},
Bownik showed that $A_p\subset A_q$ when $1<p<q<\infty$ by a different method.
\item Unlike scalar weights, matrix weights have no open property.
Indeed, Bownik \cite[Corollary 4.3]{b01} proved that there exists $W\in A_2$ such that,
for any $p\in(1,2)$, $W\notin A_p$.
\end{enumerate}
\end{remark}

The fundamental facts stated in the following lemma are essentially
contained in \cite[Section 3.3 and Lemma 3.6]{dkps}.

\begin{lemma}\label{Ap dual}
Let $p\in(1,\infty)$, $\frac{1}{p}+\frac{1}{p'}=1$, and $W\in A_p$.
Then $\widetilde W:=W^{-\frac{1}{p-1}}$ satisfies $\widetilde W\in A_{p'}$.
If $A_Q$ and $\widetilde A_Q$ denote the reducing operators,
respectively, of order $p$ for $W$ and of order $p'$ for $\widetilde W$, then
\begin{equation*}
[W]_{A_p}^{\frac{1}{p}}
\sim\left[\widetilde W\right]_{A_{p'}}^{\frac{1}{p'}}
\sim\sup_{\mathrm{cube}\,Q}\left\|A_Q\widetilde A_Q\right\|,
\end{equation*}
where the positive equivalence constants depend only on $m$ and $p$.
Moreover, for any $\vec z\in\mathbb C^m$,
\begin{equation}\label{AQ inv}
\left|A_Q^{-1}\vec z\right|
\sim\left|\widetilde A_Q\vec z\right|
\sim\left[\fint_Q\left|W^{-\frac{1}{p}}(x)\vec z\right|^{p'}\,dx\right]^{\frac{1}{p'}},
\end{equation}
where the positive equivalence constants
depend only on $m$, $p$, and $[W]_{A_p}$.
\end{lemma}

\begin{proof}
By \cite[Section 3.3]{dkps}, we find that $\widetilde W\in A_{p'}$ and
$$
[W]_{A_p}^{\frac{1}{p}}
\sim\left[\widetilde W\right]_{A_{p'}}^{\frac{1}{p'}}
\sim\sup_{\mathrm{cube}\,Q}\left\|A_Q\widetilde A_Q\right\|.
$$
From \cite[Lemma 3.6]{dkps} and \eqref{equ_reduce} with $W$ and $p$
replaced, respectively, by $\widetilde W$ and $p'$,
we deduce that, for any $\vec z\in\mathbb C^m$,
$$
\left|A_Q^{-1}\vec z\right|
\sim\left|\widetilde A_Q\vec z\right|
\sim\left[\fint_Q\left|W^{-\frac{1}{p}}(x)\vec z\right|^{p'}\,dx\right]^{\frac{1}{p'}}.
$$
This finishes the proof of Lemma \ref{Ap dual}.
\end{proof}

Using Lemma \ref{Ap dual}, we obtain the following conclusion immediately.

\begin{corollary}\label{Ap dual corollary}
Let $p\in(1,\infty)$, $\frac{1}{p}+\frac{1}{p'}=1$, $W\in A_p$,
and $\widetilde W:=W^{-\frac{1}{p-1}}$.
Let $Q$ be a cube of $\mathbb{R}^n$ and $A_Q$ and $\widetilde A_Q$
the reducing operators, respectively, of order $p$ for $W$ and of order $p'$ for $\widetilde W$.
Then, for any $M\in M_m(\mathbb C)$,
$$
\left\|A_Q^{-1}M\right\|
\sim\left\|\widetilde A_QM\right\|
\sim\left[\fint_Q\left\|W^{-\frac{1}{p}}(x)M\right\|^{p'}\,dx\right]^{\frac{1}{p'}},
$$
where the positive equivalence constants depend only on $m$, $p$, and $[W]_{A_p}$.
\end{corollary}

\begin{proof}
Using \eqref{AQ inv}, we conclude that, for any $M\in M_m(\mathbb C)$,
\begin{equation*}
\left\|A_Q^{-1}M\right\|
=\sup_{\vec e\in\mathbb C^m,\,|\vec{e}|=1}\left|A_Q^{-1}M\vec e\right|
\sim\sup_{\vec e\in\mathbb C^m,\,|\vec{e}|=1}\left|\widetilde A_QM\vec e\right|
=\left\|\widetilde A_QM\right\|.
\end{equation*}
From Lemma \ref{Ap dual}, we infer that $\widetilde W\in A_{p'}$.
Applying this and Lemma \ref{reduceM} with $W$ and $p$ replaced,
respectively, by $\widetilde W$ and $p'$, we find that,
for any $M\in M_m(\mathbb C)$,
$$
\left\|\widetilde A_QM\right\|
\sim\left[\fint_Q\left\|W^{-\frac{1}{p}}(x)M\right\|^{p'}\,dx\right]^{\frac{1}{p'}}.
$$
This finishes the proof of Corollary \ref{Ap dual corollary}.
\end{proof}

The following lemma is just \cite[Lemma 5.4]{fr04}.

\begin{lemma}\label{AQ inv p<1}
Let $p\in(0,1]$, $W\in A_p$, $Q$ be a cube of $\mathbb{R}^n$,
and $A_Q$ a reducing operator of order $p$ for $W$.
Then, for any $\vec z\in\mathbb{C}^m$,
$$
\left|A_Q^{-1}\vec z\right|
\sim\mathop{\mathrm{\,ess\,sup\,}}_{x\in Q}\left|W^{-\frac{1}{p}}(x)\vec z\right|,
$$
where the positive equivalence constants depend only on $m$, $p$, and $[W]_{A_p}$.
\end{lemma}

Applying Lemma \ref{AQ inv p<1} and an argument similar to that
used in the proof of Corollary \ref{reduceM},
we obtain the following conclusion; we omit the details.

\begin{corollary}\label{AQ inv p<1 corollary}
Let $p\in(0,1]$, $W\in A_p$, $Q$ be a cube of $\mathbb{R}^n$,
and $A_Q$ a reducing operator of order $p$ for $W$.
Then, for any $M\in M_m(\mathbb C)$,
$$
\left\|A_Q^{-1}M\right\|
\sim\mathop{\mathrm{\,ess\,sup\,}}_{x\in Q}\left\|W^{-\frac{1}{p}}(x)M\right\|,
$$
where the positive equivalence constants depend only on $m$, $p$, and $[W]_{A_p}$.
\end{corollary}

Next, we recall the concept of dyadic cubes.
For any $j\in\mathbb{Z}$ and $k:=(k_1,\ldots,k_n)\in\mathbb{Z}^n$, let
$$
Q_{j,k}:=\prod_{i=1}^n2^{-j}[k_i,k_i+1),\
\mathscr{Q}:=\{Q_{j,k}:\ j\in\mathbb{Z},\ k\in\mathbb{Z}^n\},
$$
and $\mathscr{Q}_{j}:=\{Q_{j,k}:\ k\in\mathbb{Z}^n\}.$
For any $Q:=Q_{j,k}\in\mathscr{Q}$, we let $j_Q:=j$ and $x_Q:=2^{-j}k$.

The following lemma is essentially contained in \cite[Lemmas 3.2 and 3.3]{fr21}.

\begin{lemma}\label{8}
Let $p\in(0,\infty)$, $W\in A_p$,
and $\{A_Q\}_{Q\in\mathscr{Q}}$ be a sequence of
reducing operators of order $p$ for $W$.
\begin{enumerate}[{\rm (i)}]
\item\label{8p<1} If $p\in(0,1]$, then
$$
\sup_{Q\in\mathscr{Q}}\mathop{\mathrm{\,ess\,sup\,}}_{x\in Q}
\left\|A_QW^{-\frac{1}{p}}(x)\right\|
\sim[W]_{A_p}^{\frac{1}{p}},
$$
where the positive equivalence constants depend only on $m$ and $p$.
\item\label{8p>1} If $p\in(1,\infty)$,
then there exist a positive constant $\delta$,
depending only on $n$, $m$, $p$, and $[W]_{A_p}$,
and a positive constant $C$,
depending only on $m$ and $p$,
such that, for any $r\in[0,p'+\delta]$,
\begin{equation}\label{8x}
\sup_{Q\in\mathscr{Q}}
\left[\fint_Q\left\|A_QW^{-\frac{1}{p}}(x)\right\|^r\,dx\right]^{\frac{1}{r}}
\leq C[W]_{A_p}^{\frac{1}{p}}.
\end{equation}
\item\label{8allp} For any $p\in(0,\infty)$,
there exist a positive constant $\delta$,
depending only on $n$, $m$, $p$, and $[W]_{A_p}$,
and a positive constant $C$,
depending only on $m$ and $p$, such that,
for any $r\in[0,p+\delta]$,
\begin{equation}\label{8y}
\sup_{Q\in\mathscr{Q}}
\left[\fint_Q\left\|W^{\frac{1}{p}}(x)A_Q^{-1}\right\|^r\,dx\right]^{\frac{1}{r}}
\leq C.
\end{equation}
\item Under the same assumptions as, respectively, in (ii) and (iii),
the following stronger forms of \eqref{8x} and \eqref{8y} are also valid:
\begin{equation}\label{8xx}
\sup_{Q\in\mathscr{Q}}
\fint_Q\sup_{R\in\mathscr{Q},\,x\in R\subset Q}\left\|A_RW^{-\frac{1}{p}}(x)\right\|^r\,dx
\leq C
\end{equation}
and
\begin{equation}\label{8yy}
\sup_{Q\in\mathscr{Q}}\fint_Q\sup_{R\in\mathscr{Q},\,x\in R\subset Q}
\left\|W^{\frac{1}{p}}(x)A_R^{-1}\right\|^r\,dx
\leq C,
\end{equation}
where the positive constants $C$ are allowed to depend on $[W]_{A_p}$ as well.
\end{enumerate}
\end{lemma}

\begin{proof}
The results of \eqref{8p<1}, \eqref{8p>1}, \eqref{8allp}, and \eqref{8yy}
were all already stated in \cite[Lemmas 3.2 and 3.3]{fr21},
but some of these estimates are quoted from the much earlier work \cite[pp.\,207-208 and Lemma 3.3]{g03}.

Now, we need to consider \eqref{8xx} which was not stated in \cite{fr21} as such.
As in \eqref{8x}, we have $p\in(1,\infty)$,
and hence we can consider the dual weight $\widetilde W:=W^{-\frac{1}{p-1}}$
which satisfies $\widetilde W\in A_{p'}$ by Lemma \ref{Ap dual}.
For any $R\in\mathscr{Q}$, let $\widetilde A_R$ denote the reducing operator of order $p'$ for $\widetilde W$.
Then, for any $R\in\mathscr{Q}$ and almost every $x\in\mathbb{R}^n$,
\begin{equation}\label{3.9x}
\left\|A_R W^{-\frac{1}{p}}(x)\right\|
\leq\left\|A_R\widetilde A_R\right\|
\left\|\widetilde A_R^{-1}\widetilde W^{\frac{1}{p'}}(x)\right\|
\lesssim[W]_{A_p}^{\frac{1}{p}}\left\|\widetilde W^{\frac{1}{p'}}(x)\widetilde A_R^{-1}\right\|
\end{equation}
due to Lemmas \ref{Ap dual} and \ref{exchange}. We apply \eqref{8yy} to $p'$
and $\widetilde W\in A_{p'}$ in place of $p$ and $W\in A_p$. This shows that
\begin{equation*}
\sup_{Q\in\mathscr Q}\fint_Q\sup_{R\in\mathscr Q,\,x\in R\subset Q}
\left\|\widetilde W^{\frac{1}{p'}}(x)\widetilde A_R^{-1}\right\|^r\,dx
\leq C
\end{equation*}
for any $r\in[0,p'+\delta]$, which, combined with \eqref{3.9x}, further gives us \eqref{8xx}.
This finishes the proof of Lemma \ref{8}.
\end{proof}

\subsection{The $A_p$-Dimension of Matrix Weights}

There will be a need to estimate integral expressions
like those in the definition of matrix $A_p$-weights,
but involving two different cubes $Q$ and $R$ rather than just one.
This subsection is dedicated to developing some tools for this purpose,
including a new concept of the $A_p$-dimension
that controls the order of growth of such bounds as a function of
the relative size and position of the cubes $Q$ and $R$.
Before introducing this new concept, for the sake of comparison,
we first recall its older relative that has been used for
similar purpose in the existing literature.
The following definition can be found in \cite[p.\,1230]{fr21}.

\begin{definition}\label{doubling matrix weight}
Let $p\in(0,\infty)$.
A matrix weight $W$ is called a \emph{doubling matrix weight} of order $p$
if there exists a positive constant $C$ such that,
for any cube $Q\subset\mathbb{R}^n$ and any $\vec z\in\mathbb{C}^m$,
\begin{equation}\label{243}
\int_{2Q}\left|W^{\frac{1}{p}}(x)\vec z\right|^p\,dx
\leq C\int_Q\left|W^{\frac{1}{p}}(x)\vec z\right|^p\,dx.
\end{equation}
Let
$$
\beta_W:=\min\left\{\beta\in(0,\infty):\ \eqref{243}
\text{ holds with }C=2^\beta\right\}.
$$
Then $\beta_W$ is called the \emph{doubling exponent} of
the doubling matrix weight $W$ of order $p$.
\end{definition}

\begin{remark}
On Definition \ref{doubling matrix weight},
an important observation is that $\beta_W\in[n,\infty)$
(see, for instance, \cite[Proposition 2.10]{hs14}).
\end{remark}

Using both some statements in page 493 of \cite{fr21} and \cite[Lemma 2.2]{fr21},
we have the following conclusion; we omit the details,
as we will only use this result for comparison with
our new variant in Lemma \ref{22 precise} further below.

\begin{lemma}\label{22}
Let $p\in(0,\infty)$ and $W\in A_p$.
Then $W$ is a doubling matrix weight of order $p$.
Moreover, if $\{A_Q\}_{Q\in\mathscr{Q}}$ is a sequence of
reducing operators of order $p$ for $W$,
then there exists a positive constant $C$ such that, for any $Q,R\in\mathscr{Q}$,
$$
\left\|A_QA_R^{-1}\right\|^p
\leq C\max\left\{\left[\frac{\ell(R)}{\ell(Q)}\right]^n,
\left[\frac{\ell(Q)}{\ell(R)}\right]^{\beta_W-n}\right\}
\left[1+\frac{|x_Q-x_R|}{\ell(R)\vee\ell(Q)}\right]^{\beta_W},
$$
where $\beta_W$ is the doubling exponent of the doubling matrix weight $W$ of order $p$.
\end{lemma}

A certain shortcoming of Lemma \ref{22} is the lack of
a reasonable upper bound for the exponent $\beta_W$.
In order to achieve both a sharper form of this estimate
and also some other sharp estimates further below,
we introduce the following useful elaboration of the $A_p$ condition.

\begin{definition}\label{Ap dim}
Let $p\in(0,\infty)$, $d\in\mathbb{R}$, and $W$ be a matrix weight.
Then $W$ is said to have the \emph{$A_p$-dimension $d$},
denoted by $W\in\mathbb{D}_{p,d}(\mathbb{R}^n,\mathbb{C}^m)$,
if there exists a positive constant $C$ such that,
for any cube $Q\subset\mathbb{R}^n$ and any $i\in\mathbb{Z}_+$,
when $p\in(0,1]$,
\begin{equation*}
\mathop{\mathrm{\,ess\,sup\,}}_{y\in2^iQ}\fint_Q \left\|W^{\frac{1}{p}}(x)W^{-\frac{1}{p}}(y)\right\|^p\,dx
\leq C2^{id}
\end{equation*}
or, when $p\in(1,\infty)$,
\begin{equation*}
\fint_Q\left[\fint_{2^iQ}\left\|W^{\frac{1}{p}}(x)W^{-\frac{1}{p}}(y)
\right\|^{p'}\,dy\right]^{\frac{p}{p'}}\,dx
\leq C2^{id},
\end{equation*}
where $\frac{1}{p}+\frac{1}{p'}=1$.
\end{definition}

In what follows, if there exists no confusion,
we denote $\mathbb{D}_{p,d}(\mathbb{R}^n,\mathbb{C}^m)$ simply by $\mathbb{D}_{p,d}$.
We have the following basic properties of $A_p$-dimensions.

\begin{proposition}\label{Ap dim prop}
Let $p\in(0,\infty)$. Then the following statements hold.
\begin{enumerate}[\rm(i)]
\item For any $d\in(-\infty,0)$, $\mathbb{D}_{p,d}=\emptyset$;
\item For any $d\in[0,n)$, $\mathbb{D}_{p,d}\subset A_p$;
\item For any $d\in[n,\infty)$, $\mathbb{D}_{p,d}=A_p$;
\item For any $d_1,d_2\in[0,\infty)$ with $d_1<d_2$,
$\mathbb{D}_{p,d_1}\subset\mathbb{D}_{p,d_2}$;
\item For any $q\in(p,\infty)$ and $d\in[0,\infty)$,
$\mathbb{D}_{p,d}\subset\mathbb{D}_{q,d}$;
\item The definition of $A_p$-dimensions in Definition \ref{Ap dim}
can be equivalently given with $Q$ and $2^i$ therein replaced, respectively,
by ball $B\subset\mathbb{R}^n$ and $\lambda\in[1,\infty)$, or just replace one of them.
\end{enumerate}
\end{proposition}

\begin{proof}
By the definition of $A_p$-dimensions, we directly obtain (ii)-(iv) and (vi).
Applying an argument similar to that used in the proof of Proposition \ref{178},
we find that (v) holds.
Now, we prove (i) by considering the following two cases on $p$.

\emph{Case 1)} $p\in(0,1]$.
In this case, if there exists $W\in\mathbb{D}_{p,d}$,
then, from Definitions \ref{MatrixWeight} and \ref{Ap dim},
we deduce that, for any $i\in\mathbb{Z}_+$,
\begin{align}\label{dim p<1}
0
&<\mathop{\mathrm{\,ess\,sup\,}}_{y\in Q_{0,\mathbf{0}}}
\fint_{Q_{0,\mathbf{0}}}\left\|W^{\frac{1}{p}}(x)W^{-\frac{1}{p}}(y)\right\|^p\,dx\\
&\leq\sup_{\mathrm{cube}\,Q}\mathop{\mathrm{\,ess\,sup\,}}_{y\in2^iQ}
\fint_Q\left\|W^{\frac{1}{p}}(x)W^{-\frac{1}{p}}(y)\right\|^p\,dx
\lesssim2^{id},\notag
\end{align}
which contradicts $d\in(-\infty,0)$,
and hence $\mathbb{D}_{p,d}=\emptyset$ in this case.

\emph{Case 2)} $p\in(1,\infty)$.
In this case, if there exists $W\in\mathbb{D}_{p,d}$,
then, by Definition \ref{MatrixWeight}(iii), we conclude that
$\|W^{\frac{1}{p}}A_{B(\mathbf{0},1)}^{-1}\|^p\in L^1_{\mathrm{loc}},$
where $A_{B(\mathbf{0},1)}$ is the reducing operator of order $p$ for $W$.
This, together with the Lebesgue differentiation theorem
(see, for instance, \cite[Corollary 2.1.16]{g14c}),
Lemma \ref{exchange}, Corollary \ref{Ap dual corollary}, and $d\in(-\infty,0)$,
further implies that, for almost every $x_0\in\mathbb{R}^n$ with $|x_0|<1$,
\begin{align}\label{dim p>1}
\left\|W^{\frac{1}{p}}(x_0)A_{B(\mathbf{0},1)}^{-1}\right\|^p
&=\lim_{i\to\infty}\fint_{B(x_0,2^{-i})}
\left\|W^{\frac{1}{p}}(x)A_{B(\mathbf{0},1)}^{-1}\right\|^p\,dx\\
&\sim\lim_{i\to\infty}\fint_{B(x_0,2^{-i})}\left[\fint_{B(\mathbf{0},1)}
\left\|W^{\frac{1}{p}}(x)W^{-\frac{1}{p}}(y)\right\|^{p'}\,dy\right]^{\frac{p}{p'}}\,dx\notag\\
&\lesssim\lim_{i\to\infty}\fint_{B(x_0,2^{-i})}\left[\fint_{B(x_0,2)}
\left\|W^{\frac{1}{p}}(x)W^{-\frac{1}{p}}(y)\right\|^{p'}\,dy\right]^{\frac{p}{p'}}\,dx\notag\\
&\lesssim\lim_{i\to\infty}2^{id}
=0\notag
\end{align}
and hence all entries of $W(x_0)$ are $0$,
which contradicts Definition \ref{MatrixWeight}(ii).
Thus, $\mathbb{D}_{p,d}=\emptyset$ also in this case.
This finishes the proof of (i) and hence Proposition \ref{Ap dim prop}.
\end{proof}

Next, we establish an equivalent characterization of $A_p$-dimensions.

\begin{proposition}\label{Ap dim rem}
Let $p\in(0,\infty)$, $W\in A_p$, $\{A_Q\}_{\mathrm{cube}\,Q}$
be a family of reducing operators of order $p$ for $W$, and $d\in[0,\infty)$.
Then $W$ has the $A_p$-dimension $d$ if and only if
there exists a positive constant $C$ such that,
for any cube $Q\subset\mathbb{R}^n$ and any $i\in\mathbb{Z}_+$,
$\|A_QA_{2^iQ}^{-1}\|^p\leq C2^{id}.$
\end{proposition}

\begin{proof}
To show the present proposition, we only need to prove that,
for any cube $Q\subset\mathbb{R}^n$ and any $i\in\mathbb{Z}_+$,
\begin{align}\label{equ}
\left\|A_QA_{2^iQ}^{-1}\right\|^p
&\sim\left\{\begin{aligned}
&\mathop{\mathrm{\,ess\,sup\,}}_{y\in2^iQ}
\fint_Q\left\|W^{\frac{1}{p}}(x)W^{-\frac{1}{p}}(y)\right\|^p\,dx
&&\text{if }p\in(0,1],\\
&\fint_Q\left[\fint_{2^iQ}
\left\|W^{\frac{1}{p}}(x)W^{-\frac{1}{p}}(y)\right\|^{p'}\,dy\right]^{\frac{p}{p'}}\,dx
&&\text{if }p\in(1,\infty),
\end{aligned}\right.
\end{align}
where, when $p\in(1,\infty)$, $\frac{1}{p}+\frac{1}{p'}=1$
and the positive equivalence constants depend only on $m$, $p$, and $[W]_{A_p}$.
To this end, we consider the following two cases on $p$.

\emph{Case 1)} $p\in(0,1]$.
In this case, by Lemmas \ref{reduceM} and \ref{exchange} and Corollary \ref{AQ inv p<1 corollary}, we find that,
for any cube $Q\subset\mathbb{R}^n$ and any $i\in\mathbb{Z}_+$,
\begin{align*}
\mathop{\mathrm{\,ess\,sup\,}}_{y\in2^iQ}\fint_Q \left\|W^{\frac{1}{p}}(x)W^{-\frac{1}{p}}(y)\right\|^p\,dx
&\sim\mathop{\mathrm{\,ess\,sup\,}}_{y\in2^iQ}\left\|A_QW^{-\frac{1}{p}}(y)\right\|^p
=\mathop{\mathrm{\,ess\,sup\,}}_{y\in 2^i Q}\left\|W^{-\frac{1}{p}}(y)A_Q\right\|^p\\
&\sim\left\|A_{2^iQ}^{-1}A_Q\right\|^p
=\left\|A_QA_{2^iQ}^{-1}\right\|^p.
\end{align*}
This finishes the proof of \eqref{equ} in this case.

\emph{Case 2)} $p\in(1,\infty)$.
In this case, from Lemma \ref{exchange}, Corollary \ref{Ap dual corollary},
the H\"older inequality, and Lemmas \ref{181} and \ref{reduceM}, we infer that,
for any cube $Q\subset\mathbb{R}^n$ and any $i\in\mathbb{Z}_+$,
\begin{align*}
\fint_Q\left[\fint_{2^iQ}\left\|W^{\frac{1}{p}}(x)W^{-\frac{1}{p}}(y)
\right\|^{p'}\,dy\right]^{\frac{p}{p'}}\,dx
&=\fint_Q\left[\fint_{2^iQ}\left\|W^{-\frac{1}{p}}(y)W^{\frac{1}{p}}(x)
\right\|^{p'}\,dy\right]^{\frac{p}{p'}}\,dx\\
&\sim\fint_Q\left\|A_{2^iQ}^{-1}W^{\frac{1}{p}}(x)\right\|^p\,dx \\
&=\fint_Q\left\|W^{\frac{1}{p}}(x) A_{2^iQ}^{-1}\right\|^p\,dx
\sim\left\|A_QA_{2^iQ}^{-1}\right\|^p.
\end{align*}
This finishes the proof of \eqref{equ} in this case
and hence Proposition \ref{Ap dim rem}.
\end{proof}

Applying an argument similar to that used in the proof of \cite[Lemma 2]{mr22},
we obtain the following reverse H\"older inequality for matrix weights,
which is based on \cite[Theorem 1.1]{hpr12}; we omit the details.

\begin{lemma}\label{181}
Let $p\in(0,\infty)$ and $W\in A_p$.
Then there exist $r(W)\in(1,\infty)$ and a positive constant $C$ such that,
for any $r\in[1,r(W)]$, any cube $Q\subset\mathbb{R}^n$,
and any nonnegative definite matrix $M\in M_m(\mathbb{C})$,
\begin{equation*}
\left[\fint_Q\left\|W^{\frac{1}{p}}(x)M\right\|^{pr}\,dx\right]^{\frac{1}{r}}
\leq C\fint_Q\left\|W^{\frac{1}{p}}(x)M\right\|^{p}\,dx.
\end{equation*}
\end{lemma}

Now, we can establish the relations between $A_p$ and $\mathbb{D}_{p,d}$ with $d\in[0,n)$.

\begin{proposition}\label{180}
Let $p\in(0,\infty)$ and $W\in A_p$.
Then there exists $d\in[0,n)$ such that $W$ has the $A_p$-dimension $d$.
\end{proposition}

\begin{proof}
Let $r:=r(W)$, where $r(W)\in(1,\infty)$ is the same as in Lemma \ref{181}.
We show the existence of the $A_p$-dimension $d\in[0,n)$ by considering the following
two cases on $p$.

\emph{Case 1)} $p\in(0,1]$.
In this case, by both H\"older's inequality and Lemma \ref{181},
we conclude that, for any cube $Q\subset\mathbb{R}^n$ and any $i\in\mathbb{Z}_+$,
\begin{align*}
&\mathop{\mathrm{\,ess\,sup\,}}_{y\in 2^iQ}\fint_Q
\left\|W^{\frac{1}{p}}(x)W^{-\frac{1}{p}}(y)\right\|^p\,dx\\
&\quad\leq\mathop{\mathrm{\,ess\,sup\,}}_{y\in2^iQ}\left[\fint_Q
\left\|W^{\frac{1}{p}}(x)W^{-\frac{1}{p}}(y)\right\|^{pr}\,dx\right]^{\frac{1}{r}}\\
&\quad\leq2^{i\frac{n}{r}}
\mathop{\mathrm{\,ess\,sup\,}}_{y\in2^iQ}\left[\fint_{2^iQ}
\left\|W^{\frac{1}{p}}(x)W^{-\frac{1}{p}}(y)\right\|^{pr}\,dx\right]^{\frac{1}{r}}\\
&\quad\lesssim2^{i\frac{n}{r}}
\mathop{\mathrm{\,ess\,sup\,}}_{y\in2^iQ}\fint_{2^iQ}
\left\|W^{\frac{1}{p}}(x)W^{-\frac{1}{p}}(y)\right\|^p\,dx
\leq2^{i\frac{n}{r}}[W]_{A_p},
\end{align*}
and hence $W$ has the $A_p$-dimension $\frac{n}{r}\in[0,n)$.

\emph{Case 2)} $p\in(1,\infty)$.
In this case, for any cube $Q\subset\mathbb{R}^n$ and any $i\in\mathbb{Z}_+$,
let $A_{2^iQ}$ be a reducing operator of order $p$ for $W$.
From Lemma \ref{exchange}, Corollary \ref{Ap dual corollary}, H\"older's inequality, and Lemmas \ref{181} and \ref{reduceM},
we deduce that, for any cube $Q\subset\mathbb{R}^n$ and any $i\in\mathbb{Z}_+$,
\begin{align*}
&\fint_Q\left[\fint_{2^iQ}
\left\|W^{\frac{1}{p}}(x)W^{-\frac{1}{p}}(y)\right\|^{p'}\,dy\right]^{\frac{p}{p'}}\,dx\\
&\quad=\fint_Q\left[\fint_{2^iQ}
\left\|W^{-\frac{1}{p}}(y)W^{\frac{1}{p}}(x)\right\|^{p'}\,dy\right]^{\frac{p}{p'}}\,dx\\
&\quad\sim\fint_Q \left\|A_{2^iQ}^{-1}W^{\frac{1}{p}}(x)\right\|^p\,dx
=\fint_Q\left\|W^{\frac{1}{p}}(x)A_{2^iQ}^{-1}\right\|^p\,dx\\
&\quad\leq\left[\fint_Q
\left\|W^{\frac{1}{p}}(x)A_{2^iQ}^{-1}\right\|^{pr}\,dx\right]^{\frac{1}{r}}
\leq2^{i\frac{n}{r}}\left[\fint_{2^iQ}
\left\|W^{\frac{1}{p}}(x)A_{2^iQ}^{-1}\right\|^{pr}\,dx\right]^{\frac{1}{r}}\\
&\quad\lesssim2^{i\frac{n}{r}}\fint_{2^iQ}\left\|W^{\frac{1}{p}}(x)A_{2^iQ}^{-1}\right\|^p\,dx
\sim2^{i\frac{n}{r}}\left\|A_{2^iQ}A_{2^iQ}^{-1}\right\|^p
=2^{i\frac{n}{r}},
\end{align*}
and hence $W$ has the $A_p$-dimension $\frac{n}{r}\in[0,n)$.
This finishes the proof of Proposition \ref{180}.
\end{proof}

We will also encounter a need to estimate expressions,
where the roles of $2^i Q$ and $Q$ are interchanged compared to their appearance
in Definition \ref{Ap dim}. The following result proves that
this does not require a new definition.

\begin{proposition}\label{dual Ap dim}
Let $p\in(0,\infty)$ and $W\in A_p$ be a matrix weight.
\begin{enumerate}[{\rm(i)}]
\item If $p\in(0,1]$, then there exists a positive constant $C$ such that,
for any cube $Q\subset\mathbb{R}^n$ and any $i\in\mathbb{Z}_+$,
\begin{equation*}
\mathop{\mathrm{\,ess\,sup\,}}_{y\in Q}\fint_{2^i Q} \left\|W^{\frac{1}{p}}(x)W^{-\frac{1}{p}}(y)\right\|^p\,dx
\leq C.
\end{equation*}
\item If $p\in(1,\infty)$ and $d_2\in\mathbb{R}$, then there exists a positive constant $C$ such that,
for any cube $Q\subset\mathbb{R}^n$ and any $i\in\mathbb{Z}_+$,
\begin{equation}\label{upper Ap dim}
\fint_{2^iQ}\left[\fint_{Q}\left\|W^{\frac{1}{p}}(x)W^{-\frac{1}{p}}(y)
\right\|^{p'}\,dy\right]^{\frac{p}{p'}}\,dx
\leq C2^{id_2}
\end{equation}
if and only if the dual weight $\widetilde W\in A_{p'}$
has the $A_{p'}$-dimension $\widetilde d=\frac{d_2}{p-1}$.
\end{enumerate}
\end{proposition}

\begin{proof}
Let first $p\in(0,1]$. For any $i\in\mathbb Z_+$, observing that the supremum over a smaller set is dominated by the supremum over a bigger one,
\begin{align*}
\mathop{\mathrm{\,ess\,sup\,}}_{y\in Q}
\fint_{2^i Q}\left\|W^{\frac 1p}(x)W^{-\frac 1p}(y)\right\|^p\,dx \leq\mathop{\mathrm{\,ess\,sup\,}}_{y\in 2^i Q}
\fint_{2^i Q}\left\|W^{\frac 1p}(x)W^{-\frac 1p}(y)\right\|^p\,dx
\leq[W]_{A_p},
\end{align*}
so the claim holds with $C=[W]_{A_p}$.

Let then $p\in(1,\infty)$. For each cube $Q$,
let $A_Q$ be the reducing operator of order $p$ for $W$
and $\widetilde A_Q$ the reducing operator of order $p'$
for $\widetilde W=W^{-\frac 1{p-1}}$.
Using the definition of reducing operators
and Lemma \ref{exchange},
we conclude that, for any pair $Q,R$ of cubes,
\begin{align*}
&\left\{\fint_R\left[\fint_Q\left\|W^{\frac 1p}(x)W^{-\frac 1p}(y)\right\|^{p'}\,dy
\right]^{\frac{p}{p'}}\, dx\right\}^{\frac 1p} \\
&\quad\sim\left[\fint_R\left\|W^{\frac 1p}(x)\widetilde A_Q\right\|^p\,dx\right]^{\frac 1p}
\sim\left\|A_R\widetilde A_Q\right\|
\sim\left[\fint_Q\left\|A_R W^{-\frac 1p}(y)\right\|^{p'}\,dy\right]^{\frac 1{p'}} \\
&\quad\sim\left\{\fint_Q\left[\fint_R
\left\|W^{\frac 1p}(x)W^{-\frac 1p}(y)\right\|^{p}\,dx
\right]^{\frac{p'}{p}}\,dy\right\}^{\frac 1{p'}}\\
&\quad= \left\{\fint_Q\left[\fint_R
\left\|\widetilde W^{\frac 1{p'}}(y)\widetilde W^{-\frac 1{p'}}(x)\right\|^{p}\,dx
\right]^{\frac{p'}{p}}\,dy\right\}^{\frac 1{p'}}.
\end{align*}
For $R=2^i Q$, condition \eqref{upper Ap dim}
is equivalent to the boundedness of the left-hand side above by $C 2^{i\frac{d_2}p}$.
On the other hand, the condition that $\widetilde W$ has the $A_{p'}$-dimension $\widetilde d$
is equivalent to the boundedness of the right-hand side above by
$C 2^{i\frac{\widetilde d}{p'}}$. Since both sides are comparable to each other,
it follows that \eqref{upper Ap dim} holds with dimension $d_2$
if and only if $\widetilde W$ has the $A_{p'}$-dimension $\widetilde d$ such that
$\frac{d_2}p=\frac{\widetilde d}{p'}$.
This finishes the proof of Proposition \ref{dual Ap dim}.
\end{proof}

Using the concept of $A_p$-dimensions,
we obtain the following \emph{sharp} estimate that improves Lemma \ref{22}
because $ d<n\leq\beta_W $, where $\beta_W$ is
the doubling exponent of the doubling matrix weight $W$ of order $p$.
For the sharpness, see Lemmas \ref{146} and \ref{146x} further below.

\begin{lemma}\label{22 precise}
Let $p\in(0,\infty)$, let $W\in A_p$ have the $A_p$-dimension $d\in[0,n)$,
and let $\{A_Q\}_{\mathrm{cube}\,Q}$ be a family of
reducing operators of order $p$ for $W$. If $p\in(1,\infty)$,
let further $\widetilde W:=W^{-\frac 1{p-1}}$
(which belongs to $A_{p'}$) have the $A_{p'}$-dimension $\widetilde d$,
while, if $p\in(0,1]$, let $\widetilde d:=0$.
Let
\begin{equation}\label{Delta}
\Delta:=\frac{d}{p}+\frac{\widetilde d}{p'}.
\end{equation}
Then there exists a positive constant $C$ such that,
for any cubes $Q$ and $R$ of $\mathbb{R}^n$,
$$
\left\|A_QA_R^{-1}\right\|
\leq C\max\left\{\left[\frac{\ell(R)}{\ell(Q)}\right]^{\frac{d}{p}},
\left[\frac{\ell(Q)}{\ell(R)}\right]^{\frac{\widetilde d}{p'}}\right\}
\left[1+\frac{|c_Q-c_R|}{\ell(Q)\vee\ell(R)}\right]^\Delta.
$$
\end{lemma}

\begin{proof}
Let us first consider the case when $Q\cap R\neq\emptyset$.
In this case, using some geometrical observations, we obtain
$R\subset\lambda Q$, where $\lambda\sim\max\{\ell(R)/\ell(Q),1\}$.
Next, we claim that
\begin{equation}\label{2203}
\left\|A_QA_R^{-1}\right\|\lesssim
\max\left\{\left[\frac{\ell(R)}{\ell(Q)}\right]^{\frac{d}{p}},
\left[\frac{\ell(Q)}{\ell(R)}\right]^{\frac{\widetilde d}{p'}}\right\}.
\end{equation}
To show this, we consider the following two cases on $p$.

\emph{Case 1)} $p\in(0,1]$.
In this case, we have, for almost every $y\in R$,
\begin{equation}\label{2201}
\left\|A_QA_R^{-1}\right\|^p
\leq\left\|A_QW^{-\frac{1}{p}}(y)\right\|^p\left\|W^{\frac{1}{p}}(y)A_R^{-1}\right\|^p,
\end{equation}
where, by Lemma \ref{reduceM} and Definition \ref{Ap dim}, we find that
\begin{equation*}
\left\|A_QW^{-\frac{1}{p}}(y)\right\|^p
\sim\fint_Q\left\|W^{\frac{1}{p}}(x)W^{-\frac{1}{p}}(y)\right\|^p\,dx
\lesssim\lambda^d
\end{equation*}
because $y\in R\subset\lambda Q$.
Taking an integral average of \eqref{2201} over $y\in R$
and using Lemma \ref{reduceM}, we conclude that
\begin{equation*}
\left\|A_QA_R^{-1}\right\|^p
\lesssim\lambda^d\fint_R\left\|W^{\frac{1}{p}}(y)A_R^{-1}\right\|^p\,dy
\sim\lambda^d\left\|A_RA_R^{-1}\right\|^p
=\lambda^d.
\end{equation*}
This finishes the proof of \eqref{2203} in this case.

\emph{Case 2)} $p\in(1,\infty)$.
In this case, by H\"older's inequality, we obtain
\begin{align}\label{128}
\left\|A_QA_R^{-1}\right\|
&\leq\fint_Q\left\|A_QW^{-\frac{1}{p}}(x)\right\|
\left\|W^{\frac{1}{p}}(x)A_R^{-1}\right\|\,dx\\
&\leq\left[\fint_Q \left\|A_QW^{-\frac{1}{p}}(x)\right\|^{p'}\,dx\right]^{\frac{1}{p'}}
\left[\fint_Q\left\|W^{\frac{1}{p}}(x)A_R^{-1}\right\|^p\,dx\right]^{\frac{1}{p}}
=:\mathrm{I}_1\mathrm{I}_2.\notag
\end{align}
From Lemma \ref{exchange} and Corollary \ref{Ap dual corollary}, we infer that
\begin{equation}\label{129}
\mathrm{I}_1
=\left[\fint_Q\left\|W^{-\frac{1}{p}}(x)A_Q\right\|^{p'}\,dx\right]^{\frac{1}{p'}}
\sim\left\|A_Q^{-1}A_Q\right\|
=1.
\end{equation}
Using Lemma \ref{exchange} and Corollary \ref{Ap dual corollary},
we conclude that, for any $x\in Q$,
$$
\left\|W^{\frac{1}{p}}(x)A_R^{-1}\right\|
=\left\|A_R^{-1} W^{\frac{1}{p}}(x)\right\|
\sim\left[\fint_R\left\|W^{-\frac{1}{p}}(y)W^{\frac{1}{p}}(x)
\right\|^{p'}\,dy\right]^{\frac{1}{p'}}
$$
and hence
\begin{align*}
\mathrm{I}_2
\sim\left\{\fint_Q\left[\fint_R\left\|W^{\frac{1}{p}}(x)W^{-\frac{1}{p}}(y)\right\|^{p'}
\,dy\right]^{\frac{p}{p'}}\,dx\right\}^{\frac{1}{p}}.
\end{align*}
If $\ell(R)\geq\ell(Q)$, then $R\subset 2^i Q$, where $2^i\sim\ell(R)/\ell(Q)$, hence $\fint_R\lesssim\fint_{2^i Q}$, and thus
\begin{equation*}
\mathrm{I}_2\lesssim 2^{i\frac{d}{p}}
\sim\left[\frac{\ell(R)}{\ell(Q)}\right]^{\frac{d}{p}}
\end{equation*}
by Definition \ref{Ap dim} of $A_p$-dimensions.
If $\ell(R)\leq\ell(Q)$, then $Q\subset 2^j R$,
where $2^j\sim\ell(Q)/\ell(R)$, hence $\fint_Q\lesssim\fint_{2^j R}$, and thus
\begin{equation*}
\mathrm{I}_2\lesssim 2^{j\frac{\widetilde d}{p'}}
\sim\left[\frac{\ell(Q)}{\ell(R)}\right]^{\frac{\widetilde d}{p'}}
\end{equation*}
by Proposition \ref{dual Ap dim}. These two bounds for $\mathrm{I}_2$, combined with \eqref{128} and \eqref{129},
further imply \eqref{2203} in this case.

In particular, by \eqref{2203}, we find that
\begin{equation}\label{2202}
\left\|A_QA_R^{-1}\right\|
\lesssim\begin{cases}\displaystyle
\left[\frac{\ell(R)}{\ell(Q)}\right]^{\frac{d}{p}}&\text{if }Q\subset R,\\ \displaystyle
\left[\frac{\ell(Q)}{\ell(R)}\right]^{\frac{\widetilde d}{p'}}&\text{if }R\subset Q.
\end{cases}
\end{equation}

In the general case, we choose a third cube $ S $ such that $ Q \cup R\subset S $.
This clearly can be achieved with
$\ell(S)\sim\ell(Q)+\ell(R)+|c_Q-c_R|$ by some geometrical observations.
From this and \eqref{2202}, we deduce that
\begin{align*}
\left\|A_QA_R^{-1}\right\|
&\leq\left\|A_QA_S^{-1}\right\|\left\|A_SA_R^{-1}\right\|
\lesssim\left[\frac{\ell(S)}{\ell(Q)}\right]^{\frac{d}{p}}
\left[\frac{\ell(S)}{\ell(R)}\right]^{\frac{\widetilde d}{p'}}\\
&=\left[\frac{\ell(Q)\vee\ell(R)}{\ell(Q)}\right]^{\frac{d}{p}}
\left[\frac{\ell(Q)\vee\ell(R)}{\ell(R)}\right]^{\frac{\widetilde d}{p'}}
\left[\frac{\ell(S)}{\ell(Q)\vee\ell(R)}\right]^{\frac{d}{p}+\frac{\widetilde d}{p'}}\\
&\sim\max\left\{\left[\frac{\ell(R)}{\ell(Q)}\right]^{\frac{d}{p}},
\left[\frac{\ell(Q)}{\ell(R)}\right]^{\frac{\widetilde d}{p'}}\right\}
\left[1+\frac{|c_Q-c_R|}{\ell(Q)\vee\ell(R)}\right]^\Delta.
\end{align*}
This finishes the proof of Lemma \ref{22 precise}.
\end{proof}

Since assumptions like those in Lemma \ref{22 precise} will be frequently needed below,
we introduce the following concept.

\begin{definition}
Let $p\in(0,\infty)$ and $W\in A_p$ be a matrix weight. We say that $W$ has $A_p$-dimensions $(d,\widetilde d,\Delta)$ if
\begin{enumerate}[{\rm(i)}]
\item $W$ has the $A_p$-dimension $d$,
\item $p\in(0,1]$ and $\widetilde d=0$, or $p\in(1,\infty)$
and $W^{-\frac 1{p-1}}$ (which belongs to $A_{p'}$)
has the $A_{p'}$-dimension $\widetilde d$, and
\item $\Delta$ is the same as in \eqref{Delta}.
\end{enumerate}
\end{definition}

The following lemma is widely used later in this article;
we omit the details.

\begin{lemma}\label{33y}
For any cubes $Q,R\subset\mathbb{R}^n$,
any $x,x'\in Q$, and any $y,y'\in R$,
$$
1+\frac{|x-y|}{\ell(Q)\vee\ell(R)}
\sim1+\frac{|x'-y'|}{\ell(Q)\vee\ell(R)},
$$
where the positive equivalence constants depend only on $n$.
\end{lemma}

With Lemma \ref{33y}, we obtain the following useful variant of Lemma \ref{22 precise}.

\begin{corollary}\label{237}
Let $p\in(0,\infty)$, let $W\in A_p$ have $A_p$-dimensions $(d,\widetilde d,\Delta)$,
and let $\{A_Q\}_{\mathrm{cube}\,Q}$ be a family of
reducing operators of order $p$ for $W$.
Then there exists a positive constant $C$ such that,
for any $Q,R\in\mathscr{Q}$,
$$
\left\|A_QA_R^{-1}\right\|
\leq C\max\left\{\left[\frac{\ell(R)}{\ell(Q)}\right]^{\frac dp},
\left[\frac{\ell(Q)}{\ell(R)}\right]^{\frac{\widetilde d}{p'}}\right\}
\left[1+\frac{|x_Q-x_R|}{\ell(Q)\vee\ell(R)}\right]^\Delta.
$$
\end{corollary}

\begin{proof}
By Lemmas \ref{22 precise} and \ref{33y}, we conclude that,
for any $Q,R\in\mathscr{Q}$,
\begin{align*}
\left\|A_QA_R^{-1}\right\|^p
&\lesssim\max\left\{\left[\frac{\ell(R)}{\ell(Q)}\right]^{\frac dp},
\left[\frac{\ell(Q)}{\ell(R)}\right]^{\frac{\widetilde d}{p'}}\right\}
\left[1+\frac{|c_Q-c_R|}{\ell(Q)\vee\ell(R)}\right]^\Delta \\
&\sim\max\left\{\left[\frac{\ell(R)}{\ell(Q)}\right]^{\frac dp},
\left[\frac{\ell(Q)}{\ell(R)}\right]^{\frac{\widetilde d}{p'}}\right\}
\left[1+\frac{|x_Q-x_R|}{\ell(Q)\vee\ell(R)}\right]^\Delta.
\end{align*}
This finishes the proof of Corollary \ref{237}.
\end{proof}

\begin{remark}
Observe that, in Corollary \ref{237}, if $Q,R\in\mathscr{Q}_j$ for some $j\in\mathbb{Z}$, then one has
\begin{equation*}
\left\|A_QA_R^{-1}\right\|
\lesssim\left(1+2^j\left|x_Q-x_R\right|\right)^\Delta.
\end{equation*}
From this and Lemma \ref{33y}, we infer that,
for any $j\in\mathbb{Z}$, $Q,R\in\mathscr{Q}_j$,
$x\in Q$, and $y\in R$,
\begin{equation}\label{23y}
\left\|A_QA_R^{-1}\right\|
\lesssim\left(1+2^j|x-y|\right)^\Delta,
\end{equation}
where the implicit positive constant is independent of $j$, $Q$, $R$, $x$, and $y$.
\end{remark}

\subsection{Examples Related to $A_p$-Dimensions}

In this subsection, we explore some further properties of the new concept of
$A_p$-dimensions and provide some illustrating examples,
including ones that prove the sharpness of Lemma \ref{22 precise}.
This subsection is not strictly needed for the subsequent applications of
this concept to the real-variable theory of matrix weighted function spaces
and may be skipped by a reader who prefers to proceed to
the title topic of this article without further delay.

For any matrix weight $W$ and any $p\in(0,\infty)$,
one hopes the $A_p$-dimension of $W$ to be as small as possible.
The following conclusion gives the critical point of the $A_p$-dimension.

\begin{proposition}\label{critical point}
Let $p\in(0,\infty)$, $W\in A_p$, and
\begin{equation}\label{dW}
d_p(W):=\limsup_{i\to\infty}\frac{1}{i}\log_2a_i,
\end{equation}
where, for any $i\in\mathbb{Z}_+$,
$$
a_i:=\begin{cases}
\displaystyle\sup_{\mathrm{cube}\,Q}\mathop{\mathrm{\,ess\,sup\,}}_{y\in2^iQ}\fint_Q
\left\|W^{\frac{1}{p}}(x)W^{-\frac{1}{p}}(y)\right\|^p\,dx
&\text{if }p\in(0,1],\\
\displaystyle\sup_{\mathrm{cube}\,Q}\fint_Q\left[\fint_{2^iQ}
\left\|W^{\frac{1}{p}}(x)W^{-\frac{1}{p}}(y)\right\|^{p'}\,dy\right]^{\frac{p}{p'}}\,dx
&\text{if }p\in(1,\infty)
\end{cases}
$$
with $\frac{1}{p}+\frac{1}{p'}=1$ when $p\in(1,\infty)$.
Then $d_p(W)\in[0,n)$ is the critical point,
that is, for any $\varepsilon\in(0,\infty)$,
$d_p(W)+\varepsilon$ is an $A_p$-dimension of $W$,
but $d_p(W)-\varepsilon$ is not.
\end{proposition}

\begin{proof}
We first show $d_p(W)\in[0,n)$.
By Proposition \ref{180}, we find that there exists $d\in[0,n)$ such that $W$ has the $A_p$-dimension $d$.
This, together with both \eqref{dim p<1} and \eqref{dim p>1},
further implies that there exist two positive constants $C_1$ and $C_2$ such that,
for any $i\in\mathbb{Z}_+$,
$C_1\leq a_i\leq C_22^{id}$
and hence
$$
0
=\limsup_{i\to\infty}\frac{1}{i}\log_2C_1
\leq d_p(W)
\leq\limsup_{i\to\infty}\left(d+\frac{1}{i}\log_2C_2\right)
=d
<n,
$$
that is, $d_p(W)\in[0,n)$.

Next, we prove that, for any $\varepsilon\in(0,\infty)$,
$W$ has the $A_p$-dimension $d_p(W)+\varepsilon$.
From the definition of $d_p(W)$, we deduce that,
for any $\varepsilon\in(0,\infty)$,
there exists $N\in\mathbb{Z}_+$ such that,
for any $i\in\mathbb{N}$ with $i>N$,
$\frac{1}{i}\log_2a_i<d_p(W)+\varepsilon$
and hence $
a_i<2^{i[d_p(W)+\varepsilon]}.$
Thus, for any $i\in\mathbb{Z}_+$,
$$a_i\leq\max\left\{a_0,a_12^{-[d_p(W)+\varepsilon]},
\ldots,a_N2^{-N[d_p(W)+\varepsilon]},1\right\}2^{i[d_p(W)+\varepsilon]}.$$
Therefore, $W$ has the $A_p$-dimension $d_p(W)+\varepsilon$.

Finally, we show that, for any $\varepsilon\in(0,\infty)$,
$d_p(W)-\varepsilon$ is not the $A_p$-dimension of $W$ by contradiction.
If $W$ has the $A_p$-dimension $d_p(W)-\varepsilon$,
then there exists a positive constant $C$ such that, for any $i\in\mathbb{N}$,
\begin{equation}\label{132}
a_i\leq C2^{i[d_p(W)-\varepsilon]}.
\end{equation}
By the definition of $d_p(W)$, we conclude that,
for any $\varepsilon\in(0,\infty)$ and $N\in\mathbb{N}$,
there exists $i\in\mathbb{N}$ with $i>N$ such that
$\frac{1}{i}\log_2a_i>d_p(W)-\frac{\varepsilon}{2}$
and hence $a_i>2^{i[d_p(W)-\frac{\varepsilon}{2}]}$ which contradicts \eqref{132}.
Therefore, $d_p(W)-\varepsilon$ is not the $A_p$-dimension of $W$.
This finishes the proof of Proposition \ref{critical point}.
\end{proof}

\begin{remark}\label{251}
The definition of $d_p(W)$ could have been equivalently given with
cube $Q$ replaced by ball $B\subset\mathbb{R}^n$; we omit the details.
Moreover, Proposition \ref{critical point} proves that
$$
d_p(W)=\inf\left\{d\in\mathbb{R}:\ W\in\mathbb{D}_{p,d}\right\}.
$$
\end{remark}

A natural question is whether or not $d_p(W)$ is the $A_p$-dimension of $W$.
We will give examples to show that either of
$W$ having $A_p$-dimension $d_p(W)$ or not are possible.
To this end, corresponding to Definition \ref{Ap dim},
we introduce the concept of $A_p$-dimensions of scalar weights.

\begin{definition}
Let $p\in[1,\infty)$ and $d\in[0,\infty)$.
A scalar weight $w$ is said to have the \emph{$A_p$-dimension $d$},
denoted by $w\in\mathbb{D}_{p,d}(\mathbb{R}^n)$,
if there exists a positive constant $C$ such that,
for any cube $Q$ and any $i\in\mathbb{Z}_+$, when $p=1$,
\begin{equation*}
\fint_Qw(x)\,dx\left\|w^{-1}\right\|_{L^\infty(2^iQ)}
\leq C2^{id}
\end{equation*}
or, when $p\in(1,\infty)$,
\begin{equation*}
\fint_Qw(x)\,dx\left\{\fint_{2^iQ}[w(x)]^{-\frac{p'}{p}}\,dx\right\}^{\frac{p}{p'}}\,dx
\leq C2^{id},
\end{equation*}
where $\frac{1}{p}+\frac{1}{p'}=1$.
\end{definition}

By the case $m=1$ of both Proposition \ref{Ap dim rem} and \eqref{equ_reduce},
we directly obtain the following conclusion; we omit the details.

\begin{proposition}
Let $p\in[1,\infty)$, $w\in A_p(\mathbb{R}^n)$, and $d\in[0,\infty)$.
Then $w$ has the $A_p$-dimension $d$ if and only if
there exists a positive constant $C$ such that,
for any cube $Q\subset\mathbb{R}^n$ and any $i\in\mathbb{Z}_+$,
$$
\frac{w(Q)}{w(2^iQ)}\leq C2^{i(d-n)}.
$$
\end{proposition}

Using Proposition \ref{critical point} with $m=1$,
we immediately have the following conclusion; we omit the details.
\begin{proposition}
Let $p\in[1,\infty)$, $w\in A_p(\mathbb{R}^n)$, and
\begin{equation}\label{dw}
d_p(w):=\limsup_{i\to\infty}\frac{1}{i}\log_2a_i,
\end{equation}
where, for any $i\in\mathbb{Z}_+$,
$$
a_i:=\left\{\begin{aligned}
&\sup_{\mathrm{cube}\,Q}\fint_Qw(x)\,dx\left\|w^{-1}\right\|_{L^\infty(2^i Q)}
&&\text{if }p=1,\\
&\sup_{\mathrm{cube}\,Q}\fint_Qw(x)\,dx
\left\{\fint_{2^iQ}[w(x)]^{-\frac{p'}{p}}\,dx\right\}^{\frac{p}{p'}}\,dx
&&\text{if }p\in(1,\infty)
\end{aligned}\right.
$$
with $\frac{1}{p}+\frac{1}{p'}=1$ when $p\in(1,\infty)$.
Then $d_p(w)\in[0,n)$ is a critical point,
that is, for any $\varepsilon\in(0,\infty)$,
$d_p(w)+\varepsilon$ is the $A_p$-dimension of $w$
but $d_p(w)-\varepsilon$ is not.
\end{proposition}

The following lemma gives the relation between scalar and matrix weights,
which follows immediately from their definitions; we omit the details.

\begin{lemma}\label{w vs W}
Let $p\in(0,\infty)$ and $d\in[0,\infty)$.
Let $w$ be a scalar weight and $W:=wI_m$,
where $I_m$ is the same as in \eqref{Im}. Then
\begin{enumerate}[\rm(i)]
\item $W\in A_p(\mathbb{R}^n,\mathbb{C}^m)$ if and only if
$w\in A_{\max\{1,p\}}(\mathbb{R}^n)$;
\item $W\in\mathbb{D}_{p,d}(\mathbb{R}^n,\mathbb{C}^m)$ if and only if $w\in\mathbb{D}_{\max\{1,p\},d}(\mathbb{R}^n)$.
\end{enumerate}
\end{lemma}

We are now in a position to give an example of scalar weights
to indicate the attainability of the critical $A_p$-dimensions.

\begin{lemma}\label{exa}
For any $a,b\in\mathbb{R}$, let scalar weight
$w_{a,b}(x):=|x|^a[\log(2+|x|)]^b$ for any $x\in\mathbb{R}^n$.
\begin{enumerate}[{\rm(i)}]
\item For any $a\in(-n,0]$ and $b\in(-\infty,0]$,
one has $w_{a,b}\in A_1(\mathbb{R}^n)$ and $d_1(w_{a,b})=-a$,
where $d_1(w_{a,b})$ is the same as in \eqref{dw};
moreover, $d_1(w_{a,b})$ is the $A_1$-dimension of $w_{a,b}$ if and only if $b=0$.
\item For any $p\in(1,\infty)$, $a\in(-n,n(p-1))$, and $b\in\mathbb{R}$, one has
$w_{a,b}\in A_p(\mathbb{R}^n)$ and $d_p(w_{a,b})=a_-$,
where $d_p(w_{a,b})$ is the same as in \eqref{dw};
moreover, $d_p(w_{a,b})$ is the $A_p$-dimension of $w_{a,b}$
if and only if $a\in(0,n(p-1))$ or $b\in[0,\infty)$.
\end{enumerate}
\end{lemma}

To prove Lemma \ref{exa}, we need the following estimate.

\begin{lemma}\label{135}
Let $a\in(-n,\infty)$, $b\in\mathbb{R}$, and
$w_{a,b}$ be the same as in Lemma \ref{exa}.
Then, for any $x_0\in\mathbb{R}^n$ and $r\in(0,\infty)$,
\begin{equation}\label{ave}
\fint_{B(x_0,r)}w_{a,b}(x)\,dx
\sim(|x_0|+r)^a[\log(2+|x_0|+r)]^b,
\end{equation}
where the positive equivalence constants depend only on $n$, $a$, and $b$.
\end{lemma}

\begin{proof}
We first consider the case that $x_0=\mathbf{0}$.
By a change of variables,
we find that, for any $r\in(0,\infty)$,
\begin{equation}\label{134x}
\fint_{B(\mathbf{0},r)}w_{a,b}(x)\,dx
\sim\frac{1}{r^n}\int_0^rt^{a+n-1}[\log(2+t)]^b\,dt.
\end{equation}
We next claim that, for any $r\in(0,\infty)$,
\begin{equation}\label{133}
\int_0^rt^{a+n-1}[\log(2+t)]^b\,dt
\sim r^{a+n}[\log(2+r)]^b,
\end{equation}
where the positive equivalence constants depend only on $n$, $a$, and $b$.
To show this claim, we consider the following two cases on $b$.

\emph{Case 1)} $b\in[0,\infty)$.
In this case, for any $r\in(0,\infty)$,
by the monotonicity of $\log(2+t)$ on $(0,\infty)$, we obtain
\begin{align*}
\int_0^rt^{a+n-1}[\log(2+t)]^b\,dt
\leq[\log(2+r)]^b\int_0^rt^{a+n-1}\,dt
=\frac{1}{a+n}r^{a+n}[\log(2+r)]^b
\end{align*}
and
\begin{align*}
\int_0^rt^{a+n-1}[\log(2+t)]^b\,dt
&\geq\left[\log\left(2+\frac{r}{2}\right)\right]^b\int_{\frac{r}{2}}^rt^{a+n-1}\,dt\\
&=\frac{1-2^{-(a+n)}}{a+n}r^{a+n}\left[\log\left(2+\frac{r}{2}\right)\right]^b\\
&\geq\frac{1-2^{-(a+n)}}{2^b(a+n)}r^{a+n}[\log(2+r)]^b.
\end{align*}
These finish the proof of \eqref{133} in this case.

\emph{Case 2)} $b\in(-\infty,0)$.
In this case, from the monotonicity of $\log(2+t)$ on $(0,\infty)$ again, it follows that,
for any $r\in(0,\infty)$,
\begin{align*}
\int_0^rt^{a+n-1}[\log(2+t)]^b\,dt
&>[\log(2+r)]^b\int_0^rt^{a+n-1}\,dt
=\frac{1}{a+n}r^{a+n}[\log(2+r)]^b.
\end{align*}
On the other hand, using the integration by parts, we obtain, for any $r\in(0,\infty)$,
\begin{align*}
\int_0^rt^{a+n-1}[\log(2+t)]^b\,dt
&=\frac{1}{a+n}r^{a+n}[\log(2+r)]^b\\
&\quad-\frac{b}{a+n}\int_0^rt^{a+n}\frac{[\log(2+t)]^{b-1}}{2+t}\,dt
\end{align*}
and hence
\begin{equation}\label{new add 2}
\int_0^rh(t)t^{a+n-1}[\log(2+t)]^b\,dt
=\frac{1}{a+n}r^{a+n}[\log(2+r)]^b,
\end{equation}
where
$$
h(t):=1+\frac{b}{a+n}\frac{t}{(2+t)\log(2+t)}.
$$
Notice that
$\lim_{t\to0}h(t)=1=\lim_{t\to\infty}h(t),$
which, combined with \eqref{new add 2}, further implies that
there exists a positive constant $\delta$ such that, for any $r\in(0,\delta)$,
$$
\int_0^rt^{a+n-1}[\log(2+t)]^b\,dt
<\frac{2}{a+n}r^{a+n}[\log(2+r)]^b
$$
and that there exists a positive constant $N$ such that,
for any $r\in(N,\infty)$,
\begin{align*}
&\int_0^rt^{a+n-1}[\log(2+t)]^b\,dt\\
&\quad=\int_0^Nt^{a+n-1}[\log(2+t)]^b\,dt+\int_N^r\cdots\\
&\quad<\int_0^Nt^{a+n-1}[\log(2+t)]^b\,dt+\frac{2}{a+n}r^{a+n}[\log(2+r)]^b,
\end{align*}
For any $r\in(0,\infty)$, let
$$
F(r):=\frac{\int_0^rt^{a+n-1}[\log(2+t)]^b\,dt}{r^{a+n}[\log(2+r)]^b}.
$$
Then $F$ is continuous on $(0,\infty)$,
$F(r)<\frac{2}{a+n}\text{ for any }r\in(0,\delta),$
and there exists a positive constant $\widetilde{N}$ such that,
for any $r\in(\widetilde{N},\infty)$,
$F(r)<1+\frac{2}{a+n}.$
Therefore, $F$ is bounded on $(0,\infty)$.
This finishes the proof of \eqref{133} in this case and hence \eqref{133}.

By both \eqref{134x} and \eqref{133},
we conclude that, for any $r\in(0,\infty)$,
\begin{equation}\label{134}
\fint_{B(\mathbf{0},r)}w_{a,b}(x)\,dx
\sim r^a[\log(2+r)]^b.
\end{equation}

Now, we consider the general case.
Let $x_0\in\mathbb{R}^n$ and $r\in(0,\infty)$ be fixed.
If $|x_0|\geq2r$, then, for any $x\in B(x_0,r)$,
$$
\frac13(|x_0|+r)
\leq|x_0|-r
<|x_0|-|x-x_0|
\leq|x|
\leq|x_0|+|x-x_0|
<|x_0|+r,
$$
and hence $|x|\sim|x_0|+r$, so, in this case, \eqref{ave} is obviously true.
Then we only need to consider the case when $|x_0|<2r$.
In this case, $B(x_0,r)\subset B(\mathbf{0},3r)$.
From this and \eqref{134}, we infer that
\begin{align*}
\fint_{B(x_0,r)}w_{a,b}(x)\,dx
&\lesssim\fint_{B(\mathbf{0},3r)}w_{a,b}(x)\,dx
\sim(3r)^a[\log(2+3r)]^b\\
&\sim r^a[\log(2+r)]^b
\sim(|x_0|+r)^a[\log(2+|x_0|+r)]^b.
\end{align*}
On the other hand,
if $|x_0|<\frac12r$, then $B(x_0,r)\supset B(\mathbf{0},\frac12r)$.
This, together with \eqref{134}, further implies that
\begin{align*}
\fint_{B(x_0,r)}w_{a,b}(x)\,dx
&\gtrsim\fint_{B(\mathbf{0},\frac12r)}w_{a,b}(x)\,dx
\sim\left(\frac{r}{2}\right)^a\left[\log\left(2+\frac{r}{2}\right)\right]^b\\
&\sim r^a[\log(2+r)]^b
\sim(|x_0|+r)^a[\log(2+|x_0|+r)]^b.
\end{align*}
If $|x_0|\geq\frac12r=2\frac14r$, then,
for any $x\in B(x_0,\frac14r)$, $|x|\sim|x_0|+r$ and hence
\begin{equation*}
\fint_{B(x_0,r)}w_{a,b}(x)\,dx
\gtrsim\fint_{B(x_0,\frac14r)}w_{a,b}(x)\,dx
\sim(|x_0|+r)^a[\log(2+|x_0|+r)]^b.
\end{equation*}
This finishes the proof of \eqref{ave} in this case and hence Lemma \ref{135}.
\end{proof}

Lemma \ref{135} remains true if we replace balls $B$ therein by cubes $Q$;
we omit the details.

\begin{corollary}\label{135x}
Let $a\in(-n,\infty)$, $b\in\mathbb{R}$, and
$w_{a,b}$ be the same as in Lemma \ref{exa}.
Then, for any cube $Q\subset\mathbb{R}^n$,
$$
\fint_Qw_{a,b}(x)\,dx
\sim[|c_Q|+\ell(Q)]^a\{\log[2+|c_Q|+\ell(Q)]\}^b,
$$
where the positive equivalence constants depend only on $n$, $a$, and $b$.
\end{corollary}

It is the time for us to prove Lemma \ref{exa}.

\begin{proof}[Proof of Lemma \ref{exa}]
We first show (i).
Let $a\in(-n,0]$ and $b\in(-\infty,0]$.
From \cite[Lemma 2.3(iv)]{fs97}, we deduce that $w_{a,b}\in A_1(\mathbb{R}^n)$.
Next, we calculate $d_1(w_{a,b})$. By Remark \ref{251}, we find that
$$
d_1(w_{a,b})=\limsup_{i\to\infty}\frac{1}{i}\log_2\sup_{\mathrm{ball}\,B}I(B,i),
$$
where
$$
I(B,i):=\fint_B w_{a,b}(x)\,dx
\left\|w_{a,b}^{-1}\right\|_{L^\infty(2^iB)}.
$$
From both Lemma \ref{135} and the definition of $\|\cdot\|_{L^\infty}$, we infer that,
for any $x_0\in\mathbb{R}^n$, $r\in(0,\infty)$, and $i\in\mathbb{Z}_+$,
\begin{equation}\label{136}
I(B(x_0,r),i)\sim\left(\frac{|x_0|+r}{|x_0|+2^ir}\right)^a
\left[\frac{\log(2+|x_0|+r)}{\log(2+|x_0|+2^ir)}\right]^b,
\end{equation}
where the positive equivalence constants depend only on $n$, $a$, and $b$.
Notice that
\begin{align*}
2^{-ia}
&=\sup_{r\in(0,\infty)}
\left(\frac{r}{2^ir}\right)^a
\leq\sup_{x_0\in\mathbb{R}^n,\,r\in(0,\infty)}
\left(\frac{|x_0|+r}{|x_0|+2^ir}\right)^a\\
&\leq\sup_{x_0\in\mathbb{R}^n,\,r\in(0,\infty)}
\left(\frac{|x_0|+r}{2^i|x_0|+2^ir}\right)^a
=2^{-ia},
\end{align*}
and hence
\begin{equation}\label{new add 3}
\sup_{x_0\in\mathbb{R}^n,\,r\in(0,\infty)}
\left(\frac{|x_0|+r}{|x_0|+2^ir}\right)^a
=2^{-ia}.
\end{equation}
Moreover, for any $x_0\in\mathbb{R}^n$, $r\in(0,\infty)$,
and $i\in\mathbb{Z}_+$, we have
$$
\log\left(2+|x_0|+2^ir\right)\geq\log(2+|x_0|+r)
$$
and
\begin{align*}
\log\left(2+|x_0|+2^ir\right)
\leq\log(2+|x_0|+r)^{i+1}
=(i+1)\log(2+|x_0|+r).
\end{align*}
These, combined with both \eqref{136} and \eqref{new add 3},
further imply that, for any $i\in\mathbb Z_+$,
\begin{align}\label{138}
2^{-ia}
&=\sup_{x_0\in\mathbb{R}^n,\,r\in(0,\infty)}
\left(\frac{|x_0|+r}{|x_0|+2^ir}\right)^a
\lesssim\sup_{x_0\in\mathbb{R}^n,\,r\in(0,\infty)} I(B(x_0,r),i)\\
&\lesssim\sup_{x_0\in\mathbb{R}^n,\,r\in(0,\infty)}
\left(\frac{|x_0|+r}{|x_0|+2^ir}\right)^a(i+1)^{-b}
=2^{-ia}(i+1)^{-b},\notag
\end{align}
and hence
$$
d_1\left(w_{a,b}\right)
=\limsup_{i\to\infty}\frac{1}{i}\log_2
\sup_{x_0\in\mathbb{R}^n,\,r\in(0,\infty)}I(B(x_0,r),i)
=-a.
$$

It remains to prove whether or not $d_1(w_{a,b})=-a$
is the $A_1$-dimension of $w_{a,b}$.
To show this, we consider the following two cases on $b$.

\emph{Case 1.1)} $b=0$.
In this case, by both \eqref{136} and $a\in(-n,0]$, we conclude that,
for any $x_0\in\mathbb{R}^n$, $r\in(0,\infty)$, and $i\in\mathbb{Z}_+$,
$$
I(B(x_0,r),i)
\sim\left(\frac{|x_0|+r}{|x_0|+2^ir}\right)^a
=2^{-ia}\left(\frac{|x_0|+r}{2^{-i}|x_0|+r}\right)^a
\leq2^{-ia},
$$
which further implies that $d_1(w_{a,0})=-a$ is the $A_1$-dimension of $w_{a,0}$.

\emph{Case 1.2)} $b\in(-\infty,0)$.
In this case, from \eqref{136}, we deduce that,
for any $i\in\mathbb{Z}_+$,
$$
I(B(\mathbf{0},1),i)
\sim2^{-ia}\left[\frac{\log3}{\log(2+2^i)}\right]^b,
$$
which, together with
$$
\lim_{i\to\infty}\left[\frac{\log3}{\log(2+2^i)}\right]^b=\infty,
$$
further implies that $d_1(w_{a,b})=-a$
is not the $A_1$-dimension of $w_{a,b}$.
This finishes the proof of (i).

Now, we prove (ii).
Let $p\in(1,\infty)$, $a\in(-n,n(p-1))$, and $b\in\mathbb{R}$.
By \cite[Lemma 2.3(v)]{fs97}, we find that $w_{a,b}\in A_p$.
Next, we calculate $d_p(w_{a,b})$.
From Remark \ref{251}, we infer that
$$
d_p(w_{a,b})=\limsup_{i\to\infty}\frac{1}{i}\log_2\sup_{\mathrm{ball}\,B}J(B,i),
$$
where
$$
J(B,i):=\fint_Bw_{a,b}(x)\,dx
\left\{\fint_{2^iB}\left[w_{a,b}(x)\right]^{-\frac{1}{p-1}}\,dx\right\}^{p-1}.
$$
By Lemma \ref{135}, we conclude that,
for any $x_0\in\mathbb{R}^n$, $r\in(0,\infty)$, and $i\in\mathbb{Z}_+$,
\begin{equation}\label{137}
J(B(x_0,r),i)\sim\left(\frac{|x_0|+r}{|x_0|+2^ir}\right)^a
\left[\frac{\log(2+|x_0|+r)}{\log(2+|x_0|+2^ir)}\right]^b,
\end{equation}
where the positive equivalence constants depend only on $n$, $a$, and $b$.
Applying the argument similar to that used in the estimation of \eqref{138},
we obtain, for any $i\in\mathbb Z_+$,
$$
2^{ia_-}(i+1)^{-b_+}
\lesssim\sup_{x_0\in\mathbb{R}^n,\,r\in(0,\infty)}J(B(x_0, r),i)
\lesssim2^{ia_-}(i+1)^{b_-}
$$
and hence
$$
d_p(w_{a,b})
=\limsup_{i\to\infty}\frac{1}{i}\log_2J(B(x_0,r),i)
=a_-.
$$

It remains to show whether or not $d_p(w_{a,b})=a_-$
is the $A_p$-dimension of $w_{a,b}$.
To prove this, we consider the following four cases on both $a$ and $b$.

\emph{Case 2.1)} $a\in(0,n(p-1))$ and $b\in[0,\infty)$.
In this case, by \eqref{137}, we find that,
for any $x_0\in\mathbb{R}^n$, $r\in(0,\infty)$, and $i\in\mathbb{Z}_+$,
$J(B(x_0,r),i)\lesssim1,$
which further implies that $d_p(w_{a,b})=0$
is the $A_p$-dimension of $w_{a,b}$.

\emph{Case 2.2)} $a\in(0,n(p-1))$ and $b\in(-\infty,0)$.
In this case, from \eqref{137}, we deduce that,
for any $x_0\in\mathbb{R}^n$, $r\in(0,\infty)$, and $i\in\mathbb{Z}_+$,
\begin{equation}\label{139}
J(B(x_0,r),i)
\sim\frac{f(|x_0|+r)}{f(|x_0|+2^ir)},
\end{equation}
where $f(t):=t^a[\log(2+t)]^b$ for any $t\in(0,\infty)$.
Notice that, for any $t\in(e^{-\frac{b}{a}},\infty)$,
\begin{align*}
f'(t)
&=\frac{at^a}{2+t}[\log(2+t)]^{b-1}
\left[\frac{2+t}{t}\log(2+t)+\frac{b}{a}\right]\\
&>\frac{at^a}{2+t}[\log(2+t)]^{b-1}
\left(\log t+\frac{b}{a}\right)
>0
\end{align*}
and, for any $t\in(0,-\frac{a}{b})$,
$$
f'(t)>\frac{at^a}{2+t}[\log(2+t)]^{b-1}
\left(\frac{1}{t}+\frac{b}{a}\right)
>0.
$$
Thus, for any $x_0\in\mathbb{R}^n$, $r\in(0,\infty)$,
and $i\in\mathbb{Z}_+$,
if $|x_0|+r>e^{-\frac{b}{a}}$ or $|x_0|+2^ir<-\frac{a}{b}$, then
$$
\frac{f(|x_0|+r)}{f(|x_0|+2^ir)}\leq1;
$$
if $|x_0|+r\leq e^{-\frac{b}{a}}$ and $-\frac{a}{b}\leq|x_0|+2^ir$, then
$$
\frac{f(|x_0|+r)}{f(|x_0|+2^ir)}
\leq\frac{\sup_{t\in(0,e^{-\frac{b}{a}}]}f(t)}{\inf_{t\in[-\frac{a}{b},\infty)}f(t)}
<\infty.
$$
By these and \eqref{139}, we conclude that,
for any $x_0\in\mathbb{R}^n$, $r\in(0,\infty)$, and $i\in\mathbb{Z}_+$,
$J(B(x_0,r),i)\lesssim1,$
which further implies that $d_p(w_{a,b})=0$
is the $A_p$-dimension of $w_{a,b}$.

\emph{Case 2.3)} $a\in(-n,0]$ and $b\in[0,\infty)$.
In this case, using \eqref{137}, $b\in[0,\infty)$, and $a\in(-n,0]$, we obtain,
for any $x_0\in\mathbb{R}^n$, $r\in(0,\infty)$, and $i\in\mathbb{Z}_+$,
$$
J(B(x_0,r),i)
\lesssim\left(\frac{|x_0|+r}{|x_0|+2^ir}\right)^a
=2^{-ia}\left(\frac{|x_0|+r}{2^{-i}|x_0|+r}\right)^a
\leq2^{-ia},
$$
which further implies that $d_p(w_{a,b})=-a$
is the $A_p$-dimension of $w_{a,b}$.

\emph{Case 2.4)} $a\in(-n,0]$ and $b\in(-\infty,0)$.
In this case, from \eqref{137}, we infer that,
for any $i\in\mathbb{Z}_+$,
$$
J(B(\mathbf{0},1),i)
\sim2^{-ia}\left[\frac{\log3}{\log(2+2^i)}\right]^b,
$$
which, combined with
$$
\lim_{i\to\infty}\left[\frac{\log3}{\log(2+2^i)}\right]^b=\infty,
$$
further implies that $d_p(w_{a,b})=-a$
is not the $A_p$-dimension of $w_{a,b}$.
This finishes the proof of (ii) and hence Lemma \ref{exa}.
\end{proof}

The following proposition is now immediately deduced from
Lemmas \ref{w vs W} and \ref{exa}; we omit the details.

\begin{proposition}\label{252}
Let $p\in(0,\infty)$ and $d\in[0,n)$.
Then the following statements hold.
\begin{enumerate}[\rm(i)]
\item There exists $W\in A_p$ having the $A_p$-dimension $d_p(W)=d$,
where $d_p(W)$ is the same as in \eqref{dW}.
\item There exists $W\in A_p$ such that $d_p(W)=d$ but
$d_p(W)$ is not the $A_p$-dimension of $W$.
\end{enumerate}
\end{proposition}

Next, we give more properties of $A_p$-dimensions.

\begin{proposition}\label{256}
Let $p\in(0,\infty)$.
\begin{enumerate}[\rm(i)]
\item For any $d_1,d_2\in[0,n)$ with $d_1<d_2$,
$\mathbb{D}_{p,d_1}\subsetneqq\mathbb{D}_{p,d_2}$.
\item For any $d_0\in[0,n)$,
$$
\bigcup_{d\in[0,d_0)}\mathbb{D}_{p,d}
\subsetneqq\bigcup_{d\in[0,n)}\mathbb{D}_{p,d}
=A_p.
$$
\end{enumerate}
\end{proposition}

\begin{proof}
We first show (i). For any $d_1,d_2\in[0,n)$ with $d_1<d_2$,
by Proposition \ref{252}(i), we find that there exists $W\in A_p$ such that $d_p(W)=d_1$,
where $d_p(W)$ is the same as in \eqref{dW},
but $W\notin\mathbb{D}_{p,d_1}$.
However, from Proposition \ref{critical point} and $d_2>d_1$,
we infer that $W\in\mathbb{D}_{p,d_2}$.
These, together with Proposition \ref{Ap dim prop}(iii), further imply that
$\mathbb{D}_{p,d_1}\subsetneqq\mathbb{D}_{p,d_2}$,
which completes the proof of (i).

Now, we prove (ii). By (i), we conclude that
$\bigcup_{d\in[0,d_0)}\mathbb{D}_{p,d}
\subsetneqq\bigcup_{d\in[0,n)}\mathbb{D}_{p,d}.$
From Propositions \ref{Ap dim prop}(ii) and \ref{180}, we deduce that
$\bigcup_{d\in[0,n)}\mathbb{D}_{p,d}=A_p.$
This finishes the proof of (ii) and hence Proposition \ref{256}.
\end{proof}

Next, we show that Lemma \ref{22 precise} is sharp.

\begin{lemma}\label{146}
Let $p\in(0,1]$, $d\in[0,n)$, and $a,b\in\mathbb{R}$.
Suppose that, for any $W\in A_p$ having the $A_p$-dimension $d$,
there exists a positive constant $C$ such that,
for any cubes $Q$ and $R$ of $\mathbb{R}^n$,
\begin{equation}\label{149}
\left\|A_QA_R^{-1}\right\|
\leq C\max\left\{\left[\frac{\ell(R)}{\ell(Q)}\right]^a,1\right\}
\left[1+\frac{|c_Q-c_R|}{\ell(R)\vee\ell(Q)}\right]^b,
\end{equation}
where $\{A_Q\}_{\mathrm{cube}\,Q}$ is a family of
reducing operators of order $p$ for $W$.
Then $a,b\in[\frac{d}{p},\infty)$.
\end{lemma}

\begin{proof}
Let $W(x):=w_{-d,0}(x)I_m$, where $w_{-d,0}(x):=|x|^{-d}$ is the same as in Lemma \ref{exa}.
By Lemmas \ref{w vs W} and \ref{exa}, we conclude that
$W\in A_p$ has the $A_p$-dimension $d$.
We next claim that, for any $Q\in\mathscr{Q}$ and $\vec z\in\mathbb{C}^m$,
\begin{equation}\label{147}
\left|A_Q\vec z\right|^p
\sim\left[|c_Q|+\ell(Q)\right]^{-d}\left|\vec z\right|^p
\end{equation}
and
\begin{equation}\label{148}
\left|A_Q^{-1}\vec z\right|^p
\sim\left[|c_Q|+\ell(Q)\right]^d\left|\vec z\right|^p.
\end{equation}
Indeed, from \eqref{equ_reduce} and Corollary \ref{135x}, we infer that,
for any $Q\in\mathscr{Q}$ and $\vec z\in\mathbb{C}^m$,
$$
\left|A_Q\vec z\right|^p
\sim\fint_Q\left|W^{\frac{1}{p}}(x)\vec z\right|^p\,dx
=\fint_Qw_{-d,0}(x)\,dx\left|\vec z\right|^p
\sim\left[|c_Q|+\ell(Q)\right]^{-d}\left|\vec z\right|^p.
$$
This finishes the proof of \eqref{147}.
By \eqref{147} with the change of variables $\vec z :=A_Q^{-1}\vec z$, we conclude that
$$
\left|\vec z\right|^p
=\left|A_QA_Q^{-1}\vec z\right|^p
\sim\left[|c_Q|+\ell(Q)\right]^{-d}\left|A_Q^{-1}\vec z\right|^p,
$$
which completes the proof of \eqref{148} and hence the above claim.
By this claim, we find that, for any cubes $Q$ and $R$ of $\mathbb{R}^n$,
\begin{align*}
\left\|A_QA_R^{-1}\right\|^p
&=\sup_{\vec z\in\mathbb{C}^m,\,|\vec z|=1}
\left|A_QA_R^{-1}\vec z\right|^p\\
&\sim\sup_{\vec z\in\mathbb{C}^m,\,|\vec z|=1}
\left[|c_Q|+\ell(Q)\right]^{-d}\left|A_R^{-1}\vec z\right|^p
\sim\left[\frac{|c_R|+\ell(R)}{|c_Q|+\ell(Q)}\right]^d.
\end{align*}
From this and \eqref{149}, we deduce that,
for any cube $Q\subset\mathbb{R}^n$ with $c_Q=\mathbf{0}$
and for any $\lambda\in(1,\infty)$,
$$
\lambda^{\frac{d}{p}}
\sim\left\|A_QA_{\lambda Q}^{-1}\right\|
\lesssim\lambda^a
$$
and hence $a\geq\frac{d}{p}$.
Using the same estimates, we conclude that,
for any cube $Q\subset\mathbb{R}^n$ with $c_Q=\mathbf{0}$ and $\ell(Q)=1$
and for any $x\in\mathbb{R}^n$,
$$
(1+|x|)^{\frac{d}{p}}
\sim\left\|A_QA_{Q+x}^{-1}\right\|
\lesssim(1+|x|)^b
$$
and hence $b\geq\frac{d}{p}$.
This finishes the proof of Lemma \ref{146}.
\end{proof}

Lemma \ref{146} proves that Lemma \ref{22 precise} is sharp when $p\in(0,1]$.
To show that Lemma \ref{146} is still sharp when $p\in(1,\infty)$,
we need the following conclusion which is a simple application of Proposition \ref{dual Ap dim}; we omit the details.

\begin{proposition}\label{dual Ap dim scalar}
Let $p\in[1,\infty)$ and $w\in A_p(\mathbb{R}^n)$.
\begin{enumerate}[{\rm(i)}]
\item If $p=1$, then there exists a positive constant $C$ such that,
for any cube $Q\subset\mathbb{R}^n$ and any $i\in\mathbb{Z}_+$,
\begin{equation*}
\fint_{2^i Q}w(x)\,dx\left\|w^{-1}\right\|_{L^\infty(Q)}
\leq C.
\end{equation*}
\item If $p\in(1,\infty)$ and $d_2\in\mathbb{R}$, then there exists a positive constant $C$ such that,
for any cube $Q\subset\mathbb{R}^n$ and any $i\in\mathbb{Z}_+$,
\begin{equation*}
\fint_{2^iQ}w(x)\,dx\left\{\fint_Q[w(x)]^{-\frac{p'}{p}}\,dx\right\}^{\frac{p}{p'}}
\leq C2^{id_2}
\end{equation*}
if and only if the dual weight $w^{-\frac{1}{p-1}}\in A_{p'}$
has the $A_{p'}$-dimension $\widetilde d=\frac{d_2}{p-1}$.
\end{enumerate}
\end{proposition}

\begin{lemma}\label{146x}
Let $p\in(1,\infty)$, $\frac{1}{p}+\frac{1}{p'}=1$,
$d,\widetilde d\in[0,n)$, and $a,b,c\in\mathbb{R}$.
Suppose that, for any $W\in A_p$ having the $A_p$-dimension $d$
and $\widetilde W:=W^{-\frac 1{p-1}}$ (which belongs to $A_{p'}$)
 having the $A_{p'}$-dimension $\widetilde d$,
there exists a positive constant $C$ such that,
for any cubes $Q$ and $R$ of $\mathbb{R}^n$,
\begin{equation}\label{149x}
\left\|A_QA_R^{-1}\right\|
\leq C\max\left\{\left[\frac{\ell(R)}{\ell(Q)}\right]^a,
\left[\frac{\ell(Q)}{\ell(R)}\right]^b\right\}
\left[1+\frac{|c_Q-c_R|}{\ell(R)\vee\ell(Q)}\right]^c,
\end{equation}
where $\{A_Q\}_{\mathrm{cube}\,Q}$ is a family of
reducing operators of order $p$ for $W$. Then
$a\in[\frac{d}{p},\infty),$ $b\in[\frac{\widetilde d}{p'},\infty),$
and $c\in[\Delta,\infty),$
where $\Delta$ is the same as in \eqref{Delta}.
\end{lemma}

\begin{proof}
Let $x_0:=(1,0,\ldots,0)\in\mathbb{R}^n$.
Let $W:=wI_m$, where, for any $x\in\mathbb{R}^n$,
\begin{align*}
w(x):=w_1(x)w_2(x):=|x|^{-d}|x-x_0|^{(p-1)\widetilde d}.
\end{align*}
We first prove that $w$ has the $A_p$-dimension $d$.
By Lemma \ref{exa}, we find that
$w_1=w_{-d,0}\in A_1(\mathbb{R}^n)$ has the $A_1$-dimension $d$
and $w_2^{-\frac{1}{p-1}}(\cdot)=w_{-\widetilde d,0}(\cdot-x_0)\in A_1(\mathbb{R}^n)$.
From this and Proposition \ref{dual Ap dim scalar}(i),
we infer that, for any cube $Q\subset\mathbb{R}^n$ and any $i\in\mathbb Z_+$,
\begin{equation*}
\fint_{2^iQ}w_2^{-\frac{1}{p-1}}(x)\,dx
\left\|w_2^{\frac{1}{p-1}}\right\|_{L^\infty(Q)}
\lesssim 1,
\end{equation*}
which, combined with $w_1$ having the $A_1$-dimension $d$, further implies that
\begin{align*}
&\fint_Qw(x)\,dx\left\{\fint_{2^iQ}[w(x)]^{-\frac{p'}{p}}\,dx\right\}^{\frac{p}{p'}}\\
&\quad\leq\fint_Qw_1(x)\,dx\left\|w_1^{-1}\right\|_{L^\infty(2^iQ)}
\left[\fint_{2^iQ}w_2^{-\frac{1}{p-1}}(x)\,dx
\left\|w_2^{\frac{1}{p-1}}\right\|_{L^\infty(Q)}\right]^{\frac{p}{p'}}
\lesssim2^{id}
\end{align*}
and hence $w$ has the $A_p$-dimension $d$.
Applying an argument similar to that used in the proof of
$w$ having the $A_p$-dimension $d$, we obtain
$\widetilde w:=w^{-\frac{1}{p-1}}=|x-x_0|^{-\widetilde d}|x|^{\frac{d}{p-1}}$
has the $A_{p'}$-dimension $\widetilde d$.
This, together with $w$ having the $A_p$-dimension $d$ and Lemma \ref{w vs W}(ii),
further implies that $W\in A_p$ has the $A_p$-dimension $d$
and $\widetilde W:=W^{-\frac 1{p-1}}\in A_{p'}$ has the $A_{p'}$-dimension $\widetilde d$.

Now, we estimate $a$.
From Corollary \ref{135x}, we deduce that,
for any cube $Q\subset\mathbb{R}^n$ with $c_Q=\mathbf{0}$ and $\ell(Q)<\frac12$
and for any $M\in M_m(\mathbb{C})$,
\begin{align}\label{revised1}
\|A_QM\|^p
=\fint_Qw(x)\,dx\|M\|^p
\sim\fint_Q|x|^{-d}\,dx\|M\|^p
\sim[\ell(Q)]^{-d}\|M\|^p,
\end{align}
which further implies that
$\|A_Q^{-1}M\|^p\sim[\ell(Q)]^d\|M\|^p.$
By these and \eqref{149x}, we conclude that,
for any cube $Q\subset\mathbb{R}^n$ with $c_Q=\mathbf{0}$ and $\ell(Q)<\frac12$
and for any $\lambda\in(0,1)$,
\begin{align*}
\lambda^{-a}
\gtrsim\left\|A_{\lambda Q}A_Q^{-1}\right\|
\sim[\ell(\lambda Q)]^{-\frac{d}{p}}\left\|A_Q^{-1}\right\|
\sim\lambda^{-\frac{d}{p}}
\end{align*}
and hence $a\geq\frac{d}{p}$.

Next, we estimate $b$. From Lemma \ref{135}, we infer that,
for any cube $Q\subset\mathbb{R}^n$ with $c_Q=x_0$ and $\ell(Q)<\frac12$
and for any $M\in M_m(\mathbb{C})$,
\begin{align*}
\|A_QM\|^p
&=\fint_Qw(x)\,dx\|M\|^p
\sim\fint_Q|x-x_0|^{(p-1)\widetilde d}\,dx\|M\|^p\\
&\sim\fint_{Q-x_0}|x|^{(p-1)\widetilde d}\,dx\|M\|^p
\sim[\ell(Q)]^{(p-1)\widetilde d}\|M\|^p,
\end{align*}
which further implies that
\begin{align}\label{revised2}
\left\|A_Q^{-1}M\right\|^p
\sim[\ell(Q)]^{-(p-1)\widetilde d}\|M\|^p.
\end{align}
By these and \eqref{149x}, we find that,
for any cube $Q\subset\mathbb{R}^n$ with $c_Q=x_0$ and $\ell(Q)<\frac12$
and for any $\lambda\in(0,1)$,
\begin{align*}
\lambda^{-b}
\gtrsim\left\|A_QA_{\lambda Q}^{-1}\right\|
\sim[\ell(Q)]^{\frac{\widetilde d}{p'}}\left\|A_{\lambda Q}^{-1}\right\|
\sim\lambda^{-\frac{\widetilde d}{p'}}
\end{align*}
and hence $b\geq\frac{\widetilde d}{p'}$.

Finally, we estimate $c$.
By \eqref{149x}, \eqref{revised1}, and \eqref{revised2}, we conclude that,
for any cube $Q\subset\mathbb{R}^n$ with $c_Q=\mathbf{0}$ and $\ell(Q)<\frac12$
and for the cube $R\subset\mathbb{R}^n$ with $c_R=x_0$ and $\ell(R)=\ell(Q)$,
\begin{align*}
\left[1+\frac{1}{\ell(Q)}\right]^c
\gtrsim\left\|A_QA_R^{-1}\right\|
\sim[\ell(Q)]^{-\frac{d}{p}}\left\|A_R^{-1}\right\|
\sim\left[\frac{1}{\ell(Q)}\right]^{\frac{d}{p}+\frac{\widetilde d}{p'}}
\end{align*}
and hence $c\geq\frac{d}{p}+\frac{\widetilde d}{p'}=\Delta$.
This finishes the proof of Lemma \ref{146x}.
\end{proof}

\begin{remark}
Lemmas \ref{146} and \ref{146x} show that Lemma \ref{22 precise} is sharp.
\end{remark}

\section{Matrix-Weighted Besov-Type and Triebel--Lizorkin-Type Spaces}
\label{BF type spaces}

In this section, we introduce matrix-weighted Besov-type and Triebel--Lizorkin-type spaces
and obtain their $\varphi$-transform characterization.
When $\tau=0$, our results in this section reduce to
the corresponding ones of Frazier and Roudenko in \cite{fr04,fr21,ro03}.
Let us begin with some concepts.

Let $s\in\mathbb{R}$, $\tau\in[0,\infty)$, and $p,q\in(0,\infty]$.
For any sequence $\{f_j\}_{j\in\mathbb Z}$ of measurable functions on $\mathbb{R}^n$,
any subset $J\subset\mathbb Z$, and any measurable set $E\subset\mathbb{R}^n$, let
\begin{align*}
\|\{f_j\}_{j\in\mathbb Z}\|_{L\dot B_{pq}(E\times J)}
:=\|\{f_j\}_{j\in\mathbb Z}\|_{\ell^qL^p(E\times J)}
:=\|\{f_j\}_{j\in\mathbb Z}\|_{\ell^q(J;L^p(E))}
:=\left[\sum_{j\in J}\|f_j\|_{L^p(E)}^q\right]^{\frac{1}{q}}
\end{align*}
and
\begin{align*}
\|\{f_j\}_{j\in\mathbb Z}\|_{L\dot F_{pq}(E\times J)}
:=\|\{f_j\}_{j\in\mathbb Z}\|_{L^p\ell^q(E\times J)}
:=\|\{f_j\}_{j\in\mathbb Z}\|_{L^p(E;\ell^q(J))}
:=\left\|\left(\sum_{j\in J}|f_j|^q\right)^{\frac{1}{q}}\right\|_{L^p(E)}
\end{align*}
with the usual modification made when $q=\infty$.
For simplicity of the presentation, in what follows, we may drop the domain $E\times J$ from these symbols, when it is the full space $E\times J=\mathbb R^n\times\mathbb Z$.
We use $L\dot A_{pq}\in\{L\dot B_{pq},L\dot F_{pq}\}$ as a generic notation in statements that apply to both types of spaces.

In particular, for any $P\in\mathscr{Q}$, we abbreviate $\widehat{P}:=P\times\{j_P,j_P+1,\ldots\}$ so that
$$
\|\{f_j\}_{j\in\mathbb Z}\|_{L\dot B_{pq}(\widehat{P})}
=\|\{f_j\}_{j\in\mathbb Z}\|_{\ell^qL^p(\widehat{P})}
=\left[\sum_{j=j_P}^\infty\|f_j\|_{L^p(P)}^q\right]^{\frac{1}{q}}
$$
and
$$
\|\{f_j\}_{j\in\mathbb Z}\|_{L\dot F_{pq}(\widehat{P})}
=\|\{f_j\}_{j\in\mathbb Z}\|_{L^p\ell^q(\widehat{P})}
=\left\|\left(\sum_{j=j_P}^\infty|f_j|^q\right)^{\frac{1}{q}}\right\|_{L^p(P)}.
$$
Let us further define
\begin{equation}\label{LApq}
\|\{f_j\}_{j\in\mathbb Z}\|_{L\dot A_{p,q}^\tau}
:=\sup_{P\in\mathscr{Q}}|P|^{-\tau}\|\{f_j\}_{j\in\mathbb Z}\|_{L\dot A_{pq}(\widehat{P})}
\end{equation}
for both choices of $L\dot A_{p,q}^\tau\in\{L\dot B_{p,q}^\tau,L\dot F_{p,q}^\tau\}$.

Moreover, for any $k\in\mathbb{Z}$, let $\{f_j\}_{j\geq k}
:=\{f_j\mathbf{1}_{[k,\infty)}(j)\}_{j\in\mathbb{Z}}$.

Let $\mathcal{S}$ be the space of all Schwartz functions on $\mathbb{R}^n$,
equipped with the well-known topology determined by a countable family of norms,
and let $\mathcal{S}'$ be the set of all continuous linear functionals on $\mathcal{S}$,
equipped with the weak-$*$ topology.
For any $f\in L^1$ and $\xi\in\mathbb{R}^n$, let
$$\widehat{f}(\xi):=\int_{\mathbb{R}^n}f(x)e^{-ix\cdot\xi}\,dx$$
to denote the \emph{Fourier transform} of $f$.
This agrees with the normalisation of the Fourier transform used,
for instance, in \cite[p.\,4]{fjw91} and \cite[p.\,452]{yy10},
and allows us to quote some lemmas from these works directly,
whereas using any other normalisation (such as with $2\pi$ in the exponent)
would also necessitate slight adjustments here and there in several other formulas.
For any $f\in\mathcal S$ and $x\in\mathbb{R}^n$, let
$f^\vee(x):=\widehat{f}(-x)$
to denote the \emph{inverse Fourier transform} of $f$.
It is well known that, for any $f\in\mathcal S$,
$(\widehat{f})^\vee=(f^\vee)^\wedge=f.$
We can also define the \emph{Fourier transform $\widehat{f}$}
and the \emph{inverse Fourier transform $f^\vee$} of any Schwartz distribution $f$ as follows.
For any $f\in\mathcal S'$ and $\varphi\in\mathcal S$, let
$\langle\widehat{f},\varphi\rangle
:=\langle f,\widehat{\varphi}\rangle$
and
$\langle f^\vee,\varphi\rangle
:=\langle f,\varphi^\vee\rangle.$

Let $\varphi,\psi\in\mathcal{S}$ satisfy
\begin{equation}\label{19}
\operatorname{supp}\widehat{\varphi},\operatorname{supp}\widehat{\psi}
\subset\left\{\xi\in\mathbb{R}^n:\ \frac12\leq|\xi|\leq2\right\}
\end{equation}
and
\begin{equation}\label{20}
\left|\widehat\varphi(\xi)\right|,\left|\widehat\psi(\xi)\right|\geq C>0
\text{ if }\xi\in\mathbb{R}^n\text{ with }\frac35\leq|\xi|\leq\frac53,
\end{equation}
where $C$ is a positive constant independent of $\xi$ and
\begin{equation}\label{21}
\sum_{j\in\mathbb{Z}}\overline{\widehat{\varphi}\left(2^j\xi\right)}
\widehat{\psi}\left(2^j\xi\right)=1
\text{ if }\xi\in\mathbb{R}^n\setminus\{\mathbf{0}\}.
\end{equation}

For any complex-valued function $g$ on $\mathbb{R}^n$, let
$\operatorname{supp}g:=\{x\in\mathbb{R}^n:\ g(x)\neq0\}.$
For any $f\in\mathcal{S}'$, let
$$
\operatorname{supp}f:=\bigcap\left\{\text{closed set }K\subset\mathbb{R}^n:\
\langle f,\varphi\rangle=0\text{ for any }\varphi\in\mathcal{S}
\text{ with}\operatorname{supp}\varphi\subset\mathbb{R}^n\setminus K\right\},
$$
which can be found in \cite[Definition 2.3.16]{g14c}.

Let $\varphi$ be a complex-valued function on $\mathbb{R}^n$.
For any $j\in\mathbb{Z}$ and $x\in\mathbb{R}^n$, let
$\varphi_j(x):=2^{jn}\varphi(2^jx)$.
For any $Q:=Q_{j,k}\in\mathscr{Q}$ and $x\in\mathbb{R}^n$, let
$$
\varphi_Q(x)
:=|Q|^{-\frac12}\varphi\left(2^jx-k\right)
=|Q|^{\frac12}\varphi_j(x-x_Q).
$$

As in \cite{yy10}, let
\begin{equation*}
\mathcal{S}_\infty:=\left\{\varphi\in\mathcal{S}:\
\int_{\mathbb R^n}x^\gamma\varphi(x)\,dx=0
\text{ for any }\gamma\in\mathbb{Z}_+^n\right\},
\end{equation*}
regarded as a subspace of $\mathcal{S}$ with the same topology.
We denote by $\mathcal{S}_\infty'$ the space of
all continuous linear functionals on $\mathcal{S}_\infty$, equipped with the weak-$*$ topology.
It is well known that $\mathcal{S}_\infty'$ coincides with
the quotient space $\mathcal{S}'/\mathcal{P}$ as topological spaces,
where $\mathcal{P}$ denotes the set of all polynomials on $\mathbb{R}^n$;
see, for instance, \cite[Chapter 5]{t83}, \cite[Proposition 8.1]{ysy10}, or \cite{Sa17}.

The structure of this section is organized as follows.
In Subsection \ref{W and AQ 1}, we introduce matrix-weighted Besov-type
and Triebel--Lizorkin-type spaces $\dot A^{s,\tau}_{p,q}(W)$
and then corresponding averaging spaces $\dot A^{s,\tau}_{p,q}(\mathbb{A})$,
and we prove that $\dot A^{s,\tau}_{p,q}(W)=\dot A^{s,\tau}_{p,q}(\mathbb{A})$.
In Subsection \ref{W and AQ 2}, we introduce matrix-weighted Besov-type
and Triebel--Lizorkin-type sequence spaces $\dot a^{s,\tau}_{p,q}(W)$
and corresponding averaging spaces $\dot{a}^{s,\tau}_{p,q}(\mathbb{A})$,
and we show that $\dot a^{s,\tau}_{p,q}(W)=\dot{a}^{s,\tau}_{p,q}(\mathbb{A})$.
Finally, in Subsection \ref{phi-transform},
we obtain the $\varphi$-transform characterization of $\dot A^{s,\tau}_{p,q}(W)$.

\subsection{Function Spaces: Definitions and Basic Properties}
\label{W and AQ 1}

First, we recall the concept of Besov-type and Triebel--Lizorkin-type spaces
and corresponding sequence spaces
(see, for instance, \cite[Definitions 1.1 and 3.1]{yy10}).

\begin{definition}\label{d3.1}
Let $s\in\mathbb{R}$, $\tau\in[0,\infty)$, $q\in(0,\infty]$,
and $\varphi\in\mathcal{S}$ satisfy \eqref{19} and \eqref{20}.

The \emph{homogeneous Besov-type space} $\dot B^{s,\tau}_{p,q}$, where $p\in(0,\infty]$,
and the \emph{homogeneous Triebel--Lizorkin-type space}
$\dot F^{s,\tau}_{p,q}$, where $p\in(0,\infty)$, are defined by setting
$$
\dot A^{s,\tau}_{p,q}:=\left\{f\in\mathcal{S}_\infty':\
\|f\|_{\dot A^{s,\tau}_{p,q}}<\infty\right\},
$$
where, for any $f\in\mathcal{S}_\infty'$,
$$
\|f\|_{{\dot A^{s,\tau}_{p,q}}}:=
\left\|\left\{2^{js}\varphi_j*f\right\}_{j\in\mathbb Z}\right\|_{L\dot A_{p,q}^\tau}
$$
with $\|\cdot\|_{L\dot A_{p,q}^\tau}$ the same as in \eqref{LApq}.
\end{definition}

\begin{remark}
In Definition \ref{d3.1}, if we replace the dyadic cube $P$
and the corresponding $j_P$ in \eqref{LApq},
respectively, by arbitrary cube $P$ and the corresponding $\lfloor-\log_2\ell(P)\rfloor$,
we then obtain equivalent quasi-norms.
Similar spaces below have the same property.
\end{remark}

For any $Q\in\mathscr{Q}$,
let $\widetilde{\mathbf{1}}_Q:=|Q|^{-\frac12}\mathbf{1}_Q$.

\begin{definition}
Let $s\in\mathbb{R}$, $\tau\in[0,\infty)$, and $q\in(0,\infty]$.
The \emph{homogeneous Besov-type sequence space} $\dot b^{s,\tau}_{p,q}$, where  $p\in(0,\infty]$, and the \emph{homogeneous Triebel--Lizorkin-type sequence space} $\dot f^{s,\tau}_{p,q}$, where $p\in(0,\infty)$, are defined to be the sets of all sequences
$t:=\{t_Q\}_{Q\in\mathscr{Q}}\subset\mathbb{C}$ such that
$$
\|t\|_{\dot a^{s,\tau}_{p,q}}
:=\left\|\left\{2^{js}t_j\right\}_{j\in\mathbb Z}\right\|_{L\dot A_{p,q}^\tau}
<\infty,
$$
where $\|\cdot\|_{L\dot A_{p,q}^\tau}$ is the same as in \eqref{LApq} and,
for any $j\in\mathbb{Z}$,
\begin{equation}\label{tj}
t_j:=\sum_{Q\in\mathscr{Q}_j}{t}_Q\widetilde{\mathbf{1}}_Q.
\end{equation}
\end{definition}

Above and in what follows, it is understood that the symbols $A$ and $a$
should be consistently replaced either by $B$ and $b$, or by $F$ and $f$,
respectively, throughout the entire statement.

To motivate the definition of matrix-weighted versions of the spaces just introduced, we first recall the concept of the matrix-weighted Lebesgue space
(see, for instance, \cite[p.\,450]{v97}):

\begin{definition}
Let $p\in(0,\infty)$ and $W$ be a matrix weight.
The \emph{matrix-weighted Lebesgue space} $L^p(W,\mathbb{R}^n)$
is defined to be the set of all measurable vector-valued functions
$\vec f:\ \mathbb{R}^n\to\mathbb{C}^m$ such that
$$
\left\|\vec{f}\right\|_{L^p(W,\mathbb{R}^n)}
:=\left[\int_{\mathbb{R}^n}
\left|W^{\frac{1}{p}}(x)\vec{f}(x)\right|^p\,dx\right]^{\frac{1}{p}}<\infty.
$$
\end{definition}

In what follows, we denote $L^p(W,\mathbb{R}^n)$ simply by $L^p(W)$.
For any measurable vector-valued function $\vec f:\ \mathbb{R}^n\to\mathbb{C}^m$
and any measurable set $E$, we define
$\|\vec{f}\|_{L^p(W,E)}:=\|\vec{f}\mathbf{1}_E\|_{L^p(W)}.$

Now, we introduce the matrix-weighted Besov-type and Triebel--Lizorkin-type spaces as follows.

\begin{definition}\label{def 3.4}
Let $s\in\mathbb{R}$, $\tau\in[0,\infty)$, $p\in(0,\infty)$, and $q\in(0,\infty]$.
Let $\varphi\in\mathcal{S}$ satisfy \eqref{19} and \eqref{20}, and let $W\in A_p$ be a matrix weight.
The \emph{homogeneous matrix-weighted Besov-type space}
$\dot B^{s,\tau}_{p,q}(W,\varphi)$
and the \emph{homogeneous matrix-weighted Triebel--Lizorkin-type space}
$\dot F^{s,\tau}_{p,q}(W,\varphi)$
are defined by setting
$$
\dot A^{s,\tau}_{p,q}(W,\varphi)
:=\left\{\vec{f}\in(\mathcal{S}_\infty')^m:\
\left\|\vec{f}\right\|_{\dot A^{s,\tau}_{p,q}(W,\varphi)}<\infty\right\},
$$
where, for any $\vec{f}\in(\mathcal{S}_\infty')^m$,
$$
\left\|\vec{f}\right\|_{\dot A^{s,\tau}_{p,q}(W,\varphi)}
:=\left\|\left\{2^{js}\left|W^{\frac{1}{p}}\left(\varphi_j*\vec f\right)
\right|\right\}_{j\in\mathbb Z}\right\|_{L\dot A_{p,q}^\tau}
$$
with $\|\cdot\|_{L\dot A_{p,q}^\tau}$ the same as in \eqref{LApq}.
\end{definition}

Obviously, for any $p\in(0,\infty)$,
$\dot B^{s,\tau}_{p,p}(W,\varphi)=\dot F^{s,\tau}_{p,p}(W,\varphi)$.

The following lemma is well known; we omit the details.

\begin{lemma}\label{famous 2}
Let $\alpha\in(0,1]$.
Then, for any $\{z_i\}_{i\in\mathbb{N}}\subset\mathbb{C}$,
$(\sum_{i=1}^\infty|z_i|)^{\alpha}
\leq\sum_{i=1}^\infty|z_i|^{\alpha}.$
\end{lemma}

The following proposition gives
a relation between $\dot B^{s,\tau}_{p,q}(W,\varphi)$
and $\dot F^{s,\tau}_{p,q}(W,\varphi)$.
In what follows, the symbol $\subset$ always stands for continuous embedding.

\begin{proposition}\label{39}
Let $s\in\mathbb{R}$, $\tau\in[0,\infty)$, $p\in(0,\infty)$, and $q\in(0,\infty]$.
Let $\varphi\in\mathcal{S}$ satisfy \eqref{19} and \eqref{20}, and let $W\in A_p$. Then
$\dot B_{p,p\wedge q}^{s,\tau}(W,\varphi)
\subset\dot F^{s,\tau}_{p,q}(W,\varphi)
\subset\dot B_{p,p\vee q}^{s,\tau}(W,\varphi).$
Moreover, for any $\vec f\in(\mathcal{S}_\infty')^m$,
$$
\left\|\vec f\right\|_{\dot B^{s,\tau}_{p,p\vee q}(W,\varphi)}\leq\left\|\vec f\right\|_{\dot F^{s,\tau}_{p,q}(W,\varphi)}
\leq\left\|\vec f\right\|_{\dot B_{p,p\wedge q}^{s,\tau}(W,\varphi)}.
$$
\end{proposition}

\begin{proof}
We only consider the case that $q\in(0,\infty)$
because the case that $q=\infty$ is easier and we omit the details.
We first prove that
\begin{equation}\label{40}
\dot B_{p,p\wedge q}^{s,\tau}(W,\varphi)
\subset\dot F^{s,\tau}_{p,q}(W,\varphi).
\end{equation}
For any $\vec f\in(\mathcal{S}_\infty')^m$ and $j\in\mathbb{Z}$, let
$g_j:=2^{js}|W^{\frac{1}{p}}(\varphi_j*\vec f)|.$
From Lemma \ref{famous 2} with $\alpha$ replaced by $\frac{p\wedge q}{q}$
and from the Minkowski integral inequality, we deduce that,
for any $P\in\mathscr{Q}$ and $\vec f\in(\mathcal{S}_\infty')^m$,
\begin{align*}
\left\|\left\{g_j\right\}_{j\in\mathbb Z}\right\|_{L^p\ell^q(\widehat{P})}
&\leq\left\|\left\{\left(g_j\right)^{p\wedge q}\right\}_{j\in\mathbb Z}
\right\|_{L^{\frac{p}{p\wedge q}}\ell^1(\widehat{P})}^{\frac{1}{p\wedge q}}\\
&\leq\left\|\left\{\left(g_j\right)^{p\wedge q}\right\}_{j\in\mathbb Z}
\right\|_{\ell^1L^{\frac{p}{p\wedge q}}(\widehat{P})}^{\frac{1}{p\wedge q}}
=\left\|\left\{g_j\right\}_{j\in\mathbb Z}\right\|_{\ell^{p\wedge q}L^p(\widehat{P})},
\end{align*}
and hence $\|\vec f\|_{\dot F^{s,\tau}_{p,q}(W,\varphi)}
\leq\|\vec f\|_{\dot B_{p,p\wedge q}^{s,\tau}(W,\varphi)}$.
This shows that \eqref{40}.

Next, we prove that
\begin{equation}\label{41}
\dot F^{s,\tau}_{p,q}(W,\varphi)
\subset\dot B_{p,p\vee q}^{s,\tau}(W,\varphi).
\end{equation}
By the Minkowski integral inequality
and Lemma \ref{famous 2} with $\alpha$ replaced by $\frac{p}{p\vee q}$, we conclude that,
for any $P\in\mathscr{Q}$ and $\vec f\in(\mathcal{S}_\infty')^m$,
\begin{align*}
\left\|\left\{g_j\right\}_{j\in\mathbb Z}\right\|_{\ell^{p\vee q}L^p(\widehat{P})}
&=\left\|\left\{\left(g_j\right)^p\right\}_{j\in\mathbb Z}
\right\|_{\ell^{\frac{p\vee q}{p}}L^1(\widehat{P})}^{\frac 1p}\\
&\leq\left\|\left\{\left(g_j\right)^p\right\}_{j\in\mathbb Z}
\right\|_{L^1\ell^{\frac{p\vee q}{p}}(\widehat{P})}^{\frac 1p}
\leq\left\|\left\{g_j\right\}_{j\in\mathbb Z}\right\|_{L^p\ell^q(\widehat{P})},
\end{align*}
and hence $\|\vec f\|_{\dot{B}_{p,p\vee q}^{s,\tau}(W,\varphi)}
\leq\|\vec f\|_{\dot F^{s,\tau}_{p,q}(W,\varphi)}$.
This finishes the proof of \eqref{41} and hence Proposition \ref{39}.
\end{proof}

\begin{definition}
Let $s\in\mathbb{R}$, $\tau\in[0,\infty)$, $p\in(0,\infty)$, and $q\in(0,\infty]$.
Let $\varphi\in\mathcal{S}$ satisfy \eqref{19} and \eqref{20}, let $W\in A_p$,
and let $\mathbb{A}:=\{A_Q\}_{Q\in\mathscr{Q}}$ be
a sequence of reducing operators of order $p$ for $W$.
The \emph{homogeneous averaging matrix-weighted Besov-type space}
$\dot B^{s,\tau}_{p,q}(\mathbb{A},\varphi)$ and
the \emph{homogeneous averaging matrix-weighted Triebel--Lizorkin-type space}
$\dot F^{s,\tau}_{p,q}(\mathbb{A},\varphi)$
are
defined by setting, for both $\dot A^{s,\tau}_{p,q}(\mathbb{A},\varphi)\in\{\dot B^{s,\tau}_{p,q}(\mathbb{A},\varphi),\dot F^{s,\tau}_{p,q}(\mathbb{A},\varphi)\}$,
$$
\dot A^{s,\tau}_{p,q}(\mathbb{A},\varphi)
:=\left\{\vec{f}\in(\mathcal{S}_\infty')^m:\
\left\|\vec{f}\right\|_{\dot A^{s,\tau}_{p,q}(\mathbb{A},\varphi)}<\infty\right\},
$$
where, for any $\vec f\in(\mathcal{S}_\infty')^m$,
$$
\left\|\vec{f}\right\|_{\dot A^{s,\tau}_{p,q}(\mathbb{A},\varphi)}
:=\left\|\left\{2^{js}\left|A_j\left(\varphi_j*\vec f\right)
\right|\right\}_{j\in\mathbb Z}\right\|_{L\dot A_{p,q}^\tau}
$$
with $\|\cdot\|_{L\dot A_{p,q}^\tau}$ the same as in \eqref{LApq} and,
for any $j\in\mathbb{Z}$,
\begin{equation}\label{Aj}
A_j:=\sum_{Q\in\mathscr{Q}_j}A_Q\mathbf{1}_Q.
\end{equation}
\end{definition}

By \eqref{equ_reduce}, we find that
$\dot A^{s,\tau}_{p,q}(\mathbb{A},\varphi)$
is independent of the choice of $\mathbb{A}$.

For any sequence $\mathbb{A}:=\{A_Q\}_{Q\in\mathscr{Q}}$ of matrices,
any $\varphi\in\mathcal{S}_\infty$, and any $\vec f\in(\mathcal{S}_\infty')^m$, let
\begin{equation}\label{sup}
\sup_{\mathbb{A},\varphi}\left(\vec f\right)
:=\left\{\sup_{\mathbb{A},\varphi,Q}\left(\vec f\right)\right\}_{Q\in\mathscr{Q}},
\end{equation}
where, for any $Q\in\mathscr{Q}$,
$$
\sup_{\mathbb{A},\varphi,Q}\left(\vec f\right)
:=|Q|^{\frac12}\sup_{y\in Q}
\left|A_Q\left(\varphi_{j_Q}*\vec f\right)(y)\right|.
$$

The following theorem is the main result of this subsection.

\begin{theorem}\label{11}
Let $s\in\mathbb{R}$, $\tau\in[0,\infty)$, $p\in(0,\infty)$, and $q\in(0,\infty]$.
Let $\varphi\in\mathcal{S}$ satisfy both \eqref{19} and \eqref{20}.
Let $W\in A_p$ and $\mathbb{A}:=\{A_Q\}_{Q\in\mathscr{Q}}$ be a sequence of
reducing operators of order $p$ for $W$.
Then $\vec f\in\dot A^{s,\tau}_{p,q}(W,\varphi)$
if and only if $\vec f\in\dot A^{s,\tau}_{p,q}(\mathbb{A},\varphi)$.
Moreover, for any $\vec f\in(\mathcal{S}_\infty')^m$,
$$
\left\|\vec f\right\|_{\dot A^{s,\tau}_{p,q}(W,\varphi)}
\sim\left\|\sup_{\mathbb{A},\varphi}\left(\vec f\right)\right\|_{\dot a^{s,\tau}_{p,q}}
\sim\left\|\vec f\right\|_{\dot A^{s,\tau}_{p,q}(\mathbb{A},\varphi)},
$$
where the positive equivalence constants are independent of $\vec f$.
\end{theorem}

To show Theorem \ref{11}, we need several technical lemmas.
The following lemma can be proved by some simple computations; we omit the details.

\begin{lemma}\label{253x}
Let $a\in(n,\infty)$.
Then, for any $j\in\mathbb{Z}$ and $y\in\mathbb{R}^n$,
\begin{equation}\label{33x}
\int_{\mathbb{R}^n}\frac{2^{jn}}{(1+|2^jx+y|)^a}\,dx
\sim1.
\end{equation}
Moreover, for any $j\in\mathbb{Z}$ with $j\leq0$ and for any $y\in\mathbb{R}^n$,
\begin{equation}\label{33z}
\sum_{k\in\mathbb{Z}^n}\frac{2^{jn}}{(1+|2^jk+y|)^a}
\sim1.
\end{equation}
Here all the positive equivalence constants depend only on $a$ and $n$.
\end{lemma}

For any $k:=(k_1,\ldots,k_n)\in\mathbb{Z}^n$,
let $\|k\|_{\infty}:=\max_{i\in\{1,\ldots,n\}}|k_i|$.
We also have the following simple estimate; we omit the details.

\begin{lemma}\label{33}
Let $P\in\mathscr{Q}$ and $k\in\mathbb{Z}^n$ with $\|k\|_{\infty}\geq2$.
Then, for any $j\in\{j_P,j_P+1,\ldots\}$, $x\in P$, and $y\in P+k\ell(P)$,
$1+2^j|x-y|\sim2^{j-j_P}|k|,$
where the positive equivalence constants depend only on $n$.
\end{lemma}

The following lemma is a reformulation of
the famous Fefferman--Stein vector-valued maximal inequality:

\begin{lemma}\label{Fefferman Stein}
Let $p\in(1,\infty)$ and $q\in(1,\infty]$.
Then there exists a positive constant $C$ such that,
for any sequence $\{f_j\}_{j\in\mathbb Z}$ of measurable functions on $\mathbb R^n$,
\begin{equation*}
\left\|\left\{\mathcal{M}\left(f_j\right)\right\}_{j\in\mathbb Z}\right\|_{L\dot A_{pq}}
\leq C\left\|\left\{f_j\right\}_{j\in\mathbb Z}\right\|_{L\dot A_{pq}},
\end{equation*}
where $\mathcal{M}$ is the same as in \eqref{maximal}.
\end{lemma}

\begin{proof}
For $L\dot A_{pq}=L\dot F_{pq}$, this is the Fefferman--Stein maximal inequality
$$
\left\|\left\{\sum_{j\in\mathbb Z}\left[\mathcal{M}\left(f_j\right)\right]^q
\right\}^{\frac{1}{q}}\right\|_{L^p}
\leq C\left\|\left(\sum_{j\in\mathbb Z}\left|f_j\right|^q\right)^{\frac{1}{q}}\right\|_{L^p},
$$
which was established in \cite[Theorem 1]{fs71}.

For $L\dot A_{pq}=L\dot B_{pq}$, it is simply the classical Hardy--Littlewood maximal inequality $
\|\mathcal{M}(f_j)\|_{L^p}
\leq C\|f_j\|_{L^p}$ followed by taking $\ell^q$ norms of both sides.
This finishes the proof of Lemma \ref{Fefferman Stein}.
\end{proof}

\begin{lemma}\label{summary F}
Let $s\in\mathbb R$, $\tau\in[0,\infty)$, $p\in(0,\infty)$, $q\in(0,\infty]$, and $M\in(n,\infty)$.
Suppose two sequences $\{g_j\}_{j\in\mathbb{Z}}$ and $\{h_j\}_{j\in\mathbb{Z}}$
of measurable functions on $\mathbb{R}^n$ satisfy:
there exist $r\in(0,\min\{p,q\})$ and a positive constant $C$ such that,
for any $j\in\mathbb{Z}$ and $x\in\mathbb{R}^n$,
\begin{equation}\label{143}
\left|g_j(x)\right|^r
\leq C2^{jn}\int_{\mathbb{R}^n}\frac{1}{(1+2^j|x-z|)^M}\left|h_j(z)\right|^r\,dz.
\end{equation}
Then there exists a positive constant $\widetilde{C}$,
depending only on $C$, $n$, $p$, $q$, and $M$, such that
$$
\left\|\left\{2^{js}g_j\right\}_{j\in\mathbb Z}\right\|_{L\dot A_{p,q}^\tau}
\leq\widetilde{C}\left\|\left\{2^{js}h_j\right\}_{j\in\mathbb Z}\right\|_{L\dot A_{p,q}^\tau}.
$$
\end{lemma}

\begin{proof}
Fix $P\in\mathscr{Q}$.
By \eqref{143}, we obtain, for any $j\in\{j_P,j_P+1,\ldots\}$ and $x\in P$,
\begin{align*}
\left|2^{js}g_j(x)\right|^r
&\lesssim\int_{\mathbb{R}^n}\frac{2^{jn}}{(1+2^j|x-z|)^M}
\left|2^{js}h_j(z)\right|^r\,dz\\
&=\int_{3P}\frac{2^{jn}}{(1+2^j|x-z|)^M}\left|2^{js}h_j(z)\right|^r\,dz
+\sum_{k\in\mathbb{Z}^n,\,\|k\|_\infty\geq2}\int_{P+k\ell(P)}\cdots\\
&=:I_j(x)+J_j(x)
\end{align*}
and hence
\begin{align}\label{150}
\left\|\left\{2^{js}g_j\right\}_{j\in\mathbb Z}\right\|_{L\dot A_{pq}(\widehat{P})}
&=\left\|\left\{\left|2^{js}g_j\right|^r\right\}_{j\in\mathbb Z}
\right\|_{L\dot A_{\frac{p}{r},\frac{q}{r}}(\widehat{P})}^{\frac{1}{r}}\\
&\lesssim\left\|\left\{I_j\right\}_{j\in\mathbb Z}
\right\|_{L^{\frac{p}{r}}\ell^{\frac{q}{r}}(\widehat{P})}^{\frac{1}{r}}
+\left\|\left\{J_j\right\}_{j\in\mathbb Z}
\right\|_{L\dot A_{\frac{p}{r},\frac{q}{r}}(\widehat{P})}^{\frac{1}{r}}.\notag
\end{align}

We first estimate $I_j$.
Using $M\in(n,\infty)$, we conclude that,
for any $j\in\{j_P,j_P+1,\ldots\}$ and $x\in P$,
\begin{align*}
I_j(x)
&=\int_{B(x,2^{-j})}
\frac{2^{jn}}{(1+2^j|x-z|)^M}
\left|2^{js}h_j(z)\right|^r\mathbf{1}_{3P}(z)\,dz
+\sum_{i=1}^\infty\int_{B(x,2^{i-j})\setminus B(x,2^{i-1-j})}\cdots\\
&\lesssim\sum_{i=0}^\infty2^{i(n-M)}
\fint_{B(x,2^{i-j})}\left|2^{js}h_j(z)\right|^r\mathbf{1}_{3P}(z)\,dz
\lesssim\mathcal{M}\left(\left|2^{js}h_j\right|^r\mathbf{1}_{3P}\right)(x),
\end{align*}
where $\mathcal{M}$ is the same as in \eqref{maximal}.
From this and Lemma \ref{Fefferman Stein}, we infer that
\begin{align}\label{58}
\left\|\left\{I_j\right\}_{j\in\mathbb Z}
\right\|_{L\dot A_{\frac{p}{r},\frac{q}{r}}(\widehat{P})}^{\frac{1}{r}}
&\lesssim\left\|\left\{\mathcal{M}\left(\left|2^{js}h_j\right|^r
\mathbf{1}_{3P}\right)\right\}_{j\in\mathbb Z}
\right\|_{L\dot A_{\frac{p}{r},\frac{q}{r}}(\widehat{P})}^{\frac{1}{r}}\\
&\lesssim\left\|\left\{\left|2^{js}h_j\right|^r\right\}_{j\in\mathbb Z}
\right\|_{L\dot A_{\frac{p}{r},\frac{q}{r}}([3P]\times\{j_P,j_P+1,\ldots\})}^{\frac{1}{r}}\notag\\
&=\left\|\left\{2^{js}h_j\right\}_{j\in\mathbb Z}
\right\|_{L\dot A_{p,q}([3P]\times\{j_P,j_P+1,\ldots\})}
\lesssim|P|^\tau\left\|\left\{2^{js}h_j\right\}_{j\in\mathbb Z}\right\|_{L\dot A_{p,q}^\tau}.\notag
\end{align}

Now, we estimate $J_j$.
By Lemma \ref{33} and the assumption that $M\in(n,\infty)$, we find that,
for any $j\in\{j_P,j_P+1,\ldots\}$ and $x\in P$,
\begin{align*}
J_j(x)
&\sim\sum_{k\in\mathbb{Z}^n,\,\|k\|_\infty\geq2}|k|^{-M}
2^{(j-j_P)(n-M)}\fint_{P+k \ell(P)}\left|2^{js}h_j(z)\right|^r\,dz\\
&\leq\sum_{k\in\mathbb{Z}^n,\,\|k\|_\infty\geq2}|k|^{-M}
\mathcal{M}\left(\left|2^{js}h_j\right|^r
\mathbf{1}_{P+k\ell(P)}\right)(x+k\ell(P))
=:\sum_{k\in\mathbb{Z}^n,\,\|k\|_\infty\geq2}|k|^{-M}m_{j,k}(x).
\end{align*}
From this,
the triangle inequality in $L\dot A_{\frac{p}{r},\frac{q}{r}}$,
Lemma \ref{Fefferman Stein},
and the fact that $M\in(n,\infty)$, we deduce that
\begin{align*}
\left\|\left\{J_j\right\}_{j\in\mathbb Z}
\right\|_{L\dot A_{\frac{p}{r},\frac{q}{r}}(\widehat{P})}
&\lesssim\left\|\left\{\sum_{k\in\mathbb{Z}^n,\,\|k\|_\infty\geq2}
|k|^{-M}m_{j,k}\right\}_{j\in\mathbb Z}
\right\|_{L\dot A_{\frac{p}{r},\frac{q}{r}}(\widehat{P})}\\
&\leq\sum_{k\in\mathbb{Z}^n,\,\|k\|_\infty\geq2}|k|^{-M}\left\|\,\left\{m_{j,k}\right\}_{j\in\mathbb Z}\right\|_{L\dot A_{\frac{p}{r},\frac{q}{r}}(\widehat{P})}\\
&\leq\sum_{k\in\mathbb{Z}^n,\,\|k\|_\infty\geq2}|k|^{-M}
\left\|\left\{\mathcal{M}\left(\left|2^{js}h_j\right|^r \mathbf{1}_{P+k\ell(P)}\right)\right\}_{j\geq j_P} \right\|_{L\dot A_{\frac{p}{r},\frac{q}{r}}}\\
&\lesssim\sum_{k\in\mathbb{Z}^n,\,\|k\|_\infty\geq2}|k|^{-M}
\left\|\left\{ \left|2^{js}h_j\right|^r \mathbf{1}_{P+k\ell(P)}\right\}_{j\geq j_P} \right\|_{L\dot A_{\frac{p}{r},\frac{q}{r}}}\\
&=\sum_{k\in\mathbb{Z}^n,\,\|k\|_\infty\geq2}|k|^{-M}
\left\|\left\{ \left|2^{js}h_j\right| \mathbf{1}_{P+k\ell(P)}\right\}_{j\geq j_P} \right\|_{L\dot A_{p,q}}^r\\
&\lesssim\left[|P|^{\tau}
\left\|\left\{2^{js}h_j\right\}_{j\in\mathbb Z}\right\|_{L\dot A_{p,q}^\tau}\right]^r.
\end{align*}
Combined with \eqref{58} and \eqref{150}, this further implies that
$$
\left\|\left\{2^{js}g_j\right\}_{j\in\mathbb Z}\right\|_{L\dot A_{p,q}(\widehat{P})}
\lesssim|P|^{\tau}
\left\|\left\{2^{js}h_j\right\}_{j\in\mathbb Z}\right\|_{L\dot A_{p,q}^\tau},
$$
and hence
$$
\left\|\left\{2^{js}g_j\right\}_{j\in\mathbb Z}\right\|_{L\dot A_{p,q}^\tau}
\lesssim\left\|\left\{2^{js}h_j\right\}_{j\in\mathbb Z}\right\|_{L\dot A_{p,q}^\tau}.
$$
This finishes the proof of Lemma \ref{summary F}.
\end{proof}

\begin{remark}
Let $s\in\mathbb{R}$, $\tau\in[0,\infty)$, $p\in(0,\infty)$, and $q\in(0,\infty]$.
Let $\varphi\in\mathcal{S}$ satisfy both \eqref{19} and \eqref{20}.
Let $W\in A_p$ and $\mathbb{A}:=\{A_Q\}_{Q\in\mathscr{Q}}$
be a sequence of reducing operators of order $p$ for $W$.
Let $L\dot A_{p,q}^{\tau}\in\{L\dot B_{p,q}^\tau,L\dot F_{p,q}^\tau\}$.
Observe that the norms of many spaces can be represented via $\|\cdot\|_{L\dot A_{p,q}^\tau}$, for instance,
$$
\left\|\left\{2^{js}g_j\right\}_{j\in\mathbb Z}\right\|_{L\dot A_{p,q}^\tau}
=\begin{cases}
\left\|\vec f\right\|_{\dot A^{s,\tau}_{p,q}(W)}
&\text{if }g_j:=\left|W^{\frac{1}{p}}\left(\varphi_j*\vec f\right)\right|,\ \forall\,j\in\mathbb Z,\\
\left\|\vec f\right\|_{\dot A^{s,\tau}_{p,q}(\mathbb{A})}
&\text{if }g_j:=\left|A_j \left(\varphi_j*\vec f\right)\right|,\ \forall\,j\in\mathbb Z,\\
\|t\|_{\dot a^{s,\tau}_{p,q}}
&\text{if }g_j:=t_j,\ \forall\,j\in\mathbb Z,
\end{cases}
$$
where, for any $j\in\mathbb Z$, $A_j$ and $t_j$ are the same as, respectively, in \eqref{Aj} and \eqref{tj},
and hence Lemma \ref{summary F} provides a unified way to compare above norms.
\end{remark}

By \cite[Theorem 2.3.21]{g14c}, we obtain,
if $f\in\mathcal{S}'$ and $\widehat{f}$ has compact support, then $f\in C^\infty$,
where $C^\infty$ denotes the set of all infinitely differentiable functions on $\mathbb{R}^n$.
The following lemma can be found in the proof of \cite[Theorem 2.4]{fr21}.
For the convenience of the reader, we give the details of its proof here.

\begin{lemma}\label{10x}
Let $\gamma\in\mathcal{S}$ satisfy
$\widehat{\gamma}(\xi)=1$ for any $\xi\in\mathbb{R}^n$ with $|\xi|\leq2$ and
$$
\operatorname{supp}\widehat{\gamma}\subset\{\xi\in\mathbb{R}^n:\ |\xi|<\pi\}.
$$
Then, for any $j\in\mathbb{Z}$ and $f\in\mathcal{S}'$
with $\operatorname{supp}\widehat{f}\subset\{\xi\in\mathbb{R}^n:\ |\xi|\leq2^{j+1}\}$,
one has $f\in C^\infty$ and, for any $x,y\in\mathbb{R}^n$,
\begin{equation}\label{10y}
f(x)=\sum_{R\in\mathscr{Q}_j}2^{-jn}f(x_R+y)\gamma_j(x-x_R-y)
\end{equation}
pointwise.
\end{lemma}

\begin{proof}
By \cite[Theorem 2.3.21]{g14c} and $f\in\mathcal{S}'$
with $\operatorname{supp}\widehat{f}\subset\{\xi\in\mathbb{R}^n:\ |\xi|\leq2^{j+1}\}$
for some $j\in\mathbb Z$, we find that $f\in C^\infty$.
Let $j\in\mathbb{Z}$, $y\in\mathbb{R}^n$, and $g(\cdot):=f(\cdot+y)$.
Then
$$
\operatorname{supp}\widehat{g}
=\operatorname{supp}\widehat{f}
\subset\left\{\xi\in\mathbb{R}^n:\ |\xi|\leq2^{j+1}\right\},
$$
which, together with $\widehat{\gamma_j}(\xi)=1$
for any $\xi\in\mathbb{R}^n$ with $|\xi|\leq2^{j+1}$,
further implies that $g=g*\gamma_j$.
From this and \cite[Lemma 6.10]{fjw91}, we infer that,
for any $x\in\mathbb{R}^n$,
$$
g(x)
=\left(g*\gamma_j\right)(x)
=\sum_{R\in\mathscr{Q}_j}2^{-jn}g(x_R)\gamma_j(x-x_R),
$$
and hence
$$
f(x+y)=\sum_{R\in\mathscr{Q}_j}2^{-jn}f(x_R+y)\gamma_j(x-x_R).
$$
By a change of variables, we obtain \eqref{10y}.
This finishes the proof of Lemma \ref{10x}.
\end{proof}

\begin{remark}
Let $f\in\mathcal{S}_\infty'$ and $\varphi$ satisfy both Fourier support conditions \eqref{19} and \eqref{20}.
Notice that, for any $j\in\mathbb{Z}$,
$$
\operatorname{supp}\widehat{\varphi_j*f}\subset\operatorname{supp}\widehat{\varphi_j}
\subset\left\{\xi\in\mathbb{R}^n:\ |\xi|\leq2^{j+1}\right\}.
$$
Applying Lemma \ref{10x} with $f$ replaced by $\varphi_j*f$ for any $j\in\mathbb Z$,
we conclude that, for any $j\in\mathbb{Z}$ and $x,y\in\mathbb{R}^n$,
\begin{equation}\label{10}
\left(\varphi_j*f\right)(x)
=\sum_{R\in\mathscr{Q}_j}2^{-jn}\left(\varphi_j*f\right)(x_R+y)\gamma_j(x-x_R-y),
\end{equation}
where $\gamma\in\mathcal{S}$ is the same as in Lemma \ref{10x}.
\end{remark}

Next, we establish the relations between
$\|\vec f\|_{\dot A^{s,\tau}_{p,q}(\mathbb{A},\varphi)}$
and $\|\sup_{\mathbb{A},\varphi}(\vec f)\|_{\dot a^{s,\tau}_{p,q}}$.

\begin{lemma}\label{56}
Let $s\in\mathbb{R}$, $\tau\in[0,\infty)$, $p\in(0,\infty)$, and $q\in(0,\infty]$.
Let $\varphi\in\mathcal{S}$ satisfy both \eqref{19} and \eqref{20}.
Let $W\in A_p$ and $\mathbb{A}:=\{A_Q\}_{Q\in\mathscr{Q}}$
be a sequence of reducing operators of order $p$ for $W$.
Then $\vec f\in\dot A^{s,\tau}_{p,q}(\mathbb{A},\varphi)$ if and only if
$\sup_{\mathbb{A},\varphi}(\vec f)\in\dot a^{s,\tau}_{p,q}$,
where $\sup_{\mathbb{A},\varphi}$ is the same as in \eqref{sup}.
Moreover, there exists a constant $C\in[1,\infty)$ such that,
for any $\vec f\in(\mathcal{S}_\infty')^m$,
$$
\left\|\vec f\right\|_{\dot A^{s,\tau}_{p,q}(\mathbb{A},\varphi)}
\leq\left\|\sup_{\mathbb{A},\varphi}\left(\vec f\right)\right\|_{\dot a^{s,\tau}_{p,q}}
\leq C\left\|\vec f\right\|_{\dot A^{s,\tau}_{p,q}(\mathbb{A},\varphi)}.
$$
\end{lemma}

\begin{proof}
The first inequality in the claim is immediate from the definition of $\sup_{\mathbb{A},\varphi} (\vec f) $, so it remains to show the second inequality.

From \eqref{10}, we deduce that, for any $j\in\mathbb{Z}$ and $x,y\in\mathbb{R}^n$,
\begin{equation}\label{50}
\left(\varphi_j*\vec f\right)(x)
=\sum_{R\in\mathscr{Q}_j}2^{-jn}\left(\varphi_j*\vec f\right)(x_R+y)\gamma_j(x-x_R-y),
\end{equation}
where $\gamma\in\mathcal{S}$ is the same as in Lemma \ref{10x}.
We fix $r\in(0,\min\{p,q,1\})$ and $M\in(\Delta+\frac{n}{r},\infty)$, where $\Delta$ is such that $W\in A_p$ has $A_p$-dimensions $(d,\widetilde d,\Delta)$.
By \eqref{50}, Lemma \ref{famous 2}, and the fact that $\gamma\in\mathcal{S}$,
we find that, for any $j\in\mathbb{Z}$,
$Q\in\mathscr{Q}_j$, $x\in Q$, and $y\in\mathbb{R}^n$,
\begin{align*}
\left|A_Q\left(\varphi_j*\vec f\right)(x)\right|^r
&\leq\sum_{R\in\mathscr{Q}_j}\left|2^{-jn}\gamma_j(x-x_R-y)\right|^r
\left|A_Q\left(\varphi_j*\vec f\right)(x_R+y)\right|^r\\
&\lesssim\sum_{R\in\mathscr{Q}_j}\frac{1}{(1+2^j|x-x_R-y|)^{Mr}}
\left|A_Q\left(\varphi_j*\vec f\right)(x_R+y)\right|^r.
\end{align*}
Combined with Lemma \ref{33y}, this further implies that
\begin{equation}\label{145}
\left[|Q|^{-\frac12}\sup_{\mathbb{A},\varphi,Q}\left(\vec f\right)\right]^r
\lesssim\sum_{R\in\mathscr{Q}_j}\frac{1}{(1+2^j|x-x_R-y|)^{Mr}}
\left|A_Q\left(\varphi_j*\vec f\right)(x_R+y)\right|^r.
\end{equation}
Using this, the Tonelli theorem, and \eqref{23y},
we obtain, for any $j\in\mathbb{Z}$,
$Q\in\mathscr{Q}_j$, and $x\in\mathbb{R}^n$,
\begin{align*}
\left[\sup_{\mathbb{A},\varphi,Q}\left(\vec f\right)
\widetilde{\mathbf{1}}_Q(x)\right]^r
&\lesssim2^{jn}\sum_{R\in\mathscr{Q}_j}\int_R
\frac{1}{(1+2^j|x-z|)^{Mr}}
\left|A_Q\left(\varphi_j*\vec f\right)(z)\right|^r\,dz\\
&\leq2^{jn}\sum_{R\in\mathscr{Q}_j}\int_R
\frac{\|A_QA_R^{-1}\|^r}{(1+2^j|x-z|)^{Mr}}
\left|A_R \left(\varphi_j*\vec f\right)(z)\right|^r\,dz\\
&\lesssim2^{jn}\sum_{R\in\mathscr{Q}_j}\int_R
\frac{1}{(1+2^j|x-z|)^{(M-\Delta)r}}
\left|A_R \left(\varphi_j*\vec f\right)(z)\right|^r\,dz\\
&=2^{jn}\int_{\mathbb{R}^n}
\frac{1}{(1+2^j|x-z|)^{(M-\Delta)r}}
\left|A_j(z)\left(\varphi_j*\vec f\right)(z)\right|^r\,dz.
\end{align*}
For any $j\in\mathbb{Z}$, let
\begin{equation}\label{gj}
g_j:=\sum_{Q\in\mathscr{Q}_j}
\sup_{\mathbb{A},\varphi,Q}\left(\vec f\right)\widetilde{\mathbf{1}}_Q
\text{ and }
h_j:=\left|A_j\left(\varphi_j*\vec{f}\right)\right|.
\end{equation}
Thus, for any $j\in\mathbb{Z}$ and $x\in\mathbb{R}^n$, we have
\begin{equation}\label{155}
\left|g_j(x)\right|^r
\lesssim2^{jn}\int_{\mathbb{R}^n}
\frac{1}{(1+2^j|x-z|)^{(M-\Delta)r}}\left|h_j(z)\right|^r\,dz,
\end{equation}
From this and Lemma \ref{summary F} with $M$ replaced by $(M-\Delta)r\in(n,\infty)$,
we infer that
$$
\left\|\sup_{\mathbb{A},\varphi}\left(\vec f\right)\right\|_{\dot a^{s,\tau}_{p,q}}
=\left\|\left\{2^{js}g_j\right\}_{j\in\mathbb Z}\right\|_{L\dot A_{p,q}^\tau}
\lesssim\left\|\left\{2^{js}h_j\right\}_{j\in\mathbb Z}\right\|_{L\dot A_{p,q}^\tau}
=\left\|\vec f\right\|_{\dot A^{s,\tau}_{p,q}(\mathbb{A},\varphi)}.
$$
This finishes the proof of Lemma \ref{56}.
\end{proof}

Next, to give the relations between
$\dot A^{s,\tau}_{p,q}(\mathbb{A},\varphi)$
and $\dot A^{s,\tau}_{p,q}(W,\varphi)$,
we need a technical lemma.

\begin{lemma}\label{151}
Let $p\in(1,\infty)$, $\frac{1}{p}+\frac{1}{p'}=1$, $W\in A_p$,
$\{A_Q\}_{Q\in\mathscr{Q}}$ be a sequence of
reducing operators of order $p$ for $W$, $M\in(n,\infty)$,
and $\delta$ be the same as in Lemma \ref{8}(ii).
Then there exists a positive constant $C$ such that,
for any $j\in\mathbb{Z}$, $x\in\mathbb{R}^n$,
and $r\in[0,p'+\delta]$,
$$
2^{jn}\int_{\mathbb{R}^n}\frac{1}{(1+2^j|x-z|)^M}
\left\|A_j(z)W^{-\frac{1}{p}}(z)\right\|^r\,dz
\leq C,
$$
where $A_j$ is the same as in \eqref{Aj}.
\end{lemma}

\begin{proof}
By Lemmas \ref{33y} and \ref{8}(ii), $M\in(n,\infty)$, and \eqref{33z}, we conclude that,
for any $j\in\mathbb{Z}$, $x\in\mathbb{R}^n$, and $r\in[0,p'+\delta]$,
\begin{align*}
&2^{jn}\int_{\mathbb{R}^n}
\frac{1}{(1+2^j|x-z|)^M}
\left\|A_j(z)W^{-\frac{1}{p}}(z)\right\|^r\,dz\\
&\quad=\sum_{R\in\mathscr{Q}_j}\fint_R
\frac{1}{(1+2^j|x-z|)^M}
\left\|A_RW^{-\frac{1}{p}}(z)\right\|^r\,dz\\
&\quad\sim\sum_{R\in\mathscr{Q}_j}
\frac{1}{(1+2^j|x-x_R|)^M}
\fint_R\left\|A_RW^{-\frac{1}{p}}(z)\right\|^r\,dz
\lesssim\sum_{k\in\mathbb{Z}^n}\frac{1}{(1+|2^jx-k|)^M}
\sim1.
\end{align*}
This finishes the proof of Lemma \ref{151}.
\end{proof}

\begin{lemma}\label{60}
Let $s\in\mathbb{R}$, $\tau\in[0,\infty)$, $p\in(0,\infty)$, $q\in(0,\infty]$,
$\varphi\in\mathcal{S}$ satisfy both \eqref{19} and \eqref{20}, $W\in A_p$,
and $\mathbb{A}:=\{A_Q\}_{Q\in\mathscr{Q}}$ be a sequence of
reducing operators of order $p$ for $W$.
Then there exists a positive constant $C$ such that,
for any $\vec f\in(\mathcal{S}_\infty')^m$,
\begin{align}\label{63}
\left\|\vec f\right\|_{\dot A^{s,\tau}_{p,q}(\mathbb{A},\varphi)}
\leq C\left\|\vec f\right\|_{\dot A^{s,\tau}_{p,q}(W,\varphi)}.
\end{align}
\end{lemma}

\begin{proof}
To prove \eqref{63}, we need consider the following two cases on $p$.

\emph{Case 1)} $p\in(0,1]$.
In this case, by Lemma \ref{8}(i), we find that,
for any $j\in\mathbb{Z}$ and almost every $x\in\mathbb{R}^n$,
\begin{align*}
\left|A_j(x)\left(\varphi_j*\vec{f}\right)(x)\right|
&\leq\left\|A_j(x)W^{-\frac{1}{p}}(x)\right\|
\left|W^{\frac{1}{p}}(x)\left(\varphi_j*\vec{f}\right)(x)\right|\\
&\lesssim\left|W^{\frac{1}{p}}(x)\left(\varphi_j*\vec{f}\right)(x)\right|,
\end{align*}
where $A_j$ is the same as in \eqref{Aj}.
This finishes the proof of \eqref{63} in this case.

\emph{Case 2)} $p\in(1,\infty)$.
In this case, for any $j\in\mathbb Z$, let
\begin{equation}\label{gj 2}
g_j:=\left|A_j\left(\varphi_j*\vec f\right)\right|\text{ and } h_j:=\left|W^{\frac{1}{p}}\left(\varphi_j*\vec{f}\right)\right|.
\end{equation}
We fix $P\in\mathscr{Q}$, $r\in(0,\frac{\min\{p,q,1\}}{p})$,
and $M\in(\Delta+\frac{n}{r},\infty)$,
where $\Delta$ is such that $W$ has $A_p$-dimensions $(d,\widetilde d,\Delta)$.
By \eqref{145}, the Tonelli theorem, and \eqref{23y},
we obtain, for any $j\in\mathbb{Z}$,
$Q\in\mathscr{Q}_j$, and $x\in Q$,
\begin{align*}
\left|A_Q\left(\varphi_j*\vec f\right)(x)\right|^r
&\leq\left[|Q|^{-\frac12}\sup_{\mathbb{A},\varphi,Q}\left(\vec f\right)\right]^r\\
&\lesssim2^{jn}\sum_{R\in\mathscr{Q}_j}\int_R
\frac{1}{(1+2^j|x-z|)^{Mr}}
\left|A_Q\left(\varphi_j*\vec f\right)(z)\right|^r\,dz\\
&\leq2^{jn}\sum_{R\in\mathscr{Q}_j}\int_R
\frac{\|A_QA_R^{-1}\|^r\|A_RW^{-\frac{1}{p}}(z)\|^r}{(1+2^j|x-z|)^{Mr}}
\left|h_j(z)\right|^r\,dz\\
&\lesssim2^{jn}\sum_{R\in\mathscr{Q}_j}\int_R
\frac{\|A_RW^{-\frac{1}{p}}(z)\|^r}{(1+2^j|x-z|)^{(M-\Delta)r}}
\left|h_j(z)\right|^r\,dz\\
&=2^{jn}\int_{\mathbb{R}^n}
\frac{\|A_j(z)W^{-\frac{1}{p}}(z)\|^r}{(1+2^j|x-z|)^{\widetilde{M}}}
\left|h_j(z)\right|^r\,dz,
\end{align*}
where $\widetilde{M}:=(M-\Delta)r\in(n,\infty)$
and $h_j$ is the same as in \eqref{gj 2}.
Using this, H\"older's inequality,
and Lemma \ref{151} with $M$ and $r$ replaced,
respectively, by $\widetilde{M}$ and $ rp' $, we find that,
for any $j\in\mathbb{Z}$ and $x\in\mathbb{R}^n$,
\begin{align*}
\left| g_j (x)\right|^{rp}
&\lesssim\left[2^{jn}\int_{\mathbb{R}^n}
\frac{1}{(1+2^j|x-z|)^{\widetilde{M}}}
\left\|A_j(z)W^{-\frac{1}{p}}(z)\right\|^{rp'}\,dz \right]^{\frac{p}{p'}}\notag\\
&\quad\times 2^{jn}\int_{\mathbb{R}^n}
\frac{1}{(1+2^j|x-z|)^{\widetilde{M}}}
\left|h_j(z)\right|^{rp}\,dz\\
&\lesssim2^{jn}\int_{\mathbb{R}^n}
\frac{1}{(1+2^j|x-z|)^{\widetilde{M}}}
\left|h_j(z)\right|^{rp}\,dz,
\end{align*}
where $g_j$ is the same as in \eqref{gj 2}.
From this and Lemma \ref{summary F} with $M$ and $r$ replaced,
respectively, by $\widetilde{M}\in(n,\infty)$ and $rp\in(0,\min\{p,q,1\})$, we deduce that
$$
\left\|\vec f\right\|_{\dot A^{s,\tau}_{p,q}(\mathbb{A},\varphi)}
=\left\|\left\{2^{js}g_j\right\}_{j\in\mathbb Z}\right\|_{L\dot A_{p,q}^\tau}
\lesssim\left\|\left\{2^{js}h_j\right\}_{j\in\mathbb Z}\right\|_{L\dot A_{p,q}^\tau}
=\left\|\vec f\right\|_{\dot A^{s,\tau}_{p,q}(W,\varphi)}.
$$
This finishes the proof of \eqref{63} in this case and hence Lemma \ref{60}.
\end{proof}

Finally, to establish the relations between
$\|\vec f\|_{\dot A^{s,\tau}_{p,q}(W,\varphi)}$
and $\|\sup_{\mathbb{A},\varphi}(\vec f)\|_{\dot a^{s,\tau}_{p,q}}$,
we need several technical lemmas.
For any $j\in\mathbb{Z}$ and any nonnegative measurable function
$f$ on $\mathbb R^n$ or any $f\in L^1_{\mathop\mathrm{loc}}$, let
\begin{equation}\label{Ej}
E_j (f):=\sum_{Q\in\mathscr{Q}_j}
\left[\fint_Q f(x)\,dx\right]\mathbf{1}_Q.
\end{equation}

The following lemma is just \cite[Corollary 3.8]{fr21}.

\begin{lemma}\label{46x}
Let $p\in(0,\infty)$, $q\in(0,\infty]$, $W\in A_p$,
and $\{A_Q\}_{Q\in\mathscr{Q}}$ be a sequence of
reducing operators of order $p$ for $W$.
For any $j\in\mathbb{Z}$, let
\begin{equation}\label{47}
\gamma_j:=\sum_{Q\in\mathscr{Q}_j}
\left\|W^{\frac{1}{p}}A_Q^{-1}\right\|\mathbf{1}_Q.
\end{equation}
Then there exists a positive constant $C$ such that,
for any sequence $\{f_j\}_{j\in\mathbb{Z}}$ of nonnegative measurable functions
on $\mathbb R^n$ or for any $\{f_j\}_{j\in\mathbb{Z}}\subset L^1_{\mathop\mathrm{loc}}$,
$$
\left\|\left\{\gamma_jE_j\left(f_j\right)\right\}_{j\in\mathbb Z}\right\|_{L^p\ell^q}
\leq C\left\|\left\{E_j\left(f_j\right)\right\}_{j\in\mathbb Z}\right\|_{L^p\ell^q},
$$
where $E_j$ for any $j\in\mathbb Z$ is the same as in \eqref{Ej}.
\end{lemma}

Applying Lemma \ref{46x}, we can obtain the following conclusion.

\begin{corollary}\label{46}
Let $p\in(0,\infty)$, $q\in(0,\infty]$, $W\in A_p$,
$\{A_Q\}_{Q\in\mathscr{Q}}$ be a sequence of
reducing operators of order $p$ for $W$,
and $\{\gamma_j\}_{j\in\mathbb{Z}}$ the same as in \eqref{47}.
Then there exists a positive constant $C$ such that,
for any sequence $\{f_j\}_{j\in\mathbb{Z}}$ of nonnegative measurable functions
on $\mathbb R^n$ [or for any $\{f_j\}_{j\in\mathbb{Z}}
\subset L^1_{\mathop\mathrm{loc}}$]
and for any $P\in\mathscr{Q}$,
$$
\left\|\left\{\gamma_jE_j\left(f_j\right)\right\}_{j\in\mathbb Z}\right\|_{L^p\ell^q(\widehat P)}
\leq C\left\|\left\{E_j\left(f_j\right)\right\}_{j\in\mathbb Z}
\right\|_{L^p\ell^q(\widehat P)},
$$
where $E_j$ for any $j\in\mathbb Z$ is the same as in \eqref{Ej}.
\end{corollary}

\begin{proof}
Fix $P\in\mathscr{Q}$.
For any $j\in\mathbb{Z}$, let
$$
g_j:=
\begin{cases}
\mathbf{1}_P f_j&\text{if }j\geq j_P,\\
0&\text{otherwise}.
\end{cases}
$$
Then, by Lemma \ref{46x}, we conclude that
\begin{align*}
\left\|\left(\sum_{j=j_P}^\infty
\left[\gamma_jE_j\left( f_j \right)\right]^q\right)^{\frac{1}{q}}\right\|_{L^p(P)}
&=\left\|\left(\sum_{j\in\mathbb{Z}}
\left[\gamma_jE_j\left(g_j\right)\right]^q\right)^{\frac{1}{q}}\right\|_{L^p}\\
&\lesssim\left\|\left(\sum_{j\in\mathbb{Z}}
\left[E_j\left(g_j\right)\right]^q\right)^{\frac{1}{q}}\right\|_{L^p}
=\left\|\left(\sum_{j=j_P}^\infty
\left[E_j\left(f_j\right)\right]^q\right)^{\frac{1}{q}}\right\|_{L^p(P)}.
\end{align*}
This finishes the proof of Corollary \ref{46}.
\end{proof}

\begin{remark}\label{46y}
In Corollary \ref{46}, if, for any $i\in\mathbb{Z}$, let
$$
f_i:=
\begin{cases}
g&\text{if }i=j,\\
0&\text{if }i\neq j,
\end{cases}
$$
where $g$ is a nonnegative measurable function on $\mathbb R^n$
or $g\in L^1_{\mathop\mathrm{loc}}$,
then, for any $j\in\mathbb{Z}$ and $P\in\mathscr{Q}$,
$$
\left\|\gamma_jE_j(g)\right\|_{L^p(P)}
\leq C\left\|E_j(g)\right\|_{L^p(P)},
$$
where $C$ is a positive constant independent of $g$, $j$, and $P$.
\end{remark}

\begin{lemma}\label{69}
Let $s\in\mathbb{R}$, $\tau\in[0,\infty)$, $p\in(0,\infty)$, $q\in(0,\infty]$,
$\varphi\in\mathcal{S}$ satisfy both \eqref{19} and \eqref{20}, $W\in A_p $,
and $\mathbb{A}:=\{A_Q\}_{Q\in\mathscr{Q}}$ be a sequence of
reducing operators of order $p$ for $W$.
Then there exists a positive constant $C$ such that,
for any $\vec f\in(\mathcal{S}_\infty')^m$,
$$
\left\|\vec f\right\|_{\dot A^{s,\tau}_{p,q}(W,\varphi)}
\leq C\left\|\sup_{\mathbb{A},\varphi}\left(\vec f\right)\right\|_{\dot a^{s,\tau}_{p,q}},
$$
where $\sup_{\mathbb{A},\varphi}$ is the same as in \eqref{sup}.
\end{lemma}

\begin{proof}
We first show that, for any $\vec f\in(\mathcal{S}_\infty')^m$,
\begin{equation}\label{70x}
\left\|\vec f\right\|_{\dot B^{s,\tau}_{p,q}(W,\varphi)}
\lesssim\left\|\sup_{\mathbb{A},\varphi}
\left(\vec f\right)\right\|_{\dot b^{s,\tau}_{p,q}}.
\end{equation}
For any $j\in\mathbb{Z}$, let
$$
g_j:=2^{js}\left|W^{\frac{1}{p}}\left(\varphi_j*\vec f\right)\right|
\text{ and }
f_j:=2^{js}\sum_{Q\in\mathscr{Q}_j}
\sup_{\mathbb{A},\varphi,Q}\left(\vec f\right)\widetilde{\mathbf{1}}_Q.
$$
Then, from the definitions of $g_j$ and $f_j$, we infer that, for any $j\in\mathbb Z$,
\begin{align}\label{238}
g_j\leq2^{js}\sum_{Q\in\mathscr{Q}_j}
\left\|W^{\frac{1}{p}}A_Q^{-1}\right\|
\left|A_Q\left(\varphi_j*\vec f\right)\right|\mathbf{1}_Q
\leq\gamma_jf_j
=\gamma_jE_j\left(f_j\right),
\end{align}
where $\gamma_j $ and $E_j$ are the same as, respectively, in \eqref{47} and \eqref{Ej}.
By this and Remark \ref{46y}, we find that,
for any $j\in\mathbb{Z}$ and $P\in\mathscr{Q}$,
$$
\left\| g_j \right\|_{L^p(P)}
\leq\left\|\gamma_jE_j\left(f_j\right)\right\|_{L^p(P)}
\lesssim\left\|E_j\left(f_j\right)\right\|_{L^p(P)}
=\left\|f_j\right\|_{L^p(P)},
$$
and hence \eqref{70x} holds.

Next, we prove that, for any $\vec f\in(\mathcal{S}_\infty')^m$,
\begin{equation}\label{70}
\left\|\vec f\right\|_{\dot F^{s,\tau}_{p,q}(W,\varphi)}
\lesssim\left\|\sup_{\mathbb{A},\varphi}
\left(\vec f\right)\right\|_{\dot f^{s,\tau}_{p,q}}.
\end{equation}
From \eqref{238} and Corollary \ref{46}, it follows that,
for any $P\in\mathscr{Q}$,
\begin{align*}
\left\|\left\{g_j\right\}_{j\in\mathbb Z}\right\|_{L^p\ell^q(\widehat P)}
&\leq\left\|\left\{\gamma_jE_j\left(f_j\right)\right\}_{j\in\mathbb Z}\right\|_{L^p\ell^q(\widehat P)}\\
&\lesssim\left\|\left\{E_j\left(f_j\right)\right\}_{j\in\mathbb Z}\right\|_{L^p\ell^q(\widehat P)}
=\left\|\left\{f_j\right\}_{j\in\mathbb Z}\right\|_{L^p\ell^q(\widehat P)},
\end{align*}
and hence
$$
\left\|\vec f\right\|_{\dot F^{s,\tau}_{p,q}(W,\varphi)}
\lesssim\left\|\sup_{\mathbb{A},\varphi}\left(\vec f\right)\right\|_{\dot f^{s,\tau}_{p,q}}.
$$
This finishes the proof of \eqref{70} and hence Lemma \ref{69}.
\end{proof}

\begin{proof}[Proof of Theorem \ref{11}]
Lemmas \ref{60} and \ref{69} give, respectively, the two inequalities
\begin{equation*}
\left\|\vec f\right\|_{\dot A^{s,\tau}_{p,q}(\mathbb A,\varphi)}
\lesssim \left\|\vec f\right\|_{\dot A^{s,\tau}_{p,q}(W,\varphi)}
\lesssim\left\|\sup_{\mathbb A,\varphi}\left(\vec f\right)\right\|_{\dot a^{s,\tau}_{p,q}},
\end{equation*}
while Lemma \ref{56} states the equivalence of the left and the right sides.
These give the equivalence of all three norms above
and hence finish the proof of Theorem \ref{11}.
\end{proof}

\subsection{Sequence Spaces: Definitions and Basic Properties}\label{W and AQ 2}

First, we introduce matrix-weighted
Besov-type and Triebel--Lizorkin-type sequence spaces.

\begin{definition}
Let $s\in\mathbb{R}$, $\tau\in[0,\infty)$, $p\in(0,\infty)$, $q\in(0,\infty]$, and $W\in A_p$.
The \emph{homogeneous matrix-weighted Besov-type sequence space} $\dot b^{s,\tau}_{p,q}(W)$
and the \emph{homogeneous matrix-weighted Triebel--Lizorkin-type sequence space}
$\dot f^{s,\tau}_{p,q}(W)$
are defined to be the sets of all sequences
$\vec t:=\{\vec t_Q\}_{Q\in\mathscr{Q}}\subset\mathbb{C}^m$ such that
$$
\left\|\vec{t}\right\|_{\dot a^{s,\tau}_{p,q}(W)}
:=\left\|\left\{2^{js}\left|W^{\frac{1}{p}}\vec t_j
\right|\right\}_{j\in\mathbb Z}\right\|_{L\dot A_{p,q}^\tau}<\infty,
$$
where, for any $j\in\mathbb{Z}$,
\begin{equation}\label{vec tj}
\vec t_j:=\sum_{Q\in\mathscr{Q}_j}\vec{t}_Q\widetilde{\mathbf{1}}_Q
\end{equation}
and $\|\cdot\|_{L\dot A_{p,q}^\tau}$ is the same as in \eqref{LApq}.
\end{definition}

Above and in what follows, we use $\dot a^{s,\tau}_{p,q}(W)$ to denote
either $\dot b^{s,\tau}_{p,q}(W)$ or $\dot f^{s,\tau}_{p,q}(W)$.
Applying an argument similar to that used in the proof of Proposition \ref{39},
we obtain the following relations between $\dot b^{s,\tau}_{p,q}(W)$
and $\dot f^{s,\tau}_{p,q}(W)$; we omit the details.

\begin{proposition}
Let $s\in\mathbb{R}$, $\tau\in[0,\infty)$,
$p\in(0,\infty)$, $q\in(0,\infty]$, and $W\in A_p$. Then
$$
\dot b_{p,p\wedge q}^{s,\tau}(W)
\subset\dot f^{s,\tau}_{p,q}(W)
\subset\dot b_{p,p\vee q}^{s,\tau}(W).
$$
Moreover, for any $\vec t:=\{\vec t_Q\}_{Q\in\mathscr{Q}}\subset\mathbb{C}^m$,
$$
\left\|\vec t\right\|_{\dot b^{s,\tau}_{p,p\vee q}(W)}
\leq\left\|\vec t\right\|_{\dot f^{s,\tau}_{p,q}(W)}
\leq\left\|\vec t\right\|_{\dot b_{p,p\wedge q}^{s,\tau}(W)}.
$$
\end{proposition}

Now, we introduce the averaging matrix-weighted Besov-type
and Triebel--Lizorkin-type sequence spaces.

\begin{definition}
Let $s\in\mathbb{R}$, $\tau\in[0,\infty)$, $p\in(0,\infty)$, $q\in(0,\infty]$, $W\in A_p$,
and $\mathbb{A}:=\{A_Q\}_{Q\in\mathscr{Q}}$ be a sequence of
reducing operators of order $p$ for $W$.
The \emph{homogeneous averaging matrix-weighted Besov-type sequence space}
$\dot{b}^{s,\tau}_{p,q}(\mathbb{A})$
and the \emph{homogeneous averaging matrix-weighted Triebel--Lizorkin-type sequence space}
$\dot{f}^{s,\tau}_{p,q}(\mathbb{A})$
are defined to be the sets of all sequences
$\vec t:=\{\vec t_Q\}_{Q\in\mathscr{Q}}\subset\mathbb{C}^m$ such that
$$
\left\|\vec{t}\right\|_{\dot{a}^{s,\tau}_{p,q}(\mathbb{A})}
:=\left\|\left\{2^{js}\left|A_j\vec t_j\right|\right\}_{j\in\mathbb Z}\right\|_{L\dot A_{p,q}^\tau}<\infty,
$$
where $A_j$, $\vec t_j$, and $\|\cdot\|_{L\dot A_{p,q}^\tau}$
are the same as, respectively, in \eqref{Aj}, \eqref{vec tj}, and \eqref{LApq}.
\end{definition}

Above and in what follows, we use $\dot{a}^{s,\tau}_{p,q}(\mathbb{A})$ to denote
either $\dot{b}^{s,\tau}_{p,q}(\mathbb{A})$ or $\dot{f}^{s,\tau}_{p,q}(\mathbb{A})$.
By \eqref{equ_reduce}, we find that
$\dot{a}^{s,\tau}_{p,q}(\mathbb{A})$ is independent of the choice of $\mathbb{A}$.
The following theorem is the main result of this subsection.

\begin{theorem}\label{37}
Let $s\in\mathbb{R}$, $\tau\in[0,\infty)$,
$p\in(0,\infty)$, $q\in(0,\infty]$, $W\in A_p$,
and $\mathbb{A}:=\{A_Q\}_{Q\in\mathscr{Q}}$ be a sequence of
reducing operators of order $p$ for $W$.
Then, for any $\vec t:=\{\vec t_Q\}_{Q\in\mathscr{Q}}\subset\mathbb{C}^m$,
$\|\vec{t}\|_{\dot a^{s,\tau}_{p,q}(W)}
\sim\|\vec{t}\|_{\dot{a}^{s,\tau}_{p,q}(\mathbb{A})},$
where the positive equivalence constants are independent of $\vec t$.
\end{theorem}

To show Theorem \ref{37}, we need the following lemma
which is a part of \cite[Proposition 2.4]{yyz13}.

\begin{lemma}\label{44}
Let $s\in\mathbb{R}$, $\tau\in[0,\infty)$,
$p\in(0,\infty)$, $q\in(0,\infty]$, and $\delta\in(0,1)$.
Suppose that, for any $Q\in\mathscr{Q}$, $E_Q\subset Q$ is a measurable set
with $|E_Q|\geq\delta|Q|$.
Then, for any sequence $t:=\{t_Q\}_{Q\in\mathscr{Q}}\subset\mathbb{C}$,
$$
\|t\|_{\dot f^{s,\tau}_{p,q}}
\sim\left\|\left\{2^{j(s+\frac{n}{2})}
\sum_{Q\in\mathscr{Q}_j}t_Q\mathbf{1}_{E_Q}
\right\}_{j\in\mathbb Z}\right\|_{L\dot F_{p,q}^\tau},
$$
where the positive equivalence constants are independent of $t$.
\end{lemma}

Next, we prove Theorem \ref{37}.

\begin{proof}[Proof of Theorem \ref{37}]
Using \eqref{equ_reduce}, we obtain, for any
$\vec t:=\{\vec t_Q\}_{Q\in\mathscr{Q}}\subset\mathbb{C}^m$,
$\|\vec{t}\|_{\dot b^{s,\tau}_{p,q}(W)}
\sim\|\vec{t}\|_{\dot{b}^{s,\tau}_{p,q}(\mathbb{A})}.$
It remains to show that, for any
$\vec t:=\{\vec t_Q\}_{Q\in\mathscr{Q}}\subset\mathbb{C}^m$,
$\|\vec{t}\|_{\dot f^{s,\tau}_{p,q}(W)}
\sim\|\vec{t}\|_{\dot{f}^{s,\tau}_{p,q}(\mathbb{A})}.$

We first prove that, for any
$\vec t:=\{\vec t_Q\}_{Q\in\mathscr{Q}}\subset\mathbb{C}^m$,
\begin{equation}\label{45}
\left\|\vec{t}\right\|_{\dot f^{s,\tau}_{p,q}(W)}
\lesssim\left\|\vec{t}\right\|_{\dot{f}^{s,\tau}_{p,q}(\mathbb{A})}.
\end{equation}
For any $j\in\mathbb{Z}$, let
$$
f_j:=\sum_{Q\in\mathscr{Q}_j}|Q|^{-\frac{s}{n}}
\left|A_Q\vec t_Q \right| \widetilde{\mathbf{1}}_Q\
\mathrm{and}\
g_j:=\sum_{Q\in\mathscr{Q}_j}|Q|^{-\frac{s}{n}}
\left|W^{\frac{1}{p}}\vec t_Q \right|\widetilde{\mathbf{1}}_Q.
$$
Then, for any $j\in\mathbb Z$,
$$
g_j
\leq\sum_{Q\in\mathscr{Q}_j}|Q|^{-\frac{s}{n}}
\left\|W^{\frac{1}{p}}A_Q^{-1}\right\|
\left|A_Q\vec t_Q\right|\widetilde{\mathbf{1}}_Q
=\gamma_jf_j
=\gamma_jE_j\left(f_j\right),
$$
where $\gamma_j $ and $E_j$ are the same as, respectively, in \eqref{47} and \eqref{Ej}.
From this, Corollary \ref{46},
and the fact that $f_j$ is a constant on $Q\in\mathscr{Q}_j$, we deduce that
\begin{align*}
\left\|\vec{t}\right\|_{\dot f^{s,\tau}_{p,q}(W)}
&=\|\{g_j\}_{j\in\mathbb Z}\|_{L\dot F_{p,q}^\tau}
\leq\left\|\left\{\gamma_jE_j\left(f_j\right)\right\}_{j\in\mathbb Z}
\right\|_{L\dot F_{p,q}^\tau}\\
&\lesssim\left\|\left\{E_j\left(f_j\right)\right\}_{j\in\mathbb Z}
\right\|_{L\dot F_{p,q}^\tau}
=\|\{f_j\}_{j\in\mathbb Z}\|_{L\dot F_{p,q}^\tau}
=\left\|\vec t\right\|_{\dot{f}^{s,\tau}_{p,q}(\mathbb{A})},
\end{align*}
where $\|\cdot\|_{L\dot F_{p,q}^\tau}$ is the same as in \eqref{LApq}.
This finishes the proof of \eqref{45}.

Now, we show that, for any
$\vec t:=\{\vec t_Q\}_{Q\in\mathscr{Q}}\subset\mathbb{C}^m$,
\begin{equation}\label{43}
\left\|\vec{t}\right\|_{\dot{f}^{s,\tau}_{p,q}(\mathbb{A})}
\lesssim\left\|\vec{t}\right\|_{\dot f^{s,\tau}_{p,q}(W)}.
\end{equation}
To this end, we consider the following two cases on $p$.

\emph{Case 1)} $p\in(0,1]$.
In this case, by Lemma \ref{8}(i), we conclude that,
for any $Q\in\mathscr{Q}$ and almost every $x\in Q$,
$$
\left|A_Q\vec t_Q\right|
\leq\left\|A_QW^{-\frac{1}{p}}(x)\right\|\left|W^{\frac{1}{p}}(x)\vec t_Q\right|
\lesssim\left|W^{\frac{1}{p}}(x)\vec t_Q\right|.
$$
This finishes the proof of \eqref{43} in this case.

\emph{Case 2)} $p\in(1,\infty)$.
In this case, applying Lemma \ref{8}(ii) with $r$ replaced by $1$,
we find that there exists a positive constant $C$ such that
\begin{equation}\label{140}
\sup_{Q\in\mathscr{Q}}
\fint_Q\left\|A_QW^{-\frac{1}{p}}(x)\right\|\,dx
\leq C.
\end{equation}
For any $Q\in\mathscr{Q}$, let
$$
E_Q:=\left\{x\in Q:\ \left\|A_QW^{-\frac{1}{p}}(x)\right\|\leq2C\right\}.
$$
From this, Chebyshev's inequality, and \eqref{140},
we infer that, for any $Q\in\mathscr{Q}$,
\begin{align*}
|Q\setminus E_Q|
\leq\frac{1}{2C}\int_{Q\setminus E_Q}\left\|A_QW^{-\frac{1}{p}}(x)\right\|\,dx
\leq\frac{|Q|}{2C}\fint_{Q}\left\|A_QW^{-\frac{1}{p}}(x)\right\|\,dx
\leq\frac12|Q|
\end{align*}
and hence $|E_Q|\geq\frac12|Q|$.
This, together with both Lemma \ref{44} with $\delta$ replaced by $\frac12$
and the definition of $E_Q$, further implies that
\begin{align*}
\left\|\vec{t}\right\|_{\dot{f}^{s,\tau}_{p,q}(\mathbb{A})}
&=\left\|\left\{\left|A_Q\vec t_Q\right|
\right\}_{Q\in\mathscr{Q}}\right\|_{\dot f^{s,\tau}_{p,q}}
\sim\left\|\left\{2^{js}\sum_{Q\in\mathscr{Q}_j}
\left|A_Q\vec t_Q\right|\widetilde{\mathbf{1}}_{E_Q}\right\}_{j\in\mathbb Z}
\right\|_{L\dot F_{p,q}^\tau}\\
&\lesssim\left\|\left\{2^{js}\sum_{Q\in\mathscr{Q}_j}
\left|W^{\frac{1}{p}}\vec t_Q\right|\widetilde{\mathbf{1}}_{E_Q}\right\}_{j\in\mathbb Z}
\right\|_{L\dot F_{p,q}^\tau}
\leq\left\|\left\{2^{js}\left|W^{\frac{1}{p}}\vec t_j\right|\right\}_{j\in\mathbb Z}
\right\|_{L\dot F_{p,q}^\tau}
=\left\|\vec{t}\right\|_{\dot f^{s,\tau}_{p,q}(W)},
\end{align*}
where $\vec t_j$ for any $j\in\mathbb Z$ is the same as in \eqref{vec tj}.
This finishes the proof of \eqref{43} in this case and hence Theorem \ref{37}.
\end{proof}

\subsection{The $\varphi$-Transform Characterization}
\label{phi-transform}

In this subsection, we establish the $\varphi$-transform characterization of $\dot A^{s,\tau}_{p,q}(W)$.
Recall that the \emph{$\varphi$-transform} is defined to be the map taking each
$\vec f\in(\mathcal{S}_\infty')^m$ to the sequence
$S_\varphi\vec{f}:=\{(S_\varphi\vec{f})_Q\}_{Q\in\mathscr{Q}}$,
where $(S_\varphi\vec{f})_Q:=\langle\vec{f},\varphi_Q\rangle$ for any $Q\in\mathscr{Q}$;
the \emph{inverse $\varphi$-transform} is defined to be the map taking a sequence
$\vec t:=\{\vec t_Q\}_{Q\in\mathscr{Q}}\subset\mathbb{C}^m$
to $T_\psi\vec t:=\sum_{Q\in\mathscr{Q}}\vec t_Q\psi_Q$ in $(\mathcal{S}_\infty')^m$
(see, for instance, \cite{fj85, fj88}).
Then we have the following result.

\begin{theorem}\label{phi}
Let $s\in\mathbb{R}$, $\tau\in[0,\infty)$, $p\in(0,\infty)$, and $q\in(0,\infty]$.
Let $\varphi,\psi\in\mathcal{S}$ satisfy \eqref{19} and \eqref{20}, let $\widetilde\varphi(x):=\overline{\varphi(-x)}$ for any $x\in\mathbb{R}^n$, and let $W\in A_p$.
Then the operators
\begin{equation*}
S_\varphi:\ \dot A^{s,\tau}_{p,q}(W,\widetilde{\varphi})\to\dot a^{s,\tau}_{p,q}(W)
\text{ and }
T_\psi:\ \dot a^{s,\tau}_{p,q}(W)\to\dot A^{s,\tau}_{p,q}(W,\varphi)
\end{equation*}
are bounded.
Furthermore, if $\varphi$ and $\psi$ satisfy \eqref{21}, then
$T_\psi\circ S_\varphi$ is the identity on $\dot A^{s,\tau}_{p,q}(W, \widetilde{\varphi})$.
\end{theorem}

To prove this theorem, we need several technical lemmas.
We first recall the following Calder\'on reproducing formulae
which are \cite[Lemma 2.1]{yy10}.

\begin{lemma}\label{7}
Let $\varphi,\psi\in\mathcal{S}$ satisfy \eqref{21} and
both $\overline{\operatorname{supp}\widehat{\varphi}}$ and $\overline{\operatorname{supp}\widehat{\psi}}$
are compact and bounded away from the origin.
Then, for any $f\in\mathcal{S}_\infty$,
\begin{equation}\label{7x}
f=\sum_{j\in\mathbb{Z}}2^{-jn}\sum_{k\in\mathbb{Z}^n}
\left(\widetilde{\varphi}_j*f\right)\left(2^{-j}k\right)
\psi_j\left(\cdot-2^{-j}k\right)
=\sum_{Q\in\mathscr{Q}}\left\langle f,\varphi_Q\right\rangle\psi_Q
\end{equation}
in $\mathcal{S}_\infty $, where $\widetilde{\varphi}(x):=\overline{\varphi(-x)}$
for any $x\in\mathbb{R}^n$.
Moreover, for any $f\in\mathcal{S}_\infty'$,
\eqref{7x} also converges in $\mathcal{S}_\infty'$.
\end{lemma}

The following lemma is \cite[Lemma 2.2]{yy08}.

\begin{lemma}\label{59}
Let $\varphi,\psi\in\mathcal{S}_\infty$.
For any $M\in\mathbb{N}$, there exists a positive constant $C$,
depending only on $M$ and $n$, such that,
for any $j,i\in\mathbb{Z}$ and $x\in\mathbb{R}^n$,
$$
\left|\left(\varphi_j*\psi_i\right)(x)\right|
\leq C\|\varphi\|_{S_{M+1}}\|\psi\|_{S_{M+1}}2^{-|i-j|M}
\frac{2^{-(i\wedge j)M}}{[2^{-(i\wedge j)}+|x|]^{n+M}},
$$
where, for any $\phi\in\mathcal{S}$,
\begin{equation}\label{SM}
\|\phi\|_{S_M}
:=\sup_{\gamma\in\mathbb{Z}_+^n,\,|\gamma|\leq M}
\sup_{x\in\mathbb{R}^n}|\partial^\gamma\phi(x)|(1+|x|)^{n+M+|\gamma|}.
\end{equation}
\end{lemma}

As a corollary of Lemma \ref{59}, we obtain the following estimate.

\begin{corollary}\label{59x}
Let $\varphi,\psi\in\mathcal{S}_\infty$.
Then, for any $M\in\mathbb{N}$ and $Q,R\in\mathscr{Q}$,
\begin{align*}
\left|\left\langle\varphi_Q,\psi_R\right\rangle\right|
&\leq C\|\varphi\|_{S_{M+1}}\|\psi\|_{S_{M+1}}
\left[\min\left\{\frac{\ell(R)}{\ell(Q)},\frac{\ell(Q)}{\ell(R)}\right\}\right]^{M+\frac{n}{2}}
\left[1+\frac{|x_Q-x_R|}{\ell(Q)\vee\ell(R)}\right]^{-(n+M)},
\end{align*}
where $C$ is the same as in Lemma \ref{59}.
\end{corollary}

\begin{proof}
Let $ M\in\mathbb{N}$. Then, by a change of variables, we conclude that,
for any $j,i\in\mathbb{Z}$, $Q\in\mathscr{Q}_j$, and $R\in\mathscr{Q}_i$,
\begin{align*}
\left\langle\varphi_Q,\psi_R\right\rangle
&=|Q|^{\frac12}|R|^{\frac12}\int_{\mathbb{R}^n}\varphi_j(x-x_Q)\overline{\psi_i(x-x_R)}\,dx\\
&=|Q|^{\frac12}|R|^{\frac12}\int_{\mathbb{R}^n}\varphi_j(x)\widetilde\psi_i(x_R-x_Q-x)\,dx
=|Q|^{\frac12}|R|^{\frac12}\left(\varphi_j*\widetilde\psi_i\right)(x_R-x_Q),
\end{align*}
which, combined with Lemma \ref{59}, further implies that
\begin{align*}
\left|\left\langle\varphi_Q, \psi_R \right\rangle\right|
&\lesssim|Q|^{\frac12}|R|^{\frac12}\|\varphi\|_{S_{M+1}}
\left\|\widetilde\psi\right\|_{S_{M+1}}2^{-|i-j|M}
\frac{2^{-(i\wedge j)M}}{[2^{-(i\wedge j)}+|x_Q-x_R|]^{n+M}}\\
&=|Q|^{\frac12}|R|^{\frac12}\|\varphi\|_{S_{M+1}}\|\psi\|_{S_{M+1}}
2^{-|i-j|M}2^{(i\wedge j)n}
\left[1+\frac{|x_Q-x_R|}{\ell(Q)\vee\ell(R)}\right]^{-(n+M)}\\
&=\|\varphi\|_{S_{M+1}}\|\psi\|_{S_{M+1}}
\left[\min\left\{\frac{\ell(R)}{\ell(Q)},\frac{\ell(Q)}{\ell(R)}\right\}\right]^{M+\frac{n}{2}}
\left[1+\frac{|x_Q-x_R|}{\ell(Q)\vee\ell(R)}\right]^{-(n+M)}.
\end{align*}
This finishes the proof of Corollary \ref{59x}.
\end{proof}

The following lemma shows that $T_\psi$ is well defined for any $\vec t\in\dot a^{s,\tau}_{p,q}(W)$.

\begin{lemma}\label{3}
Let $s\in\mathbb{R}$, $\tau\in[0,\infty)$, $p\in(0,\infty)$, $q\in(0,\infty]$,
and $W\in A_p$ have $A_p$-dimensions $(d,\widetilde d,\Delta)$.
Then, for any $\vec t:=\{\vec t_Q\}_{Q\in\mathscr{Q}}\in\dot a^{s,\tau}_{p,q}(W)$ and $\psi\in \mathcal{S}_\infty$,
$\sum_{Q\in\mathscr{Q}}\vec t_Q\psi_Q$ converges in $(\mathcal{S}_\infty')^m$.
Moreover, if $M\in\mathbb{Z}_+$ satisfies
\begin{equation}\label{239}
M>\max\left\{\frac{n}{p}+\frac{\widetilde d}{p'}-(s+n\tau),s+n\tau-\frac{n-d}{p},\Delta\right\},
\end{equation}
then there exists a positive constant $C$ such that,
for any $\vec t\in\dot a^{s,\tau}_{p,q}(W)$ and $\psi,\phi\in\mathcal{S}_\infty$,
\begin{equation*}
\sum_{Q\in\mathscr{Q}}\left|\vec{t}_Q\right||\langle\psi_Q,\phi\rangle|
\leq C\left\|\vec t\right\|_{\dot a^{s,\tau}_{p,q}(W)}\|\psi\|_{S_{M+1}}\|\phi\|_{S_{M+1}},
\end{equation*}
where $\|\cdot\|_{S_M}$ is the same as in \eqref{SM}.
\end{lemma}

\begin{proof}
From Theorem \ref{37}, we deduce that, for each $Q\in\mathscr{Q}$,
\begin{align*}
\left|\vec{t}_Q \right|
&\leq\left\|A_Q^{-1}\right\|\left|A_Q\vec{t}_Q\right|
\leq\left\|A_Q^{-1}\right\| |Q|^{\frac{s}{n}+\frac12-\frac{1}{p}+\tau}
\left\|\vec t\right\|_{\dot{a}^{s,\tau}_{p,q}(\mathbb{A})}
\sim\left\|A_Q^{-1}\right\| |Q|^{\frac{s}{n}+\frac12-\frac{1}{p}+\tau}
\left\|\vec t\right\|_{\dot a^{s,\tau}_{p,q}(W)}
\end{align*}
and hence, for any $\phi\in\mathcal{S}_\infty$,
\begin{equation}\label{1}
\sum_{Q\in\mathscr{Q}}\left|\vec{t}_Q\right||\langle\psi_Q,\phi\rangle|
\lesssim\left\|\vec t\right\|_{\dot a^{s,\tau}_{p,q}(W)}
\sum_{Q\in\mathscr{Q}}|Q|^{\frac{s}{n}+\frac12-\frac{1}{p}+\tau}
\left\|A_Q^{-1}\right\||\langle\psi_Q,\phi\rangle|.
\end{equation}
Next, we estimate $\|A_Q^{-1}\|$ and $|\langle\psi_Q,\phi\rangle|$, respectively.
By Corollary \ref{237}, we conclude that, for any $Q\in\mathscr{Q}$,
\begin{align}\label{2}
\left\|A_Q^{-1}\right\|
\leq\left\|A_{Q_{0,\mathbf{0}}}^{-1}\right\|
\left\|A_{Q_{0,\mathbf{0}}}A_Q^{-1}\right\|
\lesssim\max\left\{[\ell(Q)]^{\frac dp},[\ell(Q)]^{-\frac{\widetilde d}{p'}}\right\}
\left[1+\frac{|x_Q|}{1\vee\ell(Q)}\right]^\Delta.
\end{align}
Let $M\in\mathbb{N}$ satisfy \eqref{239}.
From Corollary \ref{59x}, we infer that,
for any $\varphi,\psi\in\mathcal{S}_\infty$ and $Q\in\mathscr{Q}$,
\begin{align}\label{24}
|\langle\psi_Q,\phi\rangle|
&=|\langle\psi_Q,\phi_{Q_{0,\mathbf{0}}}\rangle|\\
&\lesssim\|\psi\|_{S_{M+1}}\|\phi\|_{S_{M+1}}
\left[\min\left\{[\ell(Q)]^{-1},\ell(Q)\right\}\right]^{M+\frac{n}{2}}
\left[1+\frac{|x_Q|}{\ell(Q)\vee 1}\right]^{-(n+M)},\notag
\end{align}
where the implicit positive constant depends only on $M$ and $n$.
This, together with \eqref{1}, \eqref{2}, \eqref{239}, and Lemma \ref{253x}, further implies that
\begin{align*}
&\sum_{Q\in\mathscr{Q}}\left|\vec{t}_Q\right||\langle\psi_Q,\phi\rangle|\\
&\quad\lesssim\left\|\vec t\right\|_{\dot a^{s,\tau}_{p,q}(W)}\|\psi\|_{S_{M+1}}\|\phi\|_{S_{M+1}}
\sum_{Q\in\mathscr{Q}} |Q|^{\frac{s}{n}+\frac12-\frac{1}{p}+\tau}\\
&\qquad\times\min\left\{[\ell(Q)]^{-(M-\frac{d}{p}+\frac{n}{2})},[\ell(Q)]^{M+\frac{n}{2}-\frac{\widetilde d}{p'}}\right\}
\left[1+\frac{|x_Q|}{\ell(Q)\vee 1}\right]^{-(n+M-\Delta)}\\
&\quad=\left\|\vec t\right\|_{\dot a^{s,\tau}_{p,q}(W)}\|\phi\|_{S_{M+1}}\|\phi\|_{S_{M+1}}\\
&\qquad\times\Bigg[\sum_{j=0}^\infty2^{-j(s+n-\frac{n}{p}+n\tau+M-\frac{\widetilde d}{p'})}
\sum_{k\in\mathbb{Z}^n}\left(1+2^{-j}|k|\right)^{-(n+M-\Delta)}\\
&\qquad+\sum_{j=-\infty}^{-1}2^{-j(s-\frac{n}{p}+n\tau-M+\frac{d}{p})}
\sum_{k\in\mathbb{Z}^n} (1+|k|)^{-(n+M-\Delta)}\Bigg]\\
&\quad\sim\left\|\vec t\right\|_{\dot a^{s,\tau}_{p,q}(W)}\|\psi\|_{S_{M+1}}\|\phi\|_{S_{M+1}}
\left[\sum_{j=0}^\infty2^{-j(s-\frac{n}{p}+n\tau+M-\frac{\widetilde{d}}{p'})}
+\sum_{j=-\infty}^{-1}2^{-j(s-\frac{n}{p}+n\tau-M+\frac{d}{p})}\right]\\
&\quad\sim\left\|\vec t\right\|_{\dot a^{s,\tau}_{p,q}(W)}\|\psi\|_{S_{M+1}}\|\phi\|_{S_{M+1}}.
\end{align*}
This finishes the proof of Lemma \ref{3}.
\end{proof}

For any sequence $t:=\{t_Q\}_{Q\in\mathscr{Q}}\subset\mathbb{C}$,
$r\in(0,\infty]$, and $\lambda\in(0,\infty)$,
let $t_{r,\lambda}^*:=\{(t_{r,\lambda}^*)_Q\}_{Q\in\mathscr{Q}}$, where, for any $Q\in\mathscr{Q}$,
$$
\left(t_{r,\lambda}^*\right)_Q
:=\left[\sum_{R\in\mathscr{Q},\,\ell(R)=\ell(Q)}
\frac{|t_R|^r}{\{1+[\ell(R)]^{-1}|x_R-x_Q|\}^\lambda}\right]^{\frac{1}{r}}.
$$
Then we have the following conclusion.

\begin{lemma}\label{4}
Let $s\in\mathbb{R}$, $\tau\in[0,\infty)$,
$p\in(0,\infty)$, $q\in(0,\infty]$,
$\lambda\in(n,\infty)$, $W\in A_p$,
and $\{A_Q\}_{Q\in\mathscr{Q}}$ be a sequence of
reducing operators of order $p$ for $W$.
Then, for any $\vec{t}\in\dot a^{s,\tau}_{p,q}(W)$,
$$
\left\|\vec{t}\right\|_{\dot a^{s,\tau}_{p,q}(W)}
\sim\left\|\left(\left\{\left|A_Q\vec t_Q\right|\right\}_{Q\in\mathscr{Q}}
\right)_{p\wedge q,\lambda}^*\right\|_{\dot a^{s,\tau}_{p,q}},
$$
where the positive equivalence constants are independent of $\vec t$.
\end{lemma}

\begin{proof}
The unweighted version of this result, i.e., the case $W\equiv A_Q\equiv 1$, is contained in \cite[Lemma 3.3]{yy10}. We use this to obtain the matrix-weighted extension as follows.

Let $u:=\{u_Q\}_{Q\in\mathscr{Q}}$,
where, for any $Q\in\mathscr{Q}$, $u_Q:=|A_Q\vec t_Q|$.
Then, by Theorem \ref{37} and the mentioned unweighted version of the
assertion from \cite[Lemma 3.3]{yy10}, we find that
$$
\left\|\vec{t}\right\|_{\dot a^{s,\tau}_{p,q}(W)}
\sim\left\|\vec{t}\right\|_{\dot{a}^{s,\tau}_{p,q}(\mathbb{A})}
=\|u\|_{\dot a^{s,\tau}_{p,q}}
\sim\left\|u_{p\wedge q,\lambda}^*\right\|_{\dot a^{s,\tau}_{p,q}}.
$$
This finishes the proof of Lemma \ref{4}.
\end{proof}

Applying some ideas similar to those used in the proof of \cite[Theorem 2.2]{fj90},
we can prove Theorem \ref{phi}.

\begin{proof}[Proof of Theorem \ref{phi}]
We first show the boundedness of
$S_{\varphi}:\ \dot A^{s,\tau}_{p,q}(W,\widetilde{\varphi})\to\dot a^{s,\tau}_{p,q}(W)$.
For any $\vec f\in\dot A^{s,\tau}_{p,q}(W,\widetilde{\varphi})$, let
$$
\sup_{\mathbb{A},\widetilde{\varphi}}\left(\vec f\right)
:=\left\{\sup_{\mathbb{A},\widetilde{\varphi},Q}\left(\vec f\right)\right\}_{Q\in\mathscr{Q}}
$$
be the same as in \eqref{sup}.
Obviously, by the definition of $\sup_{\mathbb{A},\widetilde{\varphi},Q}(\vec f)$, we obtain,
for any $\vec f\in\dot A^{s,\tau}_{p,q}(W,\widetilde{\varphi})$ and $Q\in\mathscr{Q}$,
\begin{align*}
\left|A_Q\left(S_{\varphi}\vec f\right)_Q\right|
\leq\left|A_Q\left\langle\vec f,\varphi_Q\right\rangle\right|
=|Q|^{\frac12}\left|A_Q\left(\widetilde{\varphi}_{j_Q}*\vec f\right)(x_Q)\right|
\leq\sup_{\mathbb{A},\widetilde{\varphi},Q}\left(\vec f\right),
\end{align*}
which, combined with Theorems \ref{37} and \ref{11}, further implies that
$$
\left\|S_{\varphi}\vec f\right\|_{\dot a^{s,\tau}_{p,q}(W)}
\sim\left\|S_{\varphi}\vec f\right\|_{\dot{a}^{s,\tau}_{p,q}(\mathbb{A})}
\leq\left\|\sup_{\mathbb{A},\widetilde{\varphi}}\left(\vec f\right)\right\|_{\dot a^{s,\tau}_{p,q}}
\sim\left\|\vec f\right\|_{\dot A^{s,\tau}_{p,q}(W, \widetilde{\varphi})}.
$$
This finishes the proof of the boundedness of $ S_{\varphi}$.

Now, we prove the boundedness of
$T_\psi:\ {\dot a^{s,\tau}_{p,q}(W)}\to\dot A^{s,\tau}_{p,q}(W,\varphi)$.
Let $W$ have $A_p$-dimensions $(d,\widetilde d,\Delta)$ and
let $\vec t:=\{\vec t_Q\}_{Q\in\mathscr{Q}}\in\dot a^{s,\tau}_{p,q}(W)$.
Using Lemma \ref{3} and the fact that $\psi\in\mathcal{S}_\infty$, we conclude that $T_\psi$ is well defined.
Thus, by \eqref{19}, we find that,
for any $j\in\mathbb{Z}$, $Q\in\mathscr{Q}_j$, and $x\in Q$,
\begin{align}\label{71}
\left|A_Q\left[\varphi_j*\left(T_\psi\vec t\right)\right](x)\right|
&=\left|\sum_{i=j-1}^{j+1}\sum_{R\in\mathscr{Q}_i}
A_Q\vec t_R\left(\varphi_j*\psi_R\right)(x)\right| \\
&\leq\sum_{i=j-1}^{j+1}\sum_{R\in\mathscr{Q}_i}
\left\|A_QA_R^{-1}\right\|\left|A_R\vec t_R\right|
\left|\left(\varphi_j*\psi_R\right)(x)\right|.\notag
\end{align}
From Corollary \ref{237}, we deduce that,
for any $j\in\mathbb{Z}$, $i\in\{j-1,j,j+1\}$,
$Q\in\mathscr{Q}_j$, and $R\in\mathscr{Q}_i$,
\begin{align}\label{73}
\left\|A_QA_R^{-1}\right\|
&\lesssim\max\left\{\left[\frac{\ell(R)}{\ell(Q)}\right]^{\frac dp},
\left[\frac{\ell(R)}{\ell(Q)}\right]^{\frac{\widetilde d}{p'}}\right\}
\left[1+\frac{|x_Q-x_R|}{\max\{\ell(R),\ell(Q)\}}\right]^\Delta \\
&\sim\left\{1+[\ell(R)]^{-1}|x_Q-x_R|\right\}^\Delta.\notag
\end{align}
Let $M\in\mathbb{N}$ satisfy $M>n(\frac{1}{p\wedge q}-1)_+ +\Delta$.
Using Lemma \ref{59}, we obtain,
for any $j\in\mathbb{Z}$, $i\in\{j-1,j,j+1\}$,
$R\in\mathscr{Q}_i$, and $x\in\mathbb{R}^n$,
\begin{align}\label{72}
\left|\left(\varphi_j*\psi_R\right)(x)\right|
&=\left|\int_{\mathbb{R}^n}\varphi_j(x-y)\psi_R(y)\,dy\right|
=|R|^{\frac12}\left|\left(\varphi_j*\psi_i \right)(x-x_R)\right|\\
&\lesssim|R|^{\frac12}2^{-|i-j|M}\frac{2^{-(i\wedge j)M}}{[2^{-(i\wedge j)}+|x-x_R|]^{n+M}}\notag\\
&\sim|R|^{-\frac12}\frac{1}{\{1+[\ell(R)]^{-1}|x-x_R|\}^{n+M}}.\notag
\end{align}
Let $u:=\{u_Q\}_{Q\in\mathscr{Q}}$,
where, for any $Q\in\mathscr{Q}$, $u_Q:=|A_Q\vec t_Q|$.
Applying \eqref{71}, \eqref{73}, \eqref{72}, and Lemma \ref{33y}, we conclude that,
for any $j\in\mathbb{Z}$, $Q\in\mathscr{Q}_j$, and $x\in Q$,
\begin{align}\label{152}
\left|A_Q\left[\varphi_j*\left(T_\psi\vec t\right)\right](x)\right|
&\lesssim\sum_{i=j-1}^{j+1}\sum_{R\in\mathscr{Q}_i}
|R|^{-\frac12}\frac{u_R\{1+[\ell(R)]^{-1}|x_Q-x_R|\}^{\Delta}}
{\{1+[\ell(R)]^{-1}|x-x_R|\}^{n+M}}
\sim|Q|^{-\frac12}\sum_{i=j-1}^{j+1}I_i (x),
\end{align}
where, for any $i\in\mathbb Z$,
$$
I_i(x):=\sum_{R\in\mathscr{Q}_i}
\frac{u_R}{\{1+[\ell(R)]^{-1}|x-x_R|\}^{n+M-\Delta}}.
$$
By the geometrical properties of dyadic cubes, we easily find that,
for any $j\in\mathbb{Z}$ and $x\in\mathbb{R}^n$,
there exist unique $Q^*\in\mathscr{Q}_{j+1}$, $Q\in\mathscr{Q}_j$,
and $Q^{**}\in\mathscr{Q}_{j-1}$ such that
$x\in Q^{*}\subset Q\subset Q^{**}$. Next, we claim that
\begin{equation}\label{153}
\sum_{i=j-1}^{j+1}I_i(x)
\lesssim\left(u_{p\wedge q,\widetilde{\lambda}}^*\right)_{Q^*}
+\left(u_{p\wedge q,\widetilde{\lambda}}^*\right)_Q
+\left(u_{p\wedge q,\widetilde{\lambda}}^*\right)_{Q^{**}},
\end{equation}
where $\widetilde{\lambda}:=(n+M-\Delta)(p\wedge q\wedge 1)$.
Due to similarity, to show \eqref{153},
we only need to prove that,
for any $j\in\mathbb{Z}$, $Q\in\mathscr{Q}_j$, and $x\in Q$,
\begin{equation}\label{154}
I_j(x)\lesssim\left(u_{p\wedge q,\widetilde{\lambda}}^*\right)_Q.
\end{equation}
To show this, we consider the following two cases on $p\wedge q$.

\emph{Case 1)} $p\wedge q\in(0,1]$.
In this case, $M>\frac{n}{p\wedge q}-n+\Delta$
and $\widetilde{\lambda}=(n+M-\Delta)(p\wedge q)>n$.
By Lemmas \ref{33y} and \ref{famous 2}, we conclude that,
for any $j\in\mathbb{Z}$, $Q\in\mathscr{Q}_j$, and $x\in Q$,
\begin{align*}
I_j(x)
&\sim\sum_{R\in\mathscr{Q}_j}
\frac{u_R}{\{1+[\ell(R)]^{-1}|x_Q-x_R|\}^{n+M-\Delta}}\\
&\leq\left[\sum_{R\in\mathscr{Q}_j}
\frac{(u_R)^{p\wedge q}}
{\{1+[\ell(R)]^{-1}|x_Q-x_R|\}^{\widetilde\lambda}}
\right]^{\frac{1}{p\wedge q}}
=\left(u_{p\wedge q,\widetilde\lambda}^*\right)_Q.
\end{align*}
This finishes the proof of \eqref{154} in this case.

\emph{Case 2)} $p\wedge q\in(1,\infty)$.
In this case, $M>\Delta$ and $\widetilde\lambda=n+M-\Delta>n$.
From Lemma \ref{33y}, H\"older's inequality, and \eqref{33z}, we infer that,
for any $j\in\mathbb{Z}$, $Q\in\mathscr{Q}_j$, and $x\in Q$,
\begin{align*}
I_j(x)
&\sim\sum_{R\in\mathscr{Q}_j}
\frac{u_R}{\{1+[\ell(R)]^{-1}|x_Q-x_R|\}^{\widetilde\lambda}}\\
&\leq\left(\sum_{R\in\mathscr{Q}_j}
\frac{1}{\{1+[\ell(R)]^{-1}|x_Q-x_R|\}^{\widetilde\lambda}}
\right)^{\frac{1}{(p\wedge q)'}}
\left(\sum_{R\in\mathscr{Q}_j}
\frac{(u_R)^{p\wedge q}}{\{1+[\ell(R)]^{-1}|x_Q-x_R|\}^{\widetilde\lambda}}
\right)^{\frac{1}{p\wedge q}}\\
&=\left[\sum_{k\in\mathbb{Z}^n}
\frac{1}{(1+|2^jx_Q-k|)^{\widetilde\lambda}}
\right]^{\frac{1}{(p\wedge q)'}}
\left(u_{p\wedge q,\widetilde\lambda}^*\right)_Q
\sim\left(u_{p\wedge q,\widetilde\lambda}^*\right)_Q.
\end{align*}
This finishes the proof of \eqref{154} in this case and hence \eqref{153}.

Using \eqref{152} and \eqref{153}, we obtain,
for any $j\in\mathbb{Z}$ and $x\in\mathbb{R}^n$,
\begin{align}\label{242}
&\left|A_j(x)\left[\varphi_j*\left(T_\psi\vec t\right)\right](x)\right|
\lesssim\left(u_{p\wedge q,\widetilde{\lambda}}^*\right)_{j+1}(x)
+\left(u_{p\wedge q,\widetilde{\lambda}}^*\right)_{j}(x)
+\left(u_{p\wedge q,\widetilde{\lambda}}^*\right)_{j-1}(x),
\end{align}
where $(u_{p\wedge q,\widetilde{\lambda}}^*)_j$ is the same as in \eqref{tj}.
Therefore, we have
$$
\left\|T_\psi\vec t\right\|_{\dot A^{s,\tau}_{p,q}(\mathbb{A},\varphi)}
\lesssim\left\|u_{p\wedge q,\widetilde{\lambda}}^*\right\|_{\dot a^{s,\tau}_{p,q}}.
$$
From this, Theorem \ref{11}, and Lemma \ref{4}, we deduce that
$$
\left\|T_\psi\vec t\right\|_{\dot A^{s,\tau}_{p,q}(W,\varphi)}
\sim\left\|T_\psi\vec t\right\|_{\dot A^{s,\tau}_{p,q}(\mathbb{A},\varphi)}
\lesssim\left\|u_{p\wedge q,\widetilde{\lambda}}^*\right\|_{\dot a^{s,\tau}_{p,q}}
\sim\left\|\vec t\right\|_{\dot a^{s,\tau}_{p,q}(W)}.
$$
This finishes the proof of the boundedness of $T_\psi$.

Finally, if $\varphi$ and $\psi$ satisfy \eqref{21}, then, by Lemma \ref{7},
we find that $T_\psi\circ S_\varphi$
is the identity on $\dot A^{s,\tau}_{p,q}(W,\widetilde{\varphi})$.
This finishes the proof of Theorem \ref{phi}.
\end{proof}

Applying Theorem \ref{phi}, we can obtain the following proposition
which proves that $\dot A^{s,\tau}_{p,q}(W,\varphi)$
is independent of the choice of $\varphi$.

\begin{proposition}\label{38}
Let $s\in\mathbb{R}$, $\tau\in[0,\infty)$, $p\in(0,\infty)$, and $q\in(0,\infty]$.
Let $\varphi\in\mathcal{S}$ satisfy \eqref{19} and \eqref{20}, and let $W\in A_p$.
Then $\dot A^{s,\tau}_{p,q}(W,\varphi)$ is independent of the choice of $\varphi$.
\end{proposition}

\begin{proof}
Let $\varphi^{(1)}, \varphi^{(2)}, \psi^{(2)}\in\mathcal{S}$
satisfy both \eqref{19} and \eqref{20}
and assume both $\varphi^{(2)}$ and $\psi^{(2)}$ satisfy \eqref{21}.
Then, from both Lemma \ref{7} and Theorem \ref{phi}, we infer that,
for any $\vec{f}\in\dot A^{s,\tau}_{p,q}(W,\varphi^{(2)})$,
\begin{align*}
\left\|\vec f\right\|_{\dot A^{s,\tau}_{p,q}(W, \varphi^{(1)})}
&=\left\|\left(T_{\widetilde{\psi^{(2)}}}\circ S_{\widetilde{\varphi^{(2)}}}\right)
\left(\vec f\right)\right\|_{\dot A^{s,\tau}_{p,q}(W, \varphi^{(1)})}
\lesssim\left\|S_{\widetilde{\varphi^{(2)}}}\vec f\right\|_{\dot a^{s,\tau}_{p,q}(W)}
\lesssim\left\|\vec f\right\|_{\dot A^{s,\tau}_{p,q}(W, \varphi^{(2)})}.
\end{align*}
By symmetry, we also obtain the reverse inequality.
This finishes the proof of Proposition \ref{38}.
\end{proof}

Based on Proposition \ref{38}, in what follows, we denote $\dot A^{s,\tau}_{p,q}(W,\varphi)$ simply by $\dot A^{s,\tau}_{p,q}(W)$.
Moreover, using Proposition \ref{38} and Theorem \ref{11},
we easily obtain the following conclusion; we omit the details.

\begin{corollary}
Let $s\in\mathbb{R}$, $\tau\in[0,\infty)$, $p\in(0,\infty)$, $q\in(0,\infty]$,
$\varphi\in\mathcal{S}$ satisfy both \eqref{19} and \eqref{20}, $W\in A_p$,
and $\mathbb{A}:=\{A_Q\}_{Q\in\mathscr{Q}}$ be
a sequence of reducing operators of order $p$ for $W$.
Then $\dot A^{s,\tau}_{p,q}(\mathbb{A},\varphi)$
is independent of the choice of $\varphi$.
\end{corollary}

Again, in what follows, we denote $\dot A^{s,\tau}_{p,q}(\mathbb{A},\varphi) $ simply by $\dot A^{s,\tau}_{p,q}(\mathbb{A})$.
By an argument similar to that used in the proof of \cite[Corollary 3.14]{b07},
we obtain the following proposition.
For the convenience of the reader, we give the details of its proof.

\begin{proposition}\label{174}
Let $s\in\mathbb{R}$, $\tau\in[0,\infty)$, $p\in(0,\infty)$, $q\in(0,\infty]$,
and $W\in A_p$ have $A_p$-dimensions $(d,\widetilde d,\Delta)$.
Then $\dot A^{s,\tau}_{p,q}(W)\subset(\mathcal{S}_\infty')^m$.
Moreover, if $M\in\mathbb{Z}_+$ satisfies \eqref{239},
then there exists a positive constant $C$ such that,
for any $\vec f\in\dot A^{s,\tau}_{p,q}(W)$ and $\phi\in\mathcal{S}_\infty$,
$$
\left|\left\langle\vec f,\phi\right\rangle\right|
\leq C\left\|\vec f\right\|_{\dot A^{s,\tau}_{p,q}(W)}\|\phi\|_{S_{M+1}}.
$$
\end{proposition}

\begin{proof}
Let $\varphi,\psi\in\mathcal{S}$ satisfy \eqref{19}, \eqref{20}, and \eqref{21}.
By both Lemmas \ref{7} and \ref{3} and Theorem \ref{phi}, we find that,
for any $\vec f\in\dot A^{s,\tau}_{p,q}(W)$ and $\phi\in\mathcal{S}_\infty$,
\begin{align*}
\left|\left\langle\vec f,\phi\right\rangle\right|
&=\left|\left\langle\left(T_\psi\circ S_\varphi\right)\vec f,\phi\right\rangle\right|
\leq\sum_{Q\in\mathscr{Q}}\left|\left(S_\varphi\vec f\right)_Q\right||\langle\psi_Q,\phi\rangle|\\
&\lesssim\left\|S_\varphi\vec f\right\|_{\dot a^{s,\tau}_{p,q}(W)}\|\phi\|_{S_{M+1}}
\lesssim\left\|\vec f\right\|_{\dot A^{s,\tau}_{p,q}(W)}\|\phi\|_{S_{M+1}}.
\end{align*}
This finishes the proof of Proposition \ref{174}.
\end{proof}

Applying Proposition \ref{174} and
an argument similar to that used in the proof of \cite[Proposition 2.3.1]{g14},
we obtain the following conclusion; we omit the details.

\begin{proposition}\label{175}
Let $s\in\mathbb{R}$, $\tau\in[0,\infty)$,
$p\in(0,\infty)$, $q\in(0,\infty]$, and $W\in A_p$.
Then $\dot A^{s,\tau}_{p,q}(W)$ is a complete quasi-normed space.
\end{proposition}

Finally, we have the following lifting property.
Recall that, for any $\sigma\in\mathbb R$, the \emph{lifting operator} $I_\sigma$
(see, for instance, \cite[Section 5.2.3]{t83}) is defined by setting,
for any $f\in\mathcal{S}_\infty'$,
$$
\dot I_\sigma(f):=\left(|\cdot|^\sigma\widehat{f}\right)^\vee,
$$
where the symbol $\vee$ denotes the \emph{inverse Fourier transform}.
It is well known that $\dot I_\sigma$ maps $\mathcal{S}_\infty'$ onto itself.

\begin{proposition}\label{257}
Let $s,\sigma\in\mathbb{R}$, $\tau\in[0,\infty)$, $p\in(0,\infty)$,
$q\in(0,\infty]$, and $W\in A_p$.
Then $\dot I_\sigma$ maps $\dot A_{p,q}^{s,\tau}(W)$ isomorphically onto $\dot A_{p,q}^{s-\sigma,\tau}(W)$.
Moreover, for any $\vec f\in(\mathcal{S}_\infty')^m$,
$$
\left\|\vec f\right\|_{\dot A_{p,q}^{s,\tau}(W)}
\sim\left\|\dot I_\sigma\vec f\right\|_{\dot A_{p,q}^{s-\sigma,\tau}(W)},
$$
where the positive equivalence constants are independent of $\vec f$.
\end{proposition}

\begin{proof}
Let $\{\varphi_j\}_{j\in\mathbb Z}$ be the same as in Definition \ref{def 3.4}.
Observe that, by both the definition of $\dot I_\sigma$
and the property of the inverse Fourier transform, we have,
for any $j\in\mathbb Z$ and $\vec f\in(\mathcal{S}_\infty')^m$,
\begin{align}\label{258}
2^{j(s-\sigma)}W^{\frac{1}{p}}\left[\varphi_j*\left(\dot I_\sigma\vec f\right)\right]
&=2^{j(s-\sigma)}W^{\frac{1}{p}}\left[
\varphi_j*\left(|\cdot|^\sigma\widehat{\vec f}\right)^\vee\right]
=2^{j(s-\sigma)}W^{\frac{1}{p}}\left(\widehat{\varphi_j}
|\cdot|^\sigma\widehat{\vec f}\right)^\vee.
\end{align}
Let $\psi:=(|\cdot|^\sigma\widehat{\varphi})^\vee$.
Notice that $\varphi\in\mathcal{S}$ satisfy both \eqref{19} and \eqref{20},
which further implies that $\psi\in\mathcal{S}$ also satisfies both \eqref{19} and \eqref{20}.
Moreover, notice that
\begin{align*}
\psi_j*\vec f
&=\left(\widehat{\psi_j}\widehat{\vec f}\right)^\vee
=\left[ 2^{-jn}\widehat{\psi}\left( 2^{-j}\cdot\right)\widehat{\vec f}\right]^\vee \\
&=2^{-j\sigma}\left[ 2^{-jn} |\cdot|^\sigma
\widehat{\varphi}\left( 2^{-j}\cdot\right)\widehat{\vec f}\right]^\vee
=2^{-j\sigma}\left( |\cdot|^\sigma
\widehat{\varphi_j}\widehat{\vec f}\right)^\vee.
\end{align*}
From this and equation \eqref{258}, we deduce that,
for any $\vec f\in(\mathcal{S}_\infty')^m$,
\begin{equation}\label{259}
\left\|\dot I_\sigma \vec f\right\|_{\dot A_{p,q}^{s-\sigma,\tau}(W,\varphi)}
\lesssim\left\|\vec f\right\|_{\dot A_{p,q}^{s,\tau}(W,\psi)}.
\end{equation}
By Proposition \ref{38}, both norms above are
independent of the particular $\varphi,\psi\in\mathcal{S}$
with the properties \eqref{19} and \eqref{20}, and hence we drop them from the notation.

On the other hand, by \eqref{259} with $s$ and $\sigma$ replaced,
respectively, by $s-\sigma$ and $-\sigma$,
we conclude that, for any $\vec f\in(\mathcal{S}_\infty')^m$,
$$
\left\|\vec f\right\|_{\dot A_{p,q}^{s,\tau}(W)}
=\left\|\dot I_{-\sigma}\left(\dot I_\sigma\vec f\right)\right\|_{\dot A_{p,q}^{s,\tau}(W)}
\lesssim
\left\|\dot I_\sigma\vec f\right\|_{\dot A_{p,q}^{s-\sigma,\tau}(W)}.
$$
This finishes the proof of Proposition \ref{257}.
\end{proof}

\section{Averaging Matrix-Weighted Triebel--Lizorkin Spaces for $p=\infty$}
\label{F_infty}

In this section, we introduce matrix-weighted Triebel--Lizorkin spaces
$\dot F_{\infty,q}^s(\mathbb A)$ for the end-point exponent $p=\infty$
and obtain some results corresponding to Section \ref{BF type spaces}.
One of the many reasons for the relevance of these spaces, as we will see,
is that the four-parameter Besov-type and Triebel--Lizorkin-type spaces
$\dot A^{s,\tau}_{p,q}(W)$ can be identified with spaces from the
$\dot F_{\infty,q}^s(\mathbb A)$ scale as soon as
we exit the so called subcritical regime of the parameters.

Let us begin with some concepts.
For any $q\in(0,\infty]$ and any sequence $\{f_j\}_{j\in\mathbb{Z}}$
of measurable functions on $\mathbb{R}^n$, let
\begin{equation}\label{LF infty}
\|\{f_j\}_{j\in\mathbb Z}\|_{L\dot F_{\infty,q}}
:=\sup_{P\in\mathscr{Q}}\left[\fint_P\sum_{j=j_P}^\infty
\left|f_j(x)\right|^q\,dx\right]^{\frac{1}{q}}
\end{equation}
with the usual modification made when $q=\infty$. Directly from the definition, we find that
$L\dot F_{\infty,q}=L\dot F_{q,q}^{\frac{1}{q}}.$

We first recall the concepts of both Triebel--Lizorkin spaces for $p=\infty$
and corresponding sequence spaces; see \cite[(5.1) and (5.4)]{fj90}.

\begin{definition}
Let $s\in\mathbb{R}$, $q\in(0,\infty]$, and
$\varphi\in\mathcal{S}$ satisfy both \eqref{19} and \eqref{20}.
The \emph{homogeneous Triebel--Lizorkin space}
$\dot F_{\infty,q}^s$ is defined by setting
$$
\dot F_{\infty,q}^s:=\left\{f\in\mathcal{S}_\infty':\
\|f\|_{\dot F_{\infty,q}^s}<\infty\right\},
$$
where, for any $f\in\mathcal{S}_\infty'$,
$$
\|f\|_{\dot F_{\infty,q}^s}:=
\left\|\left\{2^{js}\varphi_j*f\right\}_{j\in\mathbb Z}\right\|_{L\dot F_{\infty,q}}
$$
with $\|\cdot\|_{L\dot F_{\infty,q}}$ the same as in \eqref{LF infty}.
\end{definition}

\begin{definition}
Let $s\in\mathbb{R}$ and $q\in(0,\infty]$.
The \emph{homogeneous Triebel--Lizorkin sequence space}
$\dot f_{\infty,q}^s$ is defined to be the set of all sequences
$t:=\{t_Q\}_{Q\in\mathscr{Q}}\subset\mathbb{C}$ such that
$$
\|t\|_{\dot f_{\infty,q}^s}
:=\left\|\left\{2^{js}t_j\right\}_{j\in\mathbb Z}\right\|_{L\dot F_{\infty,q}}<\infty,
$$
where $t_j$ for any $j\in\mathbb Z$ and $\|\cdot\|_{L\dot F_{\infty,q}}$
are the same as, respectively, in \eqref{tj} and \eqref{LF infty}.
\end{definition}

Notice that, directly by the definitions,
we obtain, for any $q\in(0,\infty)$,
$\dot F^s_{\infty,q}=\dot F^{s,\frac{1}{q}}_{q,q}$ and
$\dot f^s_{\infty,q}=\dot f^{s,\frac{1}{q}}_{q,q}.$
Now, we introduce the averaging matrix-weighted Triebel--Lizorkin space for $p=\infty$.

\begin{definition}
Let $s\in\mathbb{R}$, $q\in(0,\infty]$,
$\varphi\in\mathcal{S}$ satisfy both \eqref{19} and \eqref{20},
$p\in(0,\infty)$, $W\in A_p$, and
$\mathbb{A}:=\{A_Q\}_{Q\in\mathscr{Q}}$ be a sequence of
reducing operators of order $p$ for $W$.
The \emph{homogeneous averaging matrix-weighted Triebel--Lizorkin space}
$\dot F_{\infty,q}^s(\mathbb{A},\varphi)$ is defined by setting
$$
\dot F_{\infty,q}^s(\mathbb{A},\varphi):=\left\{\vec f\in(\mathcal{S}_\infty')^m:\
\left\|\vec f\right\|_{\dot F_{\infty,q}^s(\mathbb{A},\varphi)}<\infty\right\},
$$
where, for any $\vec f\in(\mathcal{S}_\infty')^m$,
$$
\left\|\vec f\right\|_{\dot F_{\infty,q}^s(\mathbb{A},\varphi)}
:=\left\|\left\{2^{js}\left|A_j \left(\varphi_j*\vec f\right)\right|
\right\}_{j\in\mathbb Z}\right\|_{L\dot F_{\infty,q}}
$$
with $A_j$ for any $j\in\mathbb Z$ and $\|\cdot\|_{L\dot F_{\infty,q}}$
the same as, respectively, in \eqref{Aj} and \eqref{LF infty}.
\end{definition}

By \eqref{equ_reduce}, we find that
$\dot F_{\infty,q}^s(\mathbb{A},\varphi)$ is independent of the choice of $\mathbb{A}$.

\begin{definition}
Let $s\in\mathbb{R}$, $q\in(0,\infty]$,
$p\in(0,\infty)$, $W\in A_p$, and
$\mathbb{A}:=\{A_Q\}_{Q\in\mathscr{Q}}$ be a sequence of
reducing operators of order $p$ for $W$.
The \emph{homogeneous averaging matrix-weighted Triebel--Lizorkin sequence space}
$\dot f_{\infty,q}^s(\mathbb{A})$ is defined to be the set of all sequences
$\vec t:=\{\vec t_Q\}_{Q\in\mathscr{Q}}\subset\mathbb{C}^m$ such that
$$
\left\|\vec t\right\|_{\dot f_{\infty,q}^s(\mathbb{A})}
:=\left\|\left\{2^{js}\left|A_j\vec t_j\right|
\right\}_{j\in\mathbb Z}\right\|_{L\dot F_{\infty,q}}<\infty,
$$
where, for any $j\in\mathbb Z$, $A_j$, $\vec t_j$, and $\|\cdot\|_{L\dot F_{\infty,q}}$
are the same as, respectively, in \eqref{Aj}, \eqref{vec tj}, and \eqref{LF infty}.
\end{definition}

Again, directly by definitions, when $p=q$, we have both
$$
\dot F^s_{\infty,q}(\mathbb{A},\varphi)=\dot F^{s,\frac{1}{q}}_{q,q}(\mathbb{A})
\text{ and }
\dot f^s_{\infty,q}(\mathbb{A},\varphi)=\dot f^{s,\frac{1}{q}}_{q,q}(\mathbb{A}).
$$
However, when $p\neq q$, the identification of
$\dot F^s_{\infty,q}(\mathbb{A},\varphi)$
and $\dot f^s_{\infty,q}(\mathbb{A})$
with those that we have studied earlier is not so obvious,
but we have the following conclusion
which is the main result of this section.

\begin{theorem}\label{97}
Let $s\in\mathbb{R}$, $q\in(0,\infty]$,
$p\in(0,\infty)$, $W\in A_p$, and
$\mathbb{A}:=\{A_Q\}_{Q\in\mathscr{Q}}$ be a sequence of
reducing operators of order $p$ for $W$.
Then
$\dot F_{\infty,q}^s(\mathbb{A})=\dot F_{p,q}^{s,\frac{1}{p}}(\mathbb{A})$
and
$\dot f_{\infty,q}^s(\mathbb{A})=\dot f_{p,q}^{s,\frac{1}{p}}(\mathbb{A})$
with equivalent quasi-norms.
\end{theorem}

In Theorem \ref{97}, we use the notation $\dot F_{\infty,q}^s(\mathbb{A})$
instead of $\dot F_{\infty,q}^s(\mathbb{A},\varphi)$
because we will show that $\dot F_{\infty,q}^s(\mathbb{A},\varphi)$
is independent of the choice of $\varphi$ (see Proposition \ref{independent 2} below).
To prove Theorem \ref{97}, we need several technical lemmas.
The following lemma is a simple corollary of \cite[Lemma 5.1]{fj90}.

\begin{lemma}\label{107}
Let $s\in\mathbb{R}$, $q\in(0,\infty]$, $\lambda\in(n,\infty)$,
$p\in(0,\infty)$, $W\in A_p$,
and $\{A_Q\}_{Q\in\mathscr{Q}}$ be a sequence of
reducing operators of order $p$ for $W$.
Then, for any $\vec{t}\in\dot f_{\infty,q}^s(\mathbb{A})$,
$$
\left\|\vec{t}\right\|_{\dot f_{\infty,q}^s(\mathbb{A})}
\sim\left\|\left(\left\{\left|A_Q\vec t_Q\right|\right\}_{Q\in\mathscr{Q}}
\right)_{q,\lambda}^*\right\|_{\dot f_{\infty,q}^s},
$$
where the positive equivalence constants are independent of $\vec t$.
\end{lemma}

\begin{proof}
Let $u:=\{u_Q\}_{Q\in\mathscr{Q}}$, where $u_Q:=|A_Q\vec t_Q|$ for any $Q\in\mathscr{Q}$.
Then, by \cite[Lemma 5.1]{fj90}, we find that
$\|\vec{t}\|_{\dot f_{\infty,q}^s(\mathbb{A})}
=\|u\|_{\dot f_{\infty,q}^s}
\sim\|u_{q,\lambda}^*\|_{\dot f_{\infty,q}^s}.$
This finishes the proof of Lemma \ref{107}.
\end{proof}

The following lemma is analogous to Lemma \ref{summary F}.

\begin{lemma}\label{summary F 2}
Let $q\in(0,\infty]$ and $M\in(n,\infty)$.
Suppose two sequences $\{g_j\}_{j\in\mathbb{Z}}$ and $\{h_j\}_{j\in\mathbb{Z}}$
of measurable functions on $\mathbb{R}^n$
satisfy that there exist $r\in(0,\min\{q,1\})$
and a positive constant $C$ such that,
for any $j\in\mathbb{Z}$ and $x\in\mathbb{R}^n$, \eqref{143} holds.
Then there exist a positive constant $\widetilde{C}$,
depending only on $C$, $n$, $p$, $q$, and $M$, such that
\begin{equation}\label{summary F 2 equ}
\left\|\left\{2^{js}g_j\right\}_{j\in\mathbb Z}\right\|_{L\dot F_{\infty,q}}
\leq\widetilde{C}\left\|\left\{2^{js}h_j\right\}_{j\in\mathbb Z}\right\|_{L\dot F_{\infty,q}},
\end{equation}
where $\|\cdot\|_{L\dot F_{\infty,q}}$ is the same as in \eqref{LF infty}.
\end{lemma}

\begin{proof}
To show \eqref{summary F 2 equ}, we need consider the following two cases on $q$.

\emph{Case 1)} $q\in(0,\infty)$. In this case,
by $L\dot F_{\infty,q}=L\dot F_{q,q}^{\frac 1q}$ and Lemma \ref{summary F} with $p=q$,
we obtain \eqref{summary F 2 equ}.

\emph{Case 2)} $q=\infty$. In this case, from \eqref{33x}, we infer that,
for any $j\in\mathbb Z$ and $x\in\mathbb R^n$,
\begin{align*}
\left|2^{js}g_j(x)\right|
&\lesssim \left[2^{jn}\int_{\mathbb{R}^n}\frac{1}{(1+2^j|x-z|)^M}
\left|2^{js}h_j(z)\right|^r\,dz\right]^{\frac 1r}\\
&\leq\left\|2^{js}h_j\right\|_{L^\infty}
\left[2^{jn}\int_{\mathbb{R}^n}\frac{1}{(1+2^j|x-z|)^M}\,dz\right]^{\frac 1r}
\sim\left\|2^{js}h_j\right\|_{L^\infty},
\end{align*}
which further implies that \eqref{summary F 2 equ} holds in this case.
This finishes the proof of Lemma \ref{summary F 2}.
\end{proof}

\begin{lemma}\label{108}
Let $s\in\mathbb{R}$, $q\in(0,\infty]$,
$\varphi\in\mathcal{S}$ satisfy both \eqref{19} and \eqref{20},
$p\in(0,\infty)$, $W\in A_p$,
and $\mathbb{A}:=\{A_Q\}_{Q\in\mathscr{Q}}$ be a sequence of
reducing operators of order $p$ for $W$.
Then $\vec f\in\dot F_{\infty,q}^s(\mathbb{A},\varphi)$ if and only if
$\vec f\in(\mathcal{S}_\infty')^m$ and $\sup_{\mathbb{A},\varphi}(\vec f)\in\dot f_{\infty,q}^s$,
where $\sup_{\mathbb{A},\varphi}$ is the same as in \eqref{sup}.
Moreover, there exists a constant $C\in[1,\infty)$ such that,
for any $\vec f\in(\mathcal{S}_\infty')^m$,
$$
\left\|\vec f\right\|_{\dot F_{\infty,q}^s(\mathbb{A},\varphi)}
\leq\left\|\sup_{\mathbb{A},\varphi}\left(\vec f\right)\right\|_{\dot f_{\infty,q}^s}
\leq C\left\|\vec f\right\|_{\dot F_{\infty,q}^s(\mathbb{A},\varphi)}.
$$
\end{lemma}

\begin{proof}
The first inequality is immediate from the definition of $\sup_{\mathbb{A},\varphi}(\vec f)$.
To prove the second inequality,
we fix $r\in(0,\min\{p,q,1\})$ and $M\in(\Delta+\frac{n}{r},\infty)$,
where $\Delta$ is such that $W$ has $A_p$-dimensions $(d,\widetilde d,\Delta)$.
Using \eqref{155}, we obtain, for any $j\in\mathbb{Z}$ and $x\in Q$,
$$
\left|g_j(x)\right|^r
\lesssim2^{jn}\int_{\mathbb{R}^n}
\frac{1}{(1+2^j|x-z|)^{(M-\Delta)r}}\left|h_j(x)\right|^r\,dz,
$$
where both $g_j$ and $h_j$ are the same as in \eqref{gj}.
From this and Lemma \ref{summary F 2} with $M$ replaced by $(M-\Delta)r$, we deduce that
\begin{align*}
\left\|\sup_{\mathbb{A},\varphi}\left(\vec f\right)\right\|_{\dot f_{\infty,q}^s}
=\left\|\left\{2^{js}g_j\right\}_{j\in\mathbb Z}\right\|_{L\dot F_{\infty,q}^{\tau}}
\lesssim\left\|\left\{2^{js}h_j\right\}_{j\in\mathbb Z}\right\|_{L\dot F_{\infty,q}^{\tau}}
=\left\|\vec f\right\|_{\dot F_{\infty,q}^s(\mathbb{A},\varphi)}.
\end{align*}
This finishes the proof of Lemma \ref{108}.
\end{proof}

To show that $T_\psi$ is well defined for any
$\vec t\in\dot f_{\infty,q}^s(\mathbb{A})$,
we have the following conclusion.

\begin{lemma}\label{103}
Let $s\in\mathbb{R}$, $\tau\in[0,\infty)$, $q\in(0,\infty]$, and $p\in(0,\infty)$.
Let $W\in A_p$ have $A_p$-dimensions $(d,\widetilde d,\Delta)$,
and $\mathbb{A}:=\{A_Q\}_{Q\in\mathscr{Q}}$ be a sequence of
reducing operators of order $p$ for $W$.
Then, for any $\vec t\in\dot f_{\infty,q}^s(\mathbb{A})$ and $\psi\in\mathcal{S}_\infty$,
$\sum_{Q\in\mathscr{Q}}\vec t_Q\psi_Q$ converges in $(\mathcal{S}_\infty')^m$.
Moreover, if $M\in\mathbb{Z}_+$ satisfies
\begin{equation}\label{255}
M>\max\left\{\frac{d}{p}+s,\frac{\widetilde{d}}{p'}-s,\Delta\right\},
\end{equation}
then there exists a positive constant $C$ such that,
for any $\vec t\in\dot f_{\infty,q}^s(\mathbb{A})$ and $\psi,\phi\in\mathcal{S}_\infty$,
$$
\sum_{Q\in\mathscr{Q}}\left|\vec{t}_Q\right||\langle\psi_Q,\phi\rangle|
\leq C\left\|\vec t\right\|_{\dot f_{\infty,q}^s(\mathbb{A})}
\|\psi\|_{S_{M+1}}\|\phi\|_{S_{M+1}}.
$$
\end{lemma}

\begin{proof}
Let $\vec t:=\{\vec t_Q\}_{Q\in\mathscr{Q}}\in\dot f_{\infty,q}^s(\mathbb{A})$.
By the definition of $\|\cdot\|_{\dot f_{\infty,q}^s(\mathbb{A})}$,
we conclude that, for any $Q\in\mathscr{Q}$,
$$
\left|\vec{t}_Q\right|
\leq\left\|A_Q^{-1}\right\|\left|A_Q\vec{t}_Q\right|
\lesssim\left\|A_{Q_{0,\mathbf{0}}}^{-1}\right\|
\left\|A_{Q_{0,\mathbf{0}}}A_Q^{-1}\right\||Q|^{\frac{s}{n}+\frac12}
\left\|\vec t\right\|_{\dot f_{\infty,q}^s(\mathbb{A})}
$$
and hence, for any $\phi\in\mathcal{S}_\infty$,
$$
\sum_{Q\in\mathscr{Q}}\left|\vec{t}_Q\right||\langle\psi_Q,\phi\rangle|
\lesssim\left\|\vec t\right\|_{\dot f_{\infty,q}^s(\mathbb{A})}
\sum_{Q\in\mathscr{Q}}|Q|^{\frac{s}{n}+\frac12}
\left\|A_{Q_{0,\mathbf{0}}}A_Q^{-1}\right\||\langle\psi_Q,\phi\rangle|.
$$
From this, \eqref{2}, \eqref{24}, \eqref{255}, and Lemma \ref{253x}, we infer that
\begin{align*}
&\sum_{Q\in\mathscr{Q}}\left|\vec{t}_Q\right||\langle\psi_Q,\phi\rangle|\\
&\quad\lesssim\left\|\vec t\right\|_{\dot f_{\infty,q}^s(\mathbb{A})}
\|\psi\|_{S_{M+1}}\|\phi\|_{S_{M+1}}\sum_{Q\in\mathscr{Q}}|Q|^{\frac{s}{n}+\frac12}\\
&\qquad\times\min\left\{[\ell(Q)]^{-(M-\frac{d}{p}+\frac{n}{2})},[\ell(Q)]^{M+\frac{n}{2}-\frac{\widetilde d}{p'}}\right\}
\left[1+\frac{|x_Q|}{\ell(Q)\vee 1}\right]^{-(n+M-\Delta)}\\
&\quad=\left\|\vec t\right\|_{\dot f_{\infty,q}^s(\mathbb{A})}
\|\psi\|_{S_{M+1}}\|\phi\|_{S_{M+1}}\\
&\qquad\times\left[\sum_{j=0}^\infty2^{-j(s+n+M-\frac{\widetilde d}{p'})}
\sum_{k\in\mathbb{Z}^n}\left(1+2^{-j}|k|\right)^{-(n+M-\Delta)}\right.\\
&\qquad\left.+\sum_{j=-\infty}^{-1}2^{-j(s-M+\frac{d}{p})}
\sum_{k\in\mathbb{Z}^n}(1+|k|)^{-(n+M-\Delta)}\right]\\
&\quad\sim\left\|\vec t\right\|_{\dot f_{\infty,q}^s(\mathbb{A})}
\|\psi\|_{S_{M+1}}\|\phi\|_{S_{M+1}}\left[\sum_{j=0}^\infty2^{-j(s+M-\frac{\widetilde d}{p'})}
+\sum_{j=-\infty}^{-1}2^{-j(s-M+\frac{d}{p})}\right]\\
&\quad\sim\left\|\vec t\right\|_{\dot f_{\infty,q}^s(\mathbb{A})}
\|\psi\|_{S_{M+1}}\|\phi\|_{S_{M+1}}.
\end{align*}
This finishes the proof of Lemma \ref{103}.
\end{proof}

Now, we establish the $\varphi$-transform characterization of $F_{\infty,q}^s(\mathbb{A})$.

\begin{theorem}\label{phi infty}
Let $s\in\mathbb{R}$, $q\in(0,\infty]$,
$\varphi,\psi\in\mathcal{S}$ satisfy both \eqref{19} and \eqref{20},
$p\in(0,\infty)$, $W\in A_p$,
and $\mathbb{A}:=\{A_Q\}_{Q\in\mathscr{Q}}$ be a sequence of
reducing operators of order $p$ for $W$. Then both
$$
S_\varphi:\ \dot F_{\infty,q}^s(\mathbb{A},\widetilde{\varphi})\to\dot f_{\infty,q}^s(\mathbb{A})
\text{ and }
T_\psi:\ \dot f_{\infty,q}^s(\mathbb{A})\to\dot F_{\infty,q}^s(\mathbb{A},\varphi)
$$
are bounded. Furthermore, if $\varphi$ and $\psi$ satisfy \eqref{21}, then
$T_\psi\circ S_\varphi$ is the identity on
$\dot F_{\infty,q}^s(\mathbb{A},\widetilde{\varphi})$.
\end{theorem}

\begin{proof}
We first prove the boundedness of
$S_\varphi:\ \dot F_{\infty,q}^s(\mathbb{A},\widetilde{\varphi})\to\dot f_{\infty,q}^s(\mathbb{A})$.
Let $\vec f\in\dot F_{\infty,q}^s(\mathbb{A},\widetilde{\varphi})$ and
$$
\sup_{\mathbb{A},\widetilde{\varphi}}\left(\vec f\right)
:=\left\{\sup_{\mathbb{A},\widetilde{\varphi},Q}\left(\vec f\right)\right\}_{Q\in\mathscr{Q}}
$$
be the same as in \eqref{sup}.
By the definition of $\sup_{\mathbb{A},\widetilde{\varphi},Q}(\vec f)$,
we find that, for any $Q\in\mathscr{Q}$,
\begin{align*}
\left|A_Q\left(S_{\varphi}\vec f\right)_Q\right|
=\left|A_Q\left\langle\vec f,\varphi_Q\right\rangle\right|
=|Q|^{\frac12}\left|A_Q\left(\widetilde{\varphi}_{j_Q}*\vec f\right)(x_Q)\right|
\leq\sup_{\mathbb{A},\widetilde{\varphi},Q}\left(\vec f\right).
\end{align*}
This, together with Lemma \ref{108}, further implies that
$$
\left\| S_{\varphi}\vec f\right\|_{\dot f_{\infty,q}^s(\mathbb{A})}
\leq\left\|\sup_{\mathbb{A},\widetilde{\varphi}}\left(\vec f\right)\right\|_{\dot f_{\infty,q}^s}
\sim\left\|\vec f\right\|_{\dot F_{\infty,q}^s(\mathbb{A},\widetilde{\varphi})},
$$
which completes the proof of the boundedness of $S_{\varphi}$.

Next, we show the boundedness of
$T_\psi:\ \dot f_{\infty,q}^s(\mathbb{A})\to\dot F_{\infty,q}^s(\mathbb{A},\varphi)$.
Let $\vec t:=\{\vec t_Q\}_{Q\in\mathscr{Q}}\in\dot f_{\infty,q}^s(\mathbb{A})$.
Let $M\in\mathbb{N}$ satisfy $M>n(\frac{1}{q}-1)_+ +\Delta$
and $\widetilde{\lambda}:=(n+M-\Delta)(q\wedge1)$.
Let $u:=\{u_Q\}_{Q\in\mathscr{Q}}$,
where $u_Q:=|A_Q\vec t_Q|$ for any $Q\in\mathscr{Q}$.
By \eqref{242} with $p$ replaced by $q$, we conclude that,
for any $j\in\mathbb{Z}$ and $x\in\mathbb{R}^n$,
$$
\left|A_j(x)\left[\varphi_j*\left(T_\psi\vec t\right)\right](x)\right|
\lesssim\left(u_{q,\widetilde{\lambda}}^*\right)_{j+1}(x)
+\left(u_{q,\widetilde{\lambda}}^*\right)_{j}(x)
+\left(u_{q,\widetilde{\lambda}}^*\right)_{j-1}(x).
$$
From this and Lemma \ref{107}, we deduce that
$$
\left\|T_\psi\vec t\right\|_{\dot F_{\infty,q}^s(\mathbb{A},\varphi)}
\lesssim\left\|u_{q,\widetilde{\lambda}}^*\right\|_{\dot f_{\infty,q}^s}
\sim\left\|\vec t\right\|_{\dot f_{\infty,q}^s(\mathbb{A})}.
$$
This finishes the proof of the boundedness of $T_\psi$.

Finally, if both $\varphi$ and $\psi$ satisfy \eqref{21}, then, by Lemma \ref{7},
we find that $T_\psi\circ S_\varphi$
is the identity on $\dot A^{s,\tau}_{p,q}(W,\widetilde{\varphi})$.
This finishes the proof of Theorem \ref{phi}.
\end{proof}

By an argument similar to that used in the proof of Proposition \ref{38},
we obtain the following conclusion; we omit the details.

\begin{proposition}\label{independent 2}
Let $s\in\mathbb{R}$, $q\in(0,\infty]$,
$\varphi\in\mathcal{S}$ satisfy both \eqref{19} and \eqref{20},
$p\in(0,\infty)$, $W\in A_p$, and
$\mathbb{A}:=\{A_Q\}_{Q\in\mathscr{Q}}$ be a sequence of
reducing operators of order $p$ for $W$.
Then $\dot F_{\infty,q}^s(\mathbb{A},\varphi)$
is independent of the choice of $\varphi$.
\end{proposition}

Based on Proposition \ref{independent 2}, in what follows,
we denote $\dot F_{\infty,q}^s(\mathbb{A},\varphi)$
simply by $\dot F_{\infty,q}^s(\mathbb{A})$.
Now, we can prove Theorem \ref{97}.

\begin{proof}[Proof of Theorem \ref{97}]
We first show that $\dot f_{\infty,q}^s(\mathbb{A})=\dot f_{p,q}^{s,\frac{1}{p}}(\mathbb{A})$.
Let $\vec t:=\{\vec t_Q\}_{Q\in\mathscr{Q}}\subset\mathbb C^m$
and define $u:=\{u_Q\}_{Q\in\mathscr{Q}}$
by setting $u_Q:=|A_Q\vec t_Q|$ for any $Q\in\mathscr{Q}$.
Then, by \cite[Corollary 5.7]{fj90}, we obtain
\begin{equation}\label{159}
\left\|\vec{t}\right\|_{\dot f_{\infty,q}^s(\mathbb{A})}
=\|u\|_{\dot f_{\infty,q}^s}
\sim\|u\|_{\dot f_{p,q}^{s,\frac{1}{p}}}
=\left\|\vec{t}\right\|_{\dot f_{p,q}^{s,\frac{1}{p}}(\mathbb{A})}
\end{equation}
and hence $\dot f_{\infty,q}^s(\mathbb{A})
=\dot f_{p,q}^{s,\frac{1}{p}}(\mathbb{A}) $.

Next, we prove that $\dot F_{\infty,q}^s(\mathbb{A})
=\dot F_{p,q}^{s,\frac{1}{p}}(\mathbb{A}) $.
Let $\psi\in\mathcal{S}$ satisfy both \eqref{19} and \eqref{20},
and let both $\varphi$ and $\psi$ satisfy \eqref{21}.
From Theorem \ref{phi infty}, \eqref{159}, and Theorem \ref{phi},
we infer that, for any $\vec f\in\dot F_{p,q}^{s,\frac{1}{p}}(\mathbb{A})$,
\begin{align}\label{new add 11}
\left\|\vec f\right\|_{\dot F_{\infty,q}^s(\mathbb{A})}
&=\left\|\left(T_\psi\circ S_\varphi\right)\left(\vec f\right)\right\|_{\dot F_{\infty,q}^s(\mathbb{A})}
\lesssim\left\|S_\varphi\left(\vec f\right)\right\|_{\dot f_{\infty,q}^s(\mathbb{A})}\\
&\sim\left\|S_\varphi\left(\vec f\right)\right\|_{\dot f_{p,q}^{s,\frac{1}{p}}(\mathbb{A})}
\lesssim\left\|\vec f\right\|_{\dot F_{p,q}^{s,\frac{1}{p}}(\mathbb{A})}.\notag
\end{align}
Applying an argument similar to that used in the estimation of \eqref{new add 11},
we also obtain the reverse inequality.
Thus, $\dot F_{\infty,q}^s(\mathbb{A})=\dot F_{p,q}^{s,\frac{1}{p}}(\mathbb{A})$
with equivalent quasi-norms.
This finishes the proof of Theorem \ref{97}.
\end{proof}

Applying Theorem \ref{97} and Propositions \ref{174}, \ref{175}, and \ref{257},
we obtain the following three propositions; we omit the details.

\begin{proposition}
Let $s\in\mathbb{R}$, $q\in(0,\infty]$, $p\in(0,\infty)$,
$W\in A_p$ have the $A_p$-dimension $d\in[0,n)$, and
$\mathbb{A}:=\{A_Q\}_{Q\in\mathscr{Q}}$ be a sequence of
reducing operators of order $p$ for $W$.
Then $\dot F_{\infty,q}^s(\mathbb{A})\subset(\mathcal{S}_\infty')^m$.
Moreover, if $M\in\mathbb{Z}_+$ satisfies \eqref{239},
then there exists a positive constant $C$ such that,
for any $\vec f\in\dot F_{\infty,q}^s(\mathbb{A})$ and $\phi\in\mathcal{S}_\infty$,
$$
\left|\left\langle\vec f,\phi\right\rangle\right|
\leq C\left\|\vec f\right\|_{\dot F_{\infty,q}^s(\mathbb{A})}\|\phi\|_{S_{M+1}}.
$$
\end{proposition}

\begin{proposition}
Let $s\in\mathbb{R}$, $q\in(0,\infty]$, $p\in(0,\infty)$, $W\in A_p$,
and $\mathbb{A}:=\{A_Q\}_{Q\in\mathscr{Q}}$ be a sequence of
reducing operators of order $p$ for $W$.
Then $\dot F_{\infty,q}^s(\mathbb{A})$ is a complete quasi-normed space.
\end{proposition}

\begin{proposition}
Let $ s, \sigma\in\mathbb{R}$, $q\in(0,\infty]$, $p\in(0,\infty)$, $W\in A_p$,
and $\mathbb{A}:=\{A_Q\}_{Q\in\mathscr{Q}}$ be a sequence of
reducing operators of order $p$ for $W$.
Then $\dot I_\sigma$ maps $\dot F_{\infty,q}^s(\mathbb{A})$ isomorphically onto $\dot F_{\infty,q}^{s-\sigma}(\mathbb{A})$.
Moreover, for any $\vec f\in(\mathcal{S}_\infty')^m$,
$$
\left\|\vec f\right\|_{\dot F_{\infty,q}^s(\mathbb{A})}
\sim\left\|\dot I_\sigma\vec f\right\|_{\dot F_{\infty,q}^{s-\sigma}(\mathbb{A})},
$$
where the positive equivalence constants are independent of $\vec f$.
\end{proposition}

Finally, we give an embedding between $\dot A^{s,\tau}_{p,q}(\mathbb{A})$
and $\dot F_{\infty,\infty}^{s+n\tau-\frac{n}{p}}(\mathbb{A})$
by the following lemma.

\begin{lemma}\label{summary B 2}
Let $p\in(0,\infty)$, $q\in(0,\infty]$, and $M\in(n,\infty)$.
Suppose two sequences $\{g_j\}_{j\in\mathbb{Z}}$ and $\{h_j\}_{j\in\mathbb{Z}}$
of measurable functions on $\mathbb{R}^n$ satisfy:
there exists a positive constant $C$ such that,
for any $j\in\mathbb{Z}$ and $x\in\mathbb{R}^n$,
\begin{equation}\label{142}
\left|g_j(x)\right|^p
\leq C2^{jn}\int_{\mathbb{R}^n}\frac{1}{(1+2^j|x-z|)^M}\left|h_j(z)\right|^p\,dz.
\end{equation}
Then there exists a positive constant $\widetilde{C}$,
depending only on $C$, $n$, and $M$, such that
$$
\left\|\left\{2^{j(s+n\tau-\frac{n}{p})}g_j\right\}_{j\in\mathbb Z}\right\|_{L\dot F_{\infty,\infty}}
\leq\widetilde{C}\left\|\left\{2^{js}h_j\right\}_{j\in\mathbb Z}\right\|_{L\dot A_{p,q}^{\tau}},
$$
where $L\dot F_{\infty,\infty}$ is the same as in \eqref{LF infty}.
\end{lemma}

\begin{proof}
By \eqref{142}, Lemma \ref{33}, and $M\in(n,\infty)$,
we conclude that, for any $P\in\mathscr{Q}$,
$ j\in\{j_P,j_P+1,\ldots\}$, and $x\in P$,
\begin{align*}
\left|g_j(x)\right|^p
&\lesssim2^{jn}\int_{3P}\frac{1}{(1+2^j|x-z|)^M}\left|h_j(z)\right|^p\,dz
+2^{jn} \sum_{k\in\mathbb{Z}^n,\,\|k\|_{\infty}\geq2}\int_{P+k\ell(P)}\cdots\\
&\lesssim2^{jn}\left[\left\|h_j\right\|_{L^p(3P)}^p
+2^{-(j-j_P)M}\sum_{k\in\mathbb{Z}^n,\,\|k\|_{\infty}\geq2}|k|^{-M}
\left\|h_j\right\|_{L^p(P+k\ell(P))}^p\right]\\
&\lesssim2^{j(\frac{n}{p}-s-n\tau)p}
\left\|\left\{2^{is}h_i\right\}_{i\in\mathbb Z}\right\|_{L\dot A_{p,q}^{\tau}}^p
\end{align*}
and hence
$$
\left\|\left\{2^{j(s+n\tau-\frac{n}{p})}g_j\right\}_{j\in\mathbb Z}\right\|_{L\dot F_{\infty,\infty}}
\lesssim\left\|\left\{2^{js}h_j\right\}_{j\in\mathbb Z}\right\|_{L\dot A_{p,q}^{\tau}}.
$$
This finishes the proof of Lemma \ref{summary B 2}.
\end{proof}

\begin{proposition}\label{embedding}
Let $s\in\mathbb R$, $\tau\in[0,\infty)$, $p\in(0,\infty)$, and $q\in(0,\infty]$.
Let $W\in A_p$ and let $\mathbb{A}:=\{A_Q\}_{Q\in\mathscr{Q}}$ be a sequence of
reducing operators of order $p$ for $W$.
Then $\dot A^{s,\tau}_{p,q}(\mathbb{A})\subset\dot F_{\infty,\infty}^{s+n\tau-\frac{n}{p}}(\mathbb{A})$.
Moreover, there exists a positive constant $C$ such that,
for any $\vec f\in(\mathcal{S}_\infty')^m$,
\begin{equation}\label{4.16x}
\left\|\vec f\right\|_{\dot F_{\infty,\infty}^{s+n\tau-\frac{n}{p}}(\mathbb{A})}
\leq C\left\|\vec f\right\|_{\dot A^{s,\tau}_{p,q}(\mathbb{A})}.
\end{equation}
\end{proposition}

\begin{proof}
Let $W$ have $A_p$-dimensions $(d,\widetilde d,\Delta)$. For any $j\in\mathbb{Z}$, let
$$
g_j:=\sum_{Q\in\mathscr{Q}_j}
\sup_{\mathbb{A},\varphi,Q}\left(\vec f\right)\widetilde{\mathbf{1}}_Q
\text{ and }
h_j:=\left|A_j\left(\varphi_j*\vec{f}\right)\right|.
$$
Let $r\in(0,\min\{1,p,q\})$ and $M>\frac{n}{r}+\Delta$.
It is shown in \eqref{155} that,
for any $j\in\mathbb Z$ and $x\in\mathbb R^n$,
\begin{equation*}
\left|g_j(x)\right|^r
\lesssim2^{jn}\int_{\mathbb{R}^n}
\frac{1}{(1+2^j|x-z|)^{(M-\Delta)r}}\left|h_j(z)\right|^r\,dz,
\end{equation*}
Using this combined with Lemmas \ref{108} and \ref{summary B 2}, with $M$ replaced by $(M-\Delta)r\in(n,\infty)$,
we find that
\begin{align*}
\left\|\vec f\right\|_{\dot F_{\infty,\infty}^{s+n\tau-\frac{n}{p}}(\mathbb{A})}
&\leq\left\|\sup_{\mathbb A,\varphi}\left(\vec f\right)\right\|_{\dot f_{\infty,\infty}^{s+n\tau-\frac{n}{p}}(\mathbb{A})}
=\left\|\left\{2^{j(s+n\tau-\frac{n}{p})}g_j\right\}_{j\in\mathbb Z}\right\|_{L\dot F_{\infty,\infty}} \\
&\lesssim\left\|\left\{2^{js}h_j\right\}_{j\in\mathbb Z}\right\|_{L\dot A_{p,q}^{\tau}}
=\left\|\vec f\right\|_{\dot A^{s,\tau}_{p,q}(\mathbb{A})}.
\end{align*}
This finishes the proof of \eqref{4.16x} and hence Proposition \ref{embedding}.
\end{proof}

While  the embedding of Proposition \ref{embedding} is valid for the full range of function space parameters, for a certain restricted range, this embedding can be improved to an isomorphism.
Motivated by the proof of \cite[Theorem 1]{yy13}, we first establish the following conclusion which gives the relation between the sequence spaces $\dot{a}^{s,\tau}_{p,q}(\mathbb{A})$ and $\dot{f}_{\infty,\infty}^s(\mathbb{A})$.

\begin{theorem}\label{92}
Let $s\in\mathbb{R}$, $p\in(0,\infty)$, and $q\in(0,\infty]$. Let $\mathbb A=\{A_Q\}_{Q\in\mathscr Q}$ be any family of positive definite matrices.
If $\tau>\frac1p$ or $(\tau,q)=(\frac1p,\infty)$, then
$\dot{a}^{s,\tau}_{p,q}(\mathbb{A})=\dot{f}_{\infty,\infty}^{s+n(\tau-\frac{1}{p})}(\mathbb{A})$
with equivalent quasi-norms.
\end{theorem}

\begin{proof}
Let $\vec t_j$ and $A_j$ for any $j\in\mathbb Z$ be the same as,
respectively, in \eqref{vec tj} and \eqref{Aj}. Recall that
\begin{equation*}
L\dot A_{p,q}:=\begin{cases}
\ell^qL^p&\text{if }\dot{a}^{s,\tau}_{p,q}(\mathbb{A})=\dot b_{p,q}^{s,\tau}(\mathbb{A}),\\
L^p\ell^q&\text{if }\dot{a}^{s,\tau}_{p,q}(\mathbb{A})=\dot f_{p,q}^{s,\tau}(\mathbb{A}).
\end{cases}
\end{equation*}
Then, by the definitions of both $\|\cdot\|_{\dot a^{s,\tau}_{p,q}(\mathbb A)}$
and $\|\cdot\|_{\dot f^{s+n(\tau-\frac{1}{p})}_{\infty,\infty}(\mathbb A)}$, we have
$$
\left\|\vec{t}\right\|_{\dot a^{s,\tau}_{p,q}(\mathbb A)}
=\sup_{P\in\mathscr Q}|P|^{-\tau}
\left\|\left\{2^{js}\left|\mathbf{1}_PA_j\vec{t}_j\right|
\right\}_{j\geq j_P}\right\|_{L\dot A_{p,q}}
$$
and
\begin{align*}
\left\|\vec{t}\right\|_{\dot f^{s+n(\tau-\frac{1}{p})}_{\infty,\infty}(\mathbb A)}
=\left\|\left\{2^{j[s+n(\tau-\frac{1}{p})]}\left|A_j\vec{t}_j\right|
\right\}_{j\in\mathbb Z}\right\|_{L\dot A_{\infty,\infty}}
=\sup_{j\in\mathbb Z}2^{j[s+n(\tau-\frac{1}{p})]}
\left\|\,\left|A_j\vec{t}_j \right|\,\right\|_{L^\infty}.
\end{align*}
Since $A_{j_P}\vec{t}_{j_P}$ for any $j_P\in\mathbb Z$ is constant on $P$
and $|P|=2^{-j_P n}$, it follows that
\begin{align*}
\left\|\vec{t}\right\|_{\dot a^{s,\tau}_{p,q}(\mathbb A)}
&\geq\sup_{P\in\mathscr Q}|P|^{-\tau}
\left\| 2^{j_Ps}\left|\mathbf{1}_P A_{j_P}\vec{t}_{j_P}\right|\,\right\|_{L^p}\\
&=\sup_{P\in\mathscr Q}|P|^{-\tau}|P|^{\frac{1}{p}}
\left\|\,\left|\mathbf{1}_P2^{j_Ps}A_{j_P}\vec{t}_{j_P}\right|\,\right\|_{L^\infty}\\
&=\sup_{j\in\mathbb Z}2^{j[s+n(\tau-\frac{1}{p})]}
\left\|\,\left|A_j\vec{t}_j\right|\,\right\|_{L^\infty}
=\left\|\vec{t}\right\|_{\dot f^{s+n(\tau-\frac{1}{p})}_{\infty,\infty}(\mathbb A)},
\end{align*}
where the assumption about the relative size of
the different parameters was not needed.

In the other direction, notice that
\begin{align}\label{5.35x}
\left\|\left\{\left|\mathbf{1}_P2^{js}A_j\vec{t}_j\right|
\right\}_{j\geq j_P}\right\|_{L\dot A_{p,q}}
&\leq\left\|\left\{\mathbf{1}_P 2^{-jn(\tau-\frac{1}{p})}
2^{j[s+n(\tau-\frac{1}{p})]}\left\|\,\left|A_j\vec{t}_j\right|\,\right\|_{L^\infty}
\right\}_{j\geq j_P}\right\|_{L\dot A_{p,q}}\\
&\leq\left\|\left\{\mathbf{1}_P2^{-jn(\tau-\frac{1}{p})}\right\}_{j\geq j_P}\right\|_{L\dot A_{p,q}}
\left\|\vec{t}\right\|_{\dot f^{s+n(\tau-\frac{1}{p})}_{\infty,\infty}(\mathbb A)},\notag
\end{align}
where
\begin{align}\label{5.35y}
\left\|\left\{\mathbf{1}_P2^{-jn(\tau-\frac{1}{p})}\right\}_{j\geq j_P}\right\|_{L\dot A_{p,q}}
&=\left\|\mathbf{1}_P\right\|_{L^p}
\left\|\left\{2^{-jn(\tau-\frac{1}{p})}\right\}_{j\geq j_P}\right\|_{\ell^q}\\
&\sim|P|^{\frac{1}{p}}2^{-j_Pn(\tau-\frac{1}{p})}
=|P|^{\frac{1}{p}}|P|^{\tau-\frac{1}{p}}
=|P|^{\tau}\notag
\end{align}
and the assumption that $q\in(0,\infty)$ and $\tau\in(\frac{1}{p},\infty)$
or $q=\infty$ and $\tau\in[\frac{1}{p},\infty)$
was used in estimating the $\ell^q$ norm.
Thus, by both \eqref{5.35x} and \eqref{5.35y}, we obtain
\begin{equation*}
\left\|\vec{t}\right\|_{\dot a^{s,\tau}_{p,q}(\mathbb A)}
=\sup_{P\in\mathscr Q}|P|^{-\tau}
\left\|\left\{2^{js}\left|\mathbf{1}_PA_j\vec{t}_j\right|
\right\}_{j\geq j_P}\right\|_{L\dot A_{p,q}}
\lesssim\left\|\vec{t}\right\|_{\dot f^{s+n(\tau-\frac{1}{p})}_{\infty,\infty}(\mathbb A)},
\end{equation*}
which then completes the proof of Theorem \ref{92}.
\end{proof}

We can now identify a range of Triebel--Lizorkin-type sequence spaces with ``plain'' Triebel--Lizorkin space for $p=\infty$.

\begin{corollary}\label{92 cor}
Let $p\in(0,\infty)$, $q\in(0,\infty]$, $s\in\mathbb R$, $W\in A_p$,
and $\mathbb A=\{A_Q\}_{Q\in\mathscr Q}$ be
a sequence of reducing operators of order $p$ for $W$.
Then we have the following identifications of spaces with equivalent quasi-norms:
\begin{enumerate}[\rm(i)]
\item for Triebel--Lizorkin-type spaces with $\tau=\frac1p$,
$\dot f^{s,\frac{1}{p}}_{p,q}(W)=\dot f^{s}_{\infty,q}(\mathbb{A});$
\item whenever $\tau>\frac{1}{p}$ or $(\tau,q)=(\frac1p,\infty)$,
$\dot a^{s,\tau}_{p,q}(W)
=\dot{f}_{\infty,\infty}^{s+n(\tau-\frac{1}{p})}(\mathbb{A}).$
\end{enumerate}
\end{corollary}

\begin{proof}
By Theorem \ref{37} we find that $\dot a^{s,\tau}_{p,q}(W)=\dot a^{s,\tau}_{p,q}(\mathbb A)$ in both cases under consideration. The critical case then follows from Theorem \ref{97}, which contains the statement that $\dot f^{s,\frac{1}{p}}_{p,q}(\mathbb A)=\dot f^{s}_{\infty,q}(\mathbb{A})$, and the supercritical case from Theorem \ref{92}.
This finishes the proof of Corollary \ref{92 cor}.
\end{proof}

Finally, we obtain the following corresponding result for function spaces.

\begin{corollary}\label{921}
Let $p\in(0,\infty)$, $q\in(0,\infty]$, $s\in\mathbb R$, $W\in A_p$,
and $\mathbb A=\{A_Q\}_{Q\in\mathscr Q}$ be
a sequence of reducing operators of order $p$ for $W$.
Then we have the following identifications of spaces with equivalent quasi-norms:
\begin{enumerate}[\rm(i)]
\item for Triebel--Lizorkin-type spaces with $\tau=\frac1p$,
$\dot F^{s,\frac{1}{p}}_{p,q}(W)=\dot F^{s}_{\infty,q}(\mathbb{A});$
\item whenever $\tau>\frac{1}{p}$ or $(\tau,q)=(\frac1p,\infty)$,
$\dot A^{s,\tau}_{p,q}(W)
=\dot{F}_{\infty,\infty}^{s+n(\tau-\frac{1}{p})}(\mathbb{A}).$
\end{enumerate}
\end{corollary}

\begin{proof}
Let $\varphi,\psi\in\mathcal{S}$ satisfy \eqref{19}, \eqref{20} and \eqref{21}. By Theorem \ref{phi} and Proposition \ref{38}, we conclude that
\begin{equation*}
S_\varphi:\ \dot A^{s,\tau}_{p,q}(W)\to\dot a^{s,\tau}_{p,q}(W)
\text{ and }
T_\psi:\ \dot a^{s,\tau}_{p,q}(W)\to\dot A^{s,\tau}_{p,q}(W)
\end{equation*}
are bounded and $T_\psi\circ S_\varphi$ is the identity on $\dot A^{s,\tau}_{p,q}(W)$. Similarly, from Theorem \ref{phi infty} and Proposition \ref{independent 2}, it follows that
\begin{equation*}
S_\varphi:\ \dot F_{\infty,q}^s(\mathbb{A})\to\dot f_{\infty,q}^s(\mathbb{A})
\text{ and }
T_\psi:\ \dot f_{\infty,q}^s(\mathbb{A})\to\dot F_{\infty,q}^s(\mathbb{A})
\end{equation*}
are bounded and $T_\psi\circ S_\varphi$ is the identity on $\dot F_{\infty,q}^s(\mathbb{A})$. By these results and Corollary \ref{92 cor}, we find that
\begin{equation*}
T_\psi\circ S_\varphi: \dot X\overset{S_\varphi}{\longrightarrow}  \dot x
= \dot y\overset{T_\psi}{\longrightarrow} \dot Y
\end{equation*}
is bounded whenever
\begin{equation*}
\{(\dot X,\dot x),(\dot Y,\dot y)\}
\subset\left\{\left(\dot A^{s,\tau}_{p,q}(W),\dot a^{s,\tau}_{p,q}(W)\right),
\left(\dot{F}_{\infty,\infty}^{s+n(\tau-\frac{1}{p})}(\mathbb{A}),\dot{f}_{\infty,\infty}^{s+n(\tau-\frac{1}{p})}(\mathbb{A})\right)\right\},
\end{equation*}
where $\tau>\frac1p$ or $(\tau,q)=(\frac1p,\infty)$, or
\begin{equation*}
\{(\dot X,\dot x),(\dot Y,\dot y)\}
\subset\left\{\left(\dot F_{p,q}^{s,\frac1p}(W),\dot f_{p,q}^{s,\frac1p}(W)\right),
\left(\dot F_{\infty,q}^s(\mathbb{A}),\dot f_{\infty,q}^s(\mathbb{A})\right)\right\}.
\end{equation*}
On the other hand, $T_\psi\circ S_\varphi$ is the identity on each such $\dot X$. It follows that the identity is bounded from $\dot X$ to $\dot Y$ for each pair $(\dot X,\dot Y)$ as above. Since the roles of $\dot X$ and $\dot Y$ are exchangeable, it follows that each $\dot X\subset\dot Y\subset\dot X$, and hence $\dot X=\dot Y$.
This finishes the proof of Corollary \ref{921}.
\end{proof}

As is evident from the last two corollaries, the value $\tau=\frac1p$ represents a qualitative turning point in the nature of the spaces $\dot A^{s,\tau}_{p,q}(W)$ and $\dot a^{s,\tau}_{p,q}(W)$. Accordingly, we introduce the following terminology that also plays a role
in the analysis of operators acting on these spaces, undertaken in the subsequent articles \cite{bhyy2,bhyy3}.

\begin{definition}\label{critical}
We say that a function or sequence space of Besov-type or Triebel--Lizorkin-type, with parameters $(p,q,s,\tau)$, is
\begin{enumerate}[\rm(i)]
\item \emph{supercritical} if $\tau>\frac{1}{p}$ or $(\tau,q)=(\frac{1}{p},\infty)$,
\item \emph{critical} if $\tau=\frac1p$ and $q<\infty$ and the space is of Triebel--Lizorkin-type,
\item \emph{subcritical} if $\tau<\frac1p$, or if $\tau=\frac1p$ and $q<\infty$ and the space is of Besov-type.
\end{enumerate}
\end{definition}

Thus, while all spaces with $\tau>\frac1p$ (resp. $\tau<\frac1p$) are supercritical (resp. subcritical), spaces with $\tau=\frac1p$ may be of any of the three types, depending on the finer details of the other parameters. This particular classification is motivated by the previous two Corollaries \ref{92 cor} and \ref{921}, where the two cases deal with critical and supercritical spaces in the sense of Definition \ref{critical}. This classification will also play a role in \cite{bhyy2}.

\begin{remark}
\begin{enumerate}[\rm(i)]
\item Except for the Triebel--Lizorkin spaces with $p=\infty$
(which include the Besov spaces $\dot b^s_{\infty,\infty}=\dot f^s_{\infty,\infty}$),
all other usual Besov spaces $\dot b^s_{p,q}=\dot b^{s,0}_{p,q}$
and Triebel--Lizorkin spaces $\dot f^s_{p,q}=\dot f^{s,0}_{p,q}$ are subcritical:
for $p\in(0,\infty)$, these have $\tau=0<\frac1p$, while the Besov spaces with
$\tau=0=\frac{1}{\infty}=\frac{1}{p}$ and $q<\infty$ are subcritical by definition.
\item The concept of the criticality is consistent with the identities of spaces
established in Corollaries \ref{92 cor} and \ref{921}. That is, the spaces
$\dot f^{s,\frac 1p}_{p,q}(W)=\dot f^{s,0}_{\infty,q}(\mathbb A),$
where $q<\infty,$
have $\tau_{\textup{left}}=\frac1p$, but also
$\tau_{\textup{right}}=0=\frac{1}{\infty}=\frac{1}{p_{\textup{right}}}$,
while $q<\infty$ is the same on both sides, and both spaces are of Triebel--Lizorkin-type;
hence one consistently classifies these spaces as critical,
whether one looks at the left-hand or the right-hand side of the equality.
Similarly, the spaces
$\dot a^{s,\tau}_{p,q}(W)=\dot f^{s+n(\tau-\frac1p),0}_{\infty,\infty}(\mathbb A),$
where $\tau>\frac1p$ or $(\tau,q)=(\frac1p,\infty),$
have also $(\tau_{\textup{right}},q_{\textup{right}})=(0,\infty)$,
where $0=\frac{1}{\infty}=\frac{1}{p_{\textup{right}}}$;
thus one consistently classifies these spaces as supercritical.
\end{enumerate}
\end{remark}

\noindent\textbf{Acknowledgements}.
The first author would like to thank Yiqun Chen for proposing Proposition \ref{equivalent}.

\bigskip

\noindent Fan Bu

\medskip

\noindent Laboratory of Mathematics and Complex Systems (Ministry of Education of China),
School of Mathematical Sciences, Beijing Normal University, Beijing 100875, The People's Republic of China

\smallskip

\noindent{\it E-mail:} \texttt{fanbu@mail.bnu.edu.cn}

\bigskip

\noindent Tuomas Hyt\"onen

\medskip

\noindent Department of Mathematics and Statistics,
University of Helsinki, (Pietari Kalmin katu 5), P.O.
Box 68, 00014 Helsinki, Finland

\smallskip

\noindent{\it E-mail:} \texttt{tuomas.hytonen@helsinki.fi}

\smallskip

\noindent (Address as of 1 Jan 2024:) Department of Mathematics and Systems Analysis, Aalto University, P.O. Box 11100, FI-00076 Aalto, Finland

\noindent{\it E-mail:} \texttt{tuomas.p.hytonen@aalto.fi}

\bigskip

\noindent Dachun Yang and Wen Yuan (Corresponding author)

\medskip

\noindent Laboratory of Mathematics and Complex Systems (Ministry of Education of China),
School of Mathematical Sciences, Beijing Normal University, Beijing 100875, The People's Republic of China

\smallskip

\noindent{\it E-mails:} \texttt{dcyang@bnu.edu.cn} (D. Yang)

\noindent\phantom{{\it E-mails:} }\texttt{wenyuan@bnu.edu.cn} (W. Yuan)

\end{document}